\documentclass[12pt]{amsart}
\usepackage{amsmath}
\usepackage{algorithm2e}
\usepackage[psamsfonts]{amssymb}
\usepackage[all]{xy}
\usepackage{tikz}
\usepackage{tikz-cd}
\usepackage{dsfont}
\usetikzlibrary{%
  matrix,%
  calc,%
  arrows%
}
\usepackage{graphicx}
\usepackage[T1]{fontenc}
\usepackage[bitstream-charter]{mathdesign}
\usepackage{hyperref}
\newtheorem{theorem}{Theorem}[section]
\newtheorem{prop}[theorem]{Proposition}
\newtheorem{lemma}[theorem]{Lemma}
\newtheorem{cor}[theorem]{Corollary}

\newtheorem{conj}[theorem]{Conjecture}

\newtheorem{ex}[theorem]{Example}
\theoremstyle{remark}
\newtheorem{dfn}[theorem]{Definition}
\newtheorem{remark}[theorem]{Remark}

\newtheorem{notation}[theorem]{Notation}

\newtheorem{const}[theorem]{Construction}

\def\co{\colon\thinspace}

\def\ep{\epsilon}
\def\Hom{\rm{Hom}}

\def\R{\mathbb{R}}
\def\Z{\mathbb{Z}}
\def\N{\mathbb{N}}
\def\D{\mathcal{D}}
\def\C{\mathbb{C}}
\def\K{\mathcal{K}}

\def\I{\mathbb{I}}
\def\mod{\rm{mod}}
\def\sign{\rm{sign}}

\def\Span{\rm{span}}

\begin{document}

\title[p-cyclic persistent homology and Hofer distance]{p-cyclic persistent homology and\\ Hofer distance}

\author{Jun Zhang}
\email{junzhang@mail.tau.ac.il}
\address{School of Mathematical Sciences\\Tel Aviv University\\Ramat Aviv, Tel Aviv 69978, Israel}

\begin{abstract}
In this paper, we prove that the Hofer distance from time-dependent Hamiltonian diffeomorphisms to the set of $p$-th power Hamiltonian diffeomorphisms can be arbitrarily large for symplectic manifold $\Sigma_g \times M$, where $M$ is {\it any} closed symplectic manifold, $p$ is sufficiently large and $g \geq 4$. This implies that, on this product, the Hofer distance can be arbitrarily large between time-dependent Hamiltonian diffeomorphisms and autonomous Hamiltonian diffeomorphisms. This generalizes the main result from L. Polterovich and E. Shelukhin's paper \cite{PS14}. The basic tools we will use are barcode and singular value decomposition which are developed in the paper \cite{UZ15}, from which we borrow many proofs and modify them so that they can be adapted to a Floer-type complex equipped with a group action. 
\end{abstract}

\maketitle
  
\tableofcontents

\section{Introduction} 

The Hofer distance between the set of autonomous Hamiltonian diffeomorphisms and a time-dependent Hamiltonian diffeomorphism has been recently studied in \cite{PS14}. We first give several definitions. 

Recall that given any closed symplectic manifold $(X, \omega)$, a (smooth) function $H: [0,1] \times X \to \R$ will generate a Hamiltonian diffeomorphism $\phi = \phi^1_H$ which is the time-1 map of flow $\{\phi_H^t\}$. Denote by $Ham(X, \omega)$ the set of all the Hamiltonian diffeomorphisms. Not only can we prove $Ham(X, \omega)$ is a group, but also we can associate a bi-invariant metric on this group which is the well-known Hofer's metric, denoted by  $d_H$. Hofer's metric is defined by the following two steps. First, for any $\phi \in Ham(X, \omega)$, define (Hofer's norm)
\[ ||\phi||_H = \inf\left\{ \int_0^1 \left( \max_{X} H(t, \cdot) - \min_{X} H(t, \cdot)\right)dt \, \bigg| \, \phi = \phi_H^1 \right\}. \]
The integral above is sometimes denoted by $||H||_H$ for a (Hamiltonian) function $H$. Then define Hofer's metric as follows. For $\phi, \psi \in Ham(X, \omega)$, 
\begin{equation} \label{hm}
 d_H(\phi, \psi) = ||\phi^{-1} \circ \psi ||_H.
 \end{equation}

\begin{dfn} For a closed symplectic manifold $(X, \omega)$, define 
\begin{align*} 
{\rm Aut}(X) = \{& \phi \in Ham(X , \omega) \,| \\
&\phi = \phi^1_H \,\, \mbox{where $H(t,x)$ is independent of $t$}\}.
\end{align*}
\end{dfn}
\begin{dfn} For a symplectic manifold $X$, define 
\[ {\rm aut}(X) = \sup_{\phi \in Ham(X, \omega)} d_H(\phi, {\rm Aut}(X)), \]
where $d_H$ is Hofer's metric defined by (\ref{hm}). \end{dfn} 
Compared with time-dependent Hamiltonian diffeomorphisms, a special feature of time independent (or usually called autonomous) Hamiltonian diffeomorphisms comes from the following observation. If a time-dependent function $F(t,x)$, as a Hamiltonian, generates $\psi$, then 
\begin{equation} \label{k-generating}
H(t,x) = p F(pt, x) \,\,\,\mbox{generates} \,\,\, \phi = \psi^p.
\end{equation}
In particular, if $\phi$ belongs to ${\rm Aut}(X)$, then for {\it any} prime $p$, we can simply take generating function $F(x) = \frac{1}{p} H(x)$ which generates a $p$-th root of $\phi$. This motivates another more delicate definition as follows. 
\begin{dfn} \label{dfn-k} Let $p \geq 2$ be a prime number. For a symplectic manifold $X$, define 
\[ Power_p (X) = \{ \phi = \psi^p\,| \, \psi \in Ham(X, \omega) \}. \]
\end{dfn}
\begin{dfn} \label{dfn-hofer-k} Let $p \geq 2$ be a prime number. For a symplectic manifold $X$, define 
\[ power_p(X) = \sup_{\phi \in Ham(X, \omega)} d_H(\phi, Power_p(X)) \]
where $d_H$ is Hofer's metric defined by (\ref{hm}).\footnote{The original definition in \cite{PS14} is defined for any integer $k \geq 2$. But since $d_H(\phi, Power_p(X)) \leq d_H(\phi, Power_k(X))$ when $p \,| \,k$, we will only consider prime number $p$ here.} \end{dfn} 
As we have noticed earlier, ${\rm Aut}(X) \subset \bigcap_{p \, \, \mbox{\tiny{is prime}}} Power_p(X)$. With the notations above, we can state the following main theorem in \cite{PS14}, 
\begin{theorem} \label{main-thm} [Theorem 1.3 in \cite{PS14}]\,\, Let $\Sigma_g$ be a fixed closed oriented surface with genus $g \geq 4$. Then for any symplectically aspherical closed manifold $M$ and for any $p \geq 2$, we have 
\[ power_{p} (\Sigma_g \times M) = +\infty. \]
\end{theorem}
Note that, under the same hypothesis of Theorem \ref{main-thm} and fixing any prime number $p$, Theorem \ref{main-thm} immediately implies that 
\[ aut(\Sigma_g \times M) = +\infty. \]
This is actually another theorem stated in \cite{PS14} (Theorem 1.2), where its original proof comes from a different (easier) argument than the original proof of Theorem \ref{main-thm}. Last but not least, we emphasize that the results mentioned above are within the effort to prove or understand the following general conjecture.

\begin{conj} \label{conj} For any closed symplectic manifold $X$, ${\rm aut}(X) = +\infty$. \end{conj}

Now back to Theorem \ref{main-thm}, from the title of \cite{PS14}, its proof involves persistence module theory or persistence homology (see \cite{ZC05} for a ``basic version'' of persistence module theory. However, the persistence module theory from \cite{ZC05} can only be used in Hamiltonian Floer theory when the symplectic manifold is monotone). Denote by $\mathcal L_{\alpha} X$ the component of loop space, consisting of all loops representing a homotopy class $\alpha$. For any given prime $p$, define an operator $R_p$ by a rotation, that is,  
\[ R_p (x(t)) = x \left( t + \frac{1}{p} \right). \]
Given a time-dependent Hamiltonian function $H(t,x)$ generating $\phi$, denote $H^{(p)}(t,x) = p H(pt,x)$ generating $\phi^p$ by (\ref{k-generating}). Notice that if $x(t)$ is a Hamiltonian $1$-periodic orbit of Hamiltonian $H^{(p)}$, then $R_p(x(t))$ is also a Hamiltonian $1$-periodic orbit of $H^{(p)}$ since $p(t+ 1/p) = pt + 1= pt$ in $\R/\Z$. This induces a filtered chain isomorphism between two Floer chain complexes (filtered by action functionals) for each degree $k \in \Z$, 
\begin{equation}\label{rt}
R_p: CF_k(H^{(p)}, J_t)_{\alpha} \,\rightarrow CF_k(H^{(p)}, (R_p)_* J_t)_{\alpha} = CF_k(H^{(p)}, J_{t+ \frac{1}{p}})_{\alpha}.
\end{equation}
After passing to the homology, taking advantage of the fact that Floer homology is independent of almost complex structures, we get a pair 
\begin{equation}\label{pm}
 (\mathbb H_k(\phi), T) 
\end{equation}
where
\begin{itemize}
\item{} $\mathbb H_k(\phi) = (\{ HF_k^{(-\infty, s)} (\phi^p)_{\alpha}\}_{s \in \R}; \pi_{s,t})$ is a persistence module (see Definition 3.1 in \cite{ZC05}), where transition function $\pi_{s,t}$ is induced by inclusion for any $s \leq t$;
\item{} $T = [(R_p)_*]$ is a filtered isomorphism (or 0-interleaving) giving a $\Z_p$ action on $\mathbb H_k(\phi)$, that is, $T^p ={\mathds 1}$.
\end{itemize}  

In order to successfully link Floer theory and persistence module theory, a numerical measurement $\mu_p(\phi)$ is defined  (on the top of page 40 in \cite{PS14}) by using combinatorics data (barcode) from the associated persistence module (\ref{pm}). It satisfies Lipschitz continuity with respect to Hofer's metric. Therefore, by {\it the} most important theorem in this theory --- {\it isometry theorem}, this proposition translates the combinatorial information from barcodes to the numerical information from (Floer) chain complexes which is captured partially by Hofer metric. This turns out to be the key step in the proof of Theorem \ref{main-thm}.

To generalize this process, we will make efforts in two directions. First, we will use more sophisticated and powerful persistence module theory that is developed in the paper \cite{UZ15} to rewrite the set-up of this problem in the (Floer) chain complex level. Second, following the idea above, we will also define some numerical measurement (in fact proved to be a symplectic invariant) which satisfies Lipschitz continuity with respect to Hofer's metric. We emphasize that the product structure as in the result of Theorem \ref{main-thm} also plays an important role. In other words, we are not able to solve Conjecture \ref{conj} completely (cf. Theorem \ref{aut-main-thm}).

\begin{remark} \label{irr-cond} Through out the paper \cite{PS14}, the working field $\mathcal K$, should satisfy the following important restriction. As we will also assume this condition in this paper, we state it separately here. {\bf Irreducible condition:} 
\begin{itemize}
\item{} $char(\mathcal K) =0$ and $\mathcal K$ contains all $p$-th roots of unity;
\item{} For any primitive $p$-th root of unity $\xi_p$, there is no solution of the following equation $x^p = \xi_p^q$ unless $p \,| \, q$.
\end{itemize}

Note that this condition gives a strong restriction on the dimension of invariant subspace. Explicitly, if $V$ is a $T$-invariant subspace such that $T^p = \xi_p \cdot {\mathds 1}$ for some $p$-th root of unity $\xi_p$, then $p\,| \,\dim(V)$. (cf. Lemma 4.15 in \cite{PS14}).\end{remark}

Here is our main theorem in this paper. 

\begin{theorem} \label{main-thm-j} Let $\Sigma_g$ be a fixed closed oriented surface with genus $g \geq 4$. For any closed symplectic manifold $M$ and for any prime $p \!>\! 2\!\sum_{0 \leq i \leq 2n}\! b_i(M)$,
\[ power_{p} (\Sigma_g \times M) = +\infty. \]
\end{theorem}
This immediately implies the following result. 
\begin{theorem} \label{aut-main-thm} Let $\Sigma_g$ be a fixed closed oriented surface with genus $g \geq 4$. For any closed symplectic manifold $M$, 
\[ aut(\Sigma_g \times M)= + \infty. \]
\end{theorem}
 
\subsection{Outline of the paper and summary of results.} 

Before starting to outline the proof of Theorem \ref{main-thm-j}, we will recall some definitions from \cite{UZ15} which describes a general framework that Floer theory can fit into. It will be helpful to state our main propositions. First, for each $k \in \Z$, $CF_k(M, H)$ is a finite dimensional vector space over Novikov field $\Lambda^{\mathcal K, \Gamma}$ equipped with a filtration function denoted by  $\ell$ defined by action functional. In general, for such filtered vector space denoted by  $(V, \ell)$, we can find an ordered basis $(x_1,\dots, x_m)$ such that it satisfies {\it orthogonality property} (see Definition 2.7 in \cite{UZ15}), that is, for any $\lambda_1, \dots, \lambda_n \in \Lambda^{\K, \Gamma}$, 
\[ \ell\left(\sum_{i=1}^n \lambda_i x_i\right) = \max_{1 \leq i \leq n} \ell(\lambda_i x_i).  \]
We call such basis an {\it orthogonal basis of $(V, \ell)$}, then $(V, \ell)$ is called an {\it orthogonalizable $\Lambda^{\K, \Gamma}$-space}. If, in particular, $\ell(x_i)=0$ for all $i$, then this orthogonal basis is called an {\it orthonormal basis}. Theorem 3.4 in \cite{UZ15} guarantees that, for any linear map $T: (V, \ell_1) \to (W, \ell_2)$, there exists a nice choice of orthogonal bases for $(V, \ell_1)$ and $(W, \ell_2)$ which are compatible with $T$ in the following sense.

\begin{dfn} \label{dfn-svd} For a nonzero $\Lambda^{\mathcal K, \Gamma}$-linear map $T: (V, \ell_1) \to (W, \ell_2)$ with rank $r$, we say that (ordered) orthogonal bases $(y_1, \dots, y_n)$ for $V$ and $(x_1, \dots, x_m)$ for $W$ form a {\it singular value decomposition of $T$} if 
\begin{itemize}
\item{} $\ker(T) = {\Span}_{\Lambda^{\mathcal K, \Gamma}} \left< y_{r+1}, \dots, y_n\right>$;
\item{} ${\rm{Im}}(T) = {\Span}_{\Lambda^{\mathcal K, \Gamma}} \left< x_{1}, \dots , x_r\right>$;
\item{} $T(y_i) = x_i$ for $i \in \{1, \dots , r\}$;
\item{} $\ell_1(y_1) - \ell_2(x_1) \geq \cdots  \geq \ell_1(y_r) - \ell_2(x_r)$.
\end{itemize}
\end{dfn}
Given a Floer-type complex (see Definition 4.1 in \cite{UZ15}) $(C_*, \partial_*, \ell)$ and any degree $k$, restricting to 
\[ C_{k+1} \xrightarrow{\partial_{k+1}} \ker(\partial_k),\]
we will get a singular value decomposition of $\partial_{k+1}$ denoted by  $((y_1, \dots , y_n)$, $(x_1, \dots , x_m))$. Then we define our barcode (see Definition 6.3 in \cite{UZ15}) as follows. 

\begin{dfn} \label{dfn-barcode} {\it Degree-$k$ verbose barcode of $(C_*, \partial_*, \ell)$} consists of multiset of elements of $(\R/\Gamma) \times [0, \infty]$ in the form
\begin{itemize}
\item[(1)] $(\ell(x_i) \,\mbox{mod} \, \Gamma, \ell(y_i) - \ell(x_i))$ for $i \in \{1, \dots , r\}$;
\item[(2)] $(\ell(x_i) \,\mbox{mod} \, \Gamma, \infty)$ for $i \in \{r+1, \dots , m\}$,
\end{itemize}
where $((y_1, \dots , y_n), (x_1, \dots , x_m))$ is a singular value decomposition of $\partial_{k+1}$. Moreover, define {\it degree-$k$ concise barcode} to be the submultiset of the verbose barcode consisting of those whose second element is positive. \end{dfn}

\begin{remark} Due to various purposes in this paper, there are a couple of places where we need to explicitly construct the orthogonal bases in Definition \ref{dfn-svd} and the resulting barcode in Definition \ref{dfn-barcode} (see subsection \ref{sub-p-c-svd} and subsection \ref{ssec-exe-computing-barcode}). To this end, we will closely follow Theorem 3.5 in \cite{UZ15}, which provides an algorithm for these computations. We emphasize that, to our best knowledge, this barcode invariant of a Floer-type chain complex was for the first time introduced by S. Barannikov in \cite{Bar94} in a similar set-up, and such an algorithm that produces a preferred basis was first worked out by Lemma 2 in \cite{Bar94}. \end{remark}

One quantitative measurement between $(C_*, \partial_*, \ell_C)$ and $(D_*, \partial_*, \ell_D)$ arising from the standard analysis in Floer theory is {\it quasiequivalence distance} (see Definition 1.3 in \cite{UZ15}). Here we give a slightly more general definition. 

\begin{dfn} \label{dfn-mc} Let $(C_*,\partial_C,\ell_C)$ and $(D_*,\partial_D,\ell_D)$ be two Floer-type complexes, and two fixed non-negative numbers $\delta_+$ and $\delta_-$.  A $(\delta_+, \delta_-)$-\it{quasi\-equivalence} between $C_*$ and $D_*$ is a quadruple $(\Phi,\Psi,K_C,K_D)$ where:
\begin{itemize} \item $\Phi\co C_*\to D_*$ and $\Psi\co D_*\to C_*$ are chain maps, with $\ell_D(\Phi c)\leq \ell_C(c)+\delta_+$ and $\ell_C(\Psi d)\leq \ell_D(d)+\delta_-$ for all $c\in C_*$ and $d\in D_*$.
\item $K_C\co C_*\to C_{*+1}$ and $K_D\co D_*\to D_{*+1}$ obey the homotopy equations $\Psi\circ\Phi-{\mathds 1}_{C_*}=\partial_CK_C+K_C\partial_C$ and $\Phi\circ\Psi-{\mathds 1}_{D_*}=\partial_DK_D+K_D\partial_D$, and for all $c\in C_*$ and $d\in D_*$ we have $\ell_C(K_Cc)\leq \ell_C(c)+\delta_+ + \delta_-$ and $\ell_D(K_Dd)\leq \ell_D(d)+\delta_+ + \delta_-$.
\end{itemize}
The \it{quasiequivalence distance} between   $(C_*,\partial_C,\ell_C)$ and $(D_*,\partial_D,\ell_D)$ is then defined to be 
	\[d_Q(C_*,D_*)=\inf\left\{\frac{\delta_+ +\delta_-}{2}\geq 0\left|\!\resizebox{.56\textwidth}{!}{$\begin{array}{l}\mbox{There exists a $(\delta_+, \delta_-)$-quasiequivalence between }\\(C_*,\partial_C,\ell_C)\mbox{ and }(D_*,\partial_D,\ell_D).\end{array}$}\right.\!\!\!\right\}. \]
\end{dfn}

An algebraic structure which has been effectively used in \cite{UZ15} and will also be used  frequently in this paper is a filtered mapping cone (see Definition 9.9 in \cite{UZ15}). Here, we give its construction in the following special case. 
\begin{dfn} \label{self-mc} Given any Floer-type complex $(C_*, \partial_*, \ell)$ and a filtration preserving chain map $T$ on it, we call the following chain complex \it{the self-mapping cone with respect to $T$},
\[ Cone_C(T)_* := C_* \oplus C_{*-1} \,\,\,\,\,\mbox{with} \,\,\,\,\,\,\, \partial_{co} := \left( \begin{array}{cc}
\partial & -T \\
0 & - \partial  \\
\end{array} \right). \]
Moreover, define its filtration $\ell_{co}$ by, for any $(x_1, x_2) \in Cone_C(T)_*$, 
\[ \ell_{co}((x_1, x_2)) = \max\{\ell(x_1), \ell(x_2)\}.\]
\end{dfn}
Now we summarize the outline. 
\subsubsection{Construction of obstruction} First, we give the recipe in this story to construct a numerical measurement (proved to be symplectic invariant) by using both Floer theory and Floer-Novikov persistent homology theory developed in \cite{UZ15}. 
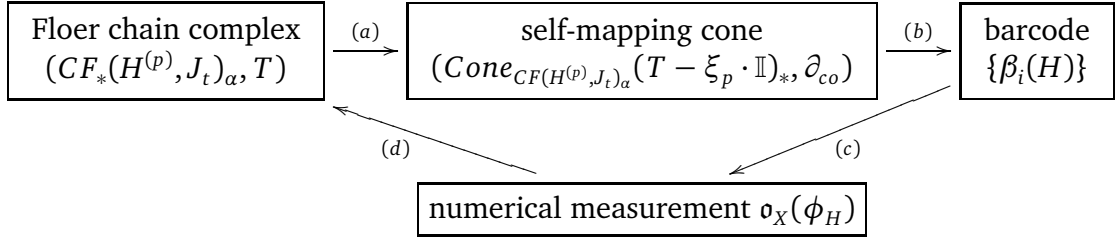
\begin{equation} \label{recipe} 
\resizebox{.95\textwidth}{!}{\xymatrix{
\boxed{\begin{array}{cc} \mbox{Floer chain complex} \\ (CF_*(H^{(p)}, J_t)_{\alpha}, T) \end{array} } \ar[r]^-{(a)} & \boxed{\begin{array}{cc}\mbox{self-mapping cone}\\(Cone_{CF(H^{(p)},J_t)_{\alpha}}(T - \xi_p \cdot {\mathds 1})_*, \partial_{co}) \end{array}} \ar[r]^-{(b)} & \boxed{\begin{array}{cc} \mbox{barcode}\\ \{\beta_i(H)\}\end{array}} \ar[ld]^{(c)} \\
	&\boxed{\mbox{numerical measurement $\mathfrak o_X(\phi_H)$}} \ar[lu]^{(d)}  }}\hspace{-1em}
\end{equation}

We will explain each of the boxes and arrows above with an emphasize on the special case when $\phi= \phi_H \in Power_p(X)$, since eventually we will use the numerical measurement constructed above to form an obstruction to the condition that $\phi \in Power_p(X)$.\\
 
\noindent {\bf Floer chain complex}. The rotation action (\ref{rt}) on a Floer chain complex will pushforward its almost complex structure, so once we are working on the Floer chain complex, the rotation action (\ref{rt}) does not behave well on the Floer homology in the sense that in order to work on the same chain complex, we need to use some continuation map $C$, 
\begin{equation} \label{comp1}
\xymatrix{
CF_k(H^{(p)}, J_t)_{\alpha} \ar[r]^{R_p} \ar@/^2.0pc/@[black][rr]^{T = C \circ R_p} & CF_{k}(H^{(p)}, J_{t + \frac{1}{p}})_{\alpha} \ar[r]^C &CF_k(H^{(p)}, J_t)_{\alpha}}.
\end{equation}
Note that in general, $T^p \neq {\mathds 1}$, which is the source of most of the difficulty when we are working on the chain complex level. Meanwhile, recall in the proof of Theorem 4.22 in \cite{PS14}, if $\phi \in Power_p(X)$, say $\phi = \psi^p$ for some $\psi \in Ham(X,\omega)$, then there exists a well-defined chain map for each degree $k \in \Z$, 
\[ R_{p^2}: CF_k(H^{(p)}, J_t)_{\alpha}  \to CF_{k}(H^{(p)}, J_{t + \frac{1}{p^2}})_{\alpha}\]
where $H(t,x) = pF(pt,x)$ and $F$ is a Hamiltonian generating $\psi$. Again, in order to work on a space itself, we need to use some continuation map $C'$ to form the following composition
\begin{equation} \label{comp2}
\xymatrix{
CF_k(H^{(p)}, J_t)_{\alpha} \ar[r]^{R_{p^2}} \ar@/^2.0pc/@[black][rr]^{S = C' \circ R_{p^2}} & CF_{k}(H^{(p)}, J_{t + \frac{1}{p^2}})_{\alpha} \ar[r]^{C'} &CF_k(H^{(p)}, J_t)_{\alpha}}.
\end{equation}
Our first observation is 
\begin{prop} \label{cont2} For any closed symplectic manifold $(X, \omega)$, if $\phi = \phi_H \in Power_p(X)$, then for any degree $k \in \Z$, there exists a continuation map 
$$C: CF_{k}(H^{(p)}, J_{t + \frac{1}{p}})_{\alpha} \to CF_k(H^{(p)}, J_t)_{\alpha}$$
and a continuation map 
$$C': CF_{k}(H^{(p)}, J_{t + \frac{1}{p^2}})_{\alpha} \to CF_k(H^{(p)}, J_t)_{\alpha}$$
such that,  $T = S^p$, where $T$ and $S$ are compositions in (\ref{comp1}) and (\ref{comp2}). 
\end{prop}
This result is proved in Section \ref{p14}. It is an analogous relation on the Floer chain complex, compared with the relation $[(R_p)_*] = [(R_{p^2})_*]^p$ on $HF_k(\phi^p)$.
\begin{notation} For the rest of the paper, whenever we use $T$, it always means the composition defined in (\ref{comp1}). If we need to emphasize the Hamiltonian $H$ of the corresponding system, we will denote it as $T^H$. Whenever we use $T_p$, it always means the resulting composition from Proposition \ref{cont2} that has $p$-th root. \end{notation}

\noindent {\bf{Self-mapping cone}}. Define the self-mapping cone of chain complex $CF_*(H^{(p)}, J_t)_{\alpha}$ with respect to map $T - \xi_p \cdot {\mathds 1}$ (see Definition \ref{self-mc}). We call it {\it self-mapping cone} of $(CF_*(H^{(p)}, J_t)_{\alpha}, \partial)$. The degree-$k$ piece is
\begin{equation} \label{cone-alg-set}
(Cone_{CF(H^{(p)}, J_t)_{\alpha}}(T - \xi_p \cdot {\mathds 1}))_k = CF_{k}(H^{(p)}, J_t)_{\alpha} \oplus CF_{k-1}(H^{(p)}, J_t)_{\alpha}
\end{equation}
and the boundary map $\partial_{co}$ is 
\begin{equation} \label{bm-mc}
\left( \begin{array}{cc}
\partial & -\mathcal (T - \xi_p \cdot {\mathds 1}) \\
0 & - \partial  \\
\end{array} \right)
\end{equation}
where $\partial$ is the Floer boundary operator of $CF_*(H^{(p)}, J_t)$. Moreover, 
\begin{dfn} \label{doublemap}
For any element $(x_1, x_2) \!\in\! (Cone_{CF(H^{(p)}, J_t)_{\alpha}}(T - \xi_p \cdot {\mathds 1}))_*$, if $CF_*(H^{(p)}, J_t)_{\alpha}$ is acted by some linear map $A$, then define its {\it double map} $\mathcal D_A$ by 
\[ \mathcal D_{A} (x_1, x_2) = (A x_1, Ax_2).\]
\end{dfn}
In particular, by (\ref{comp1}) and (\ref{comp2}), the self-mapping cone is acted by double maps 
\begin{equation} \label{DT}
 \mathcal D_{T} = \mathcal D_{R_p} + C_T, 
 \end{equation}
for some map $C_T$ which strictly lowers the filtration, and (if it exists) also 
\begin{equation} \label{DS}
\mathcal D_{S} = \mathcal D_{R_{p^2}} + C_S, 
\end{equation}
for some map $C_S$ who also strictly lowers the filtration. Moreover, by Proposition \ref{cont2}, $\mathcal D_{S}^p = \mathcal D_{T}$. The following proposition shows that our self-mapping cone is well-defined. 

\begin{prop} \label{p-he} For any closed symplectic manifold $(X, \omega)$, up to a filtered isomorphism, the construction of $((Cone_{CF(H^{(p)}, J_t)_{\alpha}}(T - \xi_p \cdot {\mathds 1}))_*, \partial_{co})$ is independent of the choice of the continuation map that is used to form map $T$. Moreover, $\mathcal D_{T}$ and $\mathcal D_{S}$ defined in (\ref{DT}) and (\ref{DS}) are chain maps on $((Cone_{CF(H^{(p)}, J_t)_{\alpha}}(T - \xi_p \cdot {\mathds 1}))_*, \partial_{co})$, \emph{i.e.}, they commute with boundary operator of mapping cone $\partial_{co}$. 
\end{prop}

\noindent {\bf{Barcode of self-mapping cone}}. By the definition of barcodes, we know that the combinatorial data can reveal algebraic structures of chain complexes. For the barcode of the self-mapping cone that was defined above, a natural question is whether this ``special'' action $\mathcal D_{S}$ (if it exists) will shape its barcode in some way. In fact, we have the following important theorem. 

\begin{theorem} \label{p-tuples} Fix a primitive\footnote{According to the definition in \cite{PS14}, a free homotopy class is primitive if it can not be represented by a multiply-covered loop.}  free homotopy class $\alpha$ represented by a non-contractible loop. If $\phi = \phi^1_{H} \in Power_p(X)$, then for any degree $k \in \Z$ and boundary map $(\partial_{co})_{k+1}:\! (Cone_{CF(H^{(p)}, J_t)_{\alpha}}(T \!-\! \xi_p \cdot {\mathds 1}))_{k+1} \!\to\! {\rm{Im}}(\partial_{co})_{k+1}$\footnote{In general, to compute (degree-$k$) barcode of $(\partial_{co})_{k+1}$, we need its codomain to be $\ker(\partial_{co})_{k}$. But in this paper, we only consider Hamiltonian Floer chain complex of homotopy class represented by a non-contractible loop, so the homology of mapping cone vanishes. Therefore, $\ker(\partial_{co})_{k} = {\rm{Im}}(\partial_{co})_{k+1}$.}, each bar in the degree-$k$ concise barcode of $(\partial_{co})_{k+1}$ has its multiplicity divisible by $p$. \end{theorem}

This proposition is an analogue (but stronger) result with Proposition 4.18 in \cite{PS14}. The proof of this theorem is the most time-consuming part of this paper. The Section~\ref{p17} is devoted to its proof. This proposition should be plausible since $\mathcal D_{S}$ is a strictly lower filtration perturbation of a group action $\mathcal D_{R_{p^2}}$ with order $p^2$, which makes each degree-$k$ piece 
$$(Cone_{CF(H^{(p)}, J_t)_{\alpha}}(T - \xi_p \cdot {\mathds 1}))_k$$
a representation. This restricts the singular value decomposition (see Definition \ref{dfn-svd}) in a certain special form. \\

\noindent {\bf Numerical measurement}. Based on Theorem \ref{p-tuples}, we can define some numerical measurement on $Ham(X, \omega)$ from this combinatorial data. 
\begin{dfn} \label{dfn-o} Fix a prime number $p$. First take the collection of lengths of bars in degree-$k$ concise barcode of self-mapping cone constructed with respect to $\phi_H$, denoted by  $\{\beta_i(\phi_H)\}_i$ and order them as follows,  
\[ \beta_1(\phi_H) \geq \beta_2(\phi_H)\geq \cdots \geq \beta_{m_k}(\phi_H) >0\]
(so $m_k=$ multiplicity of degree-$k$ concise barcode). Then the {\it degree-$k$ divisibility sensitive invariant} of $\phi_H$ is defined as follows, 
\[ \mathfrak{o}_X(\phi_H)_k = \max_{s \in \N} \left(\beta_{sp+1}(\phi_H) - \beta_{(s+1)p}(\phi_H)\right) \]
and if $l > m_k$, set $\beta_{l}(\phi_H) =0$. In general, the {\it divisibility sensitive invariant of $\phi = \phi_H$} is defined as follows, 
\[ \mathfrak{o}_X(\phi) = \max_{\tiny{\mbox{primitive $\xi_p$}}} \sup_{k \in \Z} \mathfrak{o}_X(\phi_H)_k. \] \end{dfn}

\begin{remark}
Here we want to emphasize that $\mathfrak{o}_X(\cdot)$ is well-defined on $Ham(X, \omega)$, i.e., it does not depend on the choice of Hamiltonian functions generating a given Hamiltonian diffeomorphism, while the Floer chain complex and its barcode (with full information including endpoints) do depend on the Hamiltonian functions. In fact, by Proposition 5.3 in \cite{Ush13}, different choices of Hamiltonian functions result in a {\it shift-isomorphism} (see Definition 3.4 in \cite{Ush13}) between Hamiltonian Floer chain complexes (so also between the corresponding self-mapping cones). For its effect reflected on the bars, shift-isomorphism changes the left-endpoints and possibly also degrees, but importantly keeps the length of each bar the same. A sample demonstration is Corollary 5.4 (or Proposition 3.6) in \cite{Ush13} saying the {\it boundary depth} (the length of the longest finite length bar) is well-defined over $Ham(X, \omega)$. Finally, since $\mathfrak{o}_X(\cdot)$ is defined by taking supremum over all the degrees, the shift of degrees from shift-isomorphism can be ignored. \end{remark}

Similarly to the multiplicity sensitive spread $\mu_p(\phi_H)$ in \cite{PS14}, $\mathfrak{o}_X(\phi)$ is used to provide an obstruction to the condition $\phi \in Power_p(X)$. Explicitly, we have the following proposition. 

\begin{prop} \label{non-dis} If $\phi \in Power_p$(X), then $\mathfrak{o}_X(\phi) = 0$. If $p \nmid m_k$ for some degree $k$, then $\mathfrak{o}_X(\phi) \geq \beta_{m_k}(\phi_H)$. \end{prop} 
This proposition is proved in Section \ref{p19}. \\

Now we move to the arrows in the diagram (\ref{recipe}). All $(a)$, $(b)$ and $(c)$ are in the flavor of Lipschitz continuity. From now on, we will simply denote 
\[ Cone(H)_* := ((Cone_{CF(H^{(p)}, J_t)_{\alpha}}(T^H - \xi_p \cdot {\mathds 1}))_*, \partial_{{co}, H}),\]
and 
\[ Cone(G)_* := ((Cone_{CF(G^{(p)}, J_t)_{\alpha}}(T^G - \xi_p \cdot {\mathds 1}))_*, \partial_{{co}, G}),\]
the mapping cones constructed from different Hamiltonian functions $H$ and $G$. First, $(a)$ is corresponding to the following proposition, 
\begin{prop} \label{lip1} For any two Hamiltonians $H$ and $G$, we have 
\[ d_Q(Cone(H)_*, Cone(G)_*) \leq 3p \cdot||H-G||_H. \]
\end{prop}
Moreover, $(b)$ corresponds to the following proposition which will be a direct application of Corollary 8.8 in \cite{UZ15}.
\begin{prop} \label{lip2} 
Denote by $\beta_i(\phi_H)$ the length of the $i$-th bar in degree-$k$ verbose barcode of $Cone(H)_*$ and by $\beta_i(\phi_G)$ the length of the $i$-th bar in degree-$k$ verbose barcode of $Cone(G)_*$. Then we have 
\[ | \beta_i(\phi_H) - \beta_i(\phi_G)| \leq 4\,d_Q(Cone(H)_*, Cone(G)_*) \]
for every $i \in \Z$. \end{prop}
Note that $(a)$ and $(b)$ together imply the following proposition which $(c)$ corresponds to. 
\begin{prop} \label{lip3}
For any closed symplectic manifold $(X, \omega)$ and $\phi, \psi \in Ham(X,\omega)$, we have 
\[ |\mathfrak{o}_X(\phi) - \mathfrak{o}_X(\psi)| \leq 24p \cdot d_{H}(\phi, \psi).\]
\end{prop}
All these Lipschitz style propositions are proved in Section \ref{Lc}. 

\begin{remark} Indeed, $\mathfrak{o}_X(\phi)$ defined here and $\mu_p(W)$ defined in \cite{PS14}, where $W$ is Hamiltonian Floer homology generated by $H^{(p)}$, are similar but not completely related. One should have the impression that $\mu_p(W)$ is defined ``locally'' with the help of one interval (particularly with the information from its endpoint) while $\mathfrak{o}_X(\phi)$ is defined ``globally'' by using the total multiplicity ({\it without} using the information from endpoints). On the one hand, we point out that $\mu_p(W)$ can also be defined in the Floer-Novikov persistent homology theory developed in \cite{UZ15} on the chain complex level. Here, we give the definition as follows, which is formally the same as $\mu_p(W)$ defined in \cite{PS14} but with extra care on the endpoints.  

For $I = ([a], L)$, define $I^{2c} = ([a \!+\! 2c], L \!-\! 4c)$. Meanwhile, $I ( = ([a], L)) \subset ([b], L')$ (or called bar $([b], L')$ containing $I$) if  there exist $c_a$ and $c_b$ in $\Gamma$ such that 
\[ b + c_b \leq a+ c_a \leq a+ c_a + L \leq b+ c_b + L'. \]

\begin{dfn} (Floer-Novikov multiplicity sensitive spread) For a given interval $I = ([a], L)$, denote $m(\mathcal B_k(Cone(H)_*, I))$ as the multiplicity of bars in degree-$k$ barcode $\mathcal B_k(Cone(H)_*)$ containing $I$. Fix a prime number $p$. For each primitive $\xi_p$, define 
\begin{align*} 
\mu_{p,\xi_p,k}(\phi_H) = \sup \big\{ c \geq 0 \, | \, &\mbox{there exists an interval}\\
&\mbox{$I = ([a], L)$ satisfying condition (A)} \big\} 
\end{align*}
where condition (A) is 
\[ \mbox{$m(\mathcal B_k(Cone(H)_*, I)) = m(\mathcal B_k(Cone(H)_*, I^{2c})$ is {\it not} divisible by $p$}. \]
Finally define
\begin{equation} \label{F-N mds}
\mu_p(\phi_H) = \max_{\tiny{\mbox{primitive $\xi_p$}}} \sup_{k \in \Z} \mu_{p, \xi_p,k}(\phi_H). 
\end{equation}
\end{dfn} 

Unfortunately, the (weak) stabilization proposition - Theorem 4.23 in \cite{PS14} (with {\it corrected version}) can be modified to hold for a general symplectic manifold $M$, but it can not be applied in the same way as in \cite{PS14}, especially when $c_1(M)$ is not zero. On the other hand, in some special cases, for instance, $\Gamma$ is dense, interested reader can verify that $\mathfrak{o}_X(\phi)$ and $\mu_p(\phi_H)$ are comparable. Specifically, there exist positive constants $C_1$ and $C_2$ both depending on $p$ such that $C_1\mu_p(\phi_H) \leq  \mathfrak{o}_X(\phi) \leq C_2 \mu_p(\phi_H)$. The key observation is that in this extreme case when $\Gamma$ is dense, $\R/\Gamma$ can be identified with a single point, which implies that each bar has the same left-endpoint. Definition (\ref{F-N mds}) is then considerably simplified to contain only the information of the lengths of bars. \end{remark} 

Finally, $(d)$ combines all the results together giving the following intermediate theorem which is the crux of the proof of our main theorem.

\begin{theorem} \label{imm-thm} For any closed symplectic manifold $(X, \omega)$, suppose that there exists a Hamiltonian diffeomorphism $\phi = \phi_H$ such that for some $k \in \Z$, $p \nmid m_k$ where $m_k$ is the multiplicity of degree-$k$ concise barcode of $Cone(H)_*$. Then we have 
\[ power_p(X) \geq \frac{1}{24p} \beta_{m_k}(\phi_H). \]
\end{theorem}
\begin{proof} For any given $\ep>0$, we have 
\begin{align*} 
power_p(X) + \ep &= \sup_{\phi \in Ham(X, \omega)} d_H(\phi, Power_p(X)) + \ep \\
& \geq d_H(\phi_H, Power_p(X)) + \ep &&\hspace{-2em} \text{by Definition 1.4} \\
	& \geq d_H(\phi_H, \psi)&&\hspace{-2em} \text{for some $\psi \in Power_p(X)$}\\
& \geq \frac{1}{24p} |\mathfrak{o}_X(\phi_H) - \mathfrak{o}_X(\psi)| &&\hspace{-2em} \text{by Proposition \ref{lip3}} \\
& = \frac{1}{24p} \mathfrak{o}_X(\phi_H) &&\hspace{-2em} \text{by Proposition \ref{non-dis}} \\
& \geq \frac{1}{24p} \beta_{m_k}(\phi_H). &&\hspace{-2em} \text{by Proposition \ref{non-dis}}
\end{align*}
Since $\ep$ is arbitrarily chosen, we get the conclusion. \end{proof}

\subsubsection{Egg-beater model and product structure} From the argument above, we notice that in order to prove Conjecture \ref{conj}, for any given symplectic manifold $X$, we should be able to handle the following two issues:
\begin{itemize}
\item[(i)] find a family of $\phi_{\lambda} = \phi_{H_{\lambda}}^1 \in Ham(X, \omega)$ such that $\beta_{m_k}(\phi_{\lambda}) \to \infty$ as $\lambda \to \infty$;
\item[(ii)] control the non-divisibility of the multiplicity of the concise barcode of $Cone(H_{\lambda})_*$ by prime number $p$.
\end{itemize}
In general, this might be difficult, especially for condition (i). Our main theorem indicates that we can do these on $X = \Sigma_g \times M$ for any closed symplectic manifold, which is the key to succeed in generating results from \cite{PS14}.  Here since we take advantage of a chaotic model called ``egg-beater model'' which has been carefully studied in \cite{PS14} (see Section 5.4), a brief introduction of this is needed.\\

\noindent{\bf Egg-beater model}. Egg-beater model $(\Sigma_g, \phi_{\lambda})$ with $g \geq 4$ is constructed to create large action gap. On $\Sigma_g$, we will focus on a pair of annuli intersecting each other in such a way that in each component separated by the annuli there exists some genus (see Figure 2 in \cite{PS14}). Moreover, our Hamiltonian dynamics comes from a family of shear flow $\phi^t_{\lambda}$ (highly degenerated) generated by a family of special Hamiltonian functions supported only in the union of these annuli (see Figure 3 in \cite{PS14}). The upshot is that we have a well-defined Floer chain complex (for a non-contractible loop), denoted by  
\begin{equation} \label{egg-beater}
CF_*(\Sigma_g, \phi_{\lambda}^p)_{\alpha_{\lambda}} 
\end{equation}
where $\alpha_{\lambda}$ represents a homotopy class of non-contractible loops and $\phi_{\lambda}$ is a Hamiltonian diffeomorphism, parametrized by sufficiently large $\lambda$. The corresponding generating Hamiltonian of $\phi_{\lambda}$ is denoted by  $H_\lambda$. The generators of this chain complex are non-contractible Hamiltonian orbits which can be identified with fixed point and interestingly, by the construction of $\phi_{\lambda}^p$, there are exactly $2^{2p}$ $p$-tuples of generators in the sense that if $z$ is a non-degenerate fixed point, then each one from the following cyclic permutation 
\begin{equation} \label{egg-p-tuples}
\{z, \phi_{\lambda} z, \dots , \phi_{\lambda}^{p-1} z \} 
\end{equation}
is also a non-degenerate fixed point. Moreover, actions and indices of $\phi^j_{\lambda} z$ are the same for all $j \in \{0, \dots, p-1\}$. The rotation action $R_p$ acts on generators as $R_p (\phi^j_{\lambda} z) = \phi_{\lambda}^{j+1 \,{\mod} \,p}z$. 

This model itself provides an example that we can carry on explicit computation concerning (i) and (ii) mentioned above. On the one hand, Proposition 5.1 in \cite{PS14} confirms the asymptotic (to infinity) behavior of action gap required in (i). On the other hand, we have the following proposition, proved in Section \ref{egg}, handling the other issue on multiplicity in (ii).
\begin{prop} \label{multi-egg}
For any given prime number $p \geq 3$, the total multiplicity of concise barcode of self-mapping cone of $CF_*(\Sigma_g , \phi_{\lambda}^p)_{\alpha_{\lambda}}$ is $2^{2p}$. In particular, there exists some degree $k$ such that $p \nmid m_k$ where $m_k$ is the multiplicity of degree-$k$ concise barcodes. \\
\end{prop}

\noindent{\bf Product structure}. First, for the following product of Floer chain complexes, 
\[ CF_*(\Sigma_g \times M, \phi^p_{\lambda} \times {\mathds 1})_{\alpha_{\lambda} \times \{pt\}} = CF_*(\Sigma_g, \phi^p_{\lambda})_{\alpha_{\lambda}} \otimes CF_*(M, {\mathds 1})_{\{pt\}} \]
by the recipe (\ref{recipe}) above, we will consider degree-$1$ concise barcode of self-mapping cone of $CF_*(\Sigma_g \times M, \phi^p_{\lambda} \times {\mathds 1})_{\alpha_{\lambda} \times \{pt\}}$, denoted by  
\begin{equation} \label{htp-egn-prod} 
Cone_{\otimes}(H_{\lambda})_*.
\end{equation}
By a careful study of barcodes under product structure, we can prove the following proposition in Section \ref{prod}.
\begin{prop} \label{action-prod}For any positive $i \in \Z$, the length of the $i$-th bar in degree-$1$ concise barcode of $Cone_{\otimes}(H_{\lambda})_*$ satisfies 
\[ \beta_i(\phi^p_{\lambda} \times {\mathds 1}) \to \infty \,\,\,\, \mbox{as} \,\,\, \lambda \to \infty\]
for any $i \leq m_1$ where $m_1$ is the multiplicity of degree-$1$ concise barcode.  
\end{prop}

\begin{dfn} \label{qb} Define {\it $k$-th quantum Betti number of $M$} $qb_k(M)$ by 
\[ qb_k(M) = \sum_{s \in \Z} b_{k+2Ns}(M)\]
where $b_{k+2Ns}(M)$ is the classical $(k+2Ns)$-th Betti number of $M$ and $N$ is minimal Chern number of $M$. Note when $N$ is sufficiently large, for instance, $c_1(TM) =0$, $qb_k(M) = b_k(M)$. \end{dfn}
Finally, by referring to the CZ-index formula of the generators of egg-beater model (see Theorem 5.2 in \cite{egg-team}) and Definition \ref{qb}, we can show the following proposition in Section \ref{prod}.
\begin{prop} \label{multi-prod} Denote by $m_1$ the multiplicity of degree-$1$ concise barcode of $Cone_{\otimes}(H_{\lambda})_*$. If 
\begin{equation} \label{div-cond}
p \nmid (qb_{p}(M) + 2 qb_{0}(M) + qb_{-p}(M)) 
\end{equation}
then $p \nmid m_1$. In particular, if $p > 2\sum_{i=0}^{\dim(M)} b_i(M)$, then (\ref{div-cond}) is always satisfied. \end{prop}
Note that the non-divisibility condition in Proposition \ref{multi-prod} can sometimes be improved. See Remark \ref{div-imp}.

\subsubsection*{Acknowledgements}  I am grateful to my advisor, Michael Usher, for various useful conversations and suggestions. I am also grateful to L. Polterovich and D. Rosen for useful conversations. Some of these conversations occurred during my visit to ICERM in Summer 2015. Moreover, my understanding of egg-beater model profited from Computational Symplectic Topology Workshop that was organized by R. Hind, Y. Ostrover, L.Polterovich and M. Usher, which was hosted by Tel Aviv University and ICREM. I thank them for their hospitality. I thank S. Barannikov for pointing out the missing reference \cite{Bar94}. Finally I am grateful to the anonymous referee for his/her suggestions and corrections. 

\section{Preliminaries on Hamiltonian Floer theory}
Hamiltonian Floer theory associates a Floer-type complex to any ``non-degenerate'' Hamiltonian $H: \R/\Z \times X \to \R$ on a compact symplectic manifold $(X, \omega)$. We will briefly review some of its key ingredients for our purpose in this section. These ingredients have been progressively generalized along the way from \cite{F89}, \cite{SZ92}, \cite{HS}, \cite{FO}, \cite{Pa}. 

Suppose $(X, \omega)$ is a closed connected symplectic manifold. Given any smooth Hamiltonian function $H: \R/\Z \!\times\! X \!\to\! \R$, it induces a time-dependent Hamiltonian vector field $X_H$ by taking advantage of non-degeneracy of $\omega$, i.e., 
\[ \omega (\cdot, X_H) = d(H(t, \cdot)). \]
This vector field $X_H$ induces a flow denoted by $\phi_H^t$. A fixed point of the time-1 map $\phi_H^1$ is corresponding to a loop $\gamma: \R/\Z \to X$ such that $\gamma(t) = \phi^t_H (\gamma(0))$. We say that $H$ is {\it non-degenerate} if for each such loop $\gamma$, 
\[ (d\phi_H^1)_{\gamma(0)}: T_{\gamma(0)} X \to T_{\gamma(0)} X \]
has all eigenvalues distinct from $1$. In terms of constructing Floer chain complex, we will not use all the loops. There are two different versions. One is to use all the contractible loops. Denote the collection of all such loops by $\mathcal{L}(X)$ (or $\mathcal{L}_{\{pt\}}(X)$ if necessary). The other is to use non-contractible loops in a fixed homotopy class $\alpha \in \pi_1(M)$. Denote the collection of all such loops by $\mathcal{L}_{\alpha}(X)$. 

We will start from the construction of the first one. View $\gamma$ as a boundary of a embedded disk $D^2$ in $X$, i.e., there is a map $u: D^2 \to M$ and $u|_{S^1} = \gamma$. Now consider a covering space of $\mathcal{L}(X)$, denoted by  $\widetilde{\mathcal{L}(X)}$ constructed by 
\begin{align*}
\widetilde{\mathcal{L}(X)} = \Bigg\{& \begin{array}{cc} \mbox{equivalence class} \,\,[\gamma, u] \\ \mbox{of pair $(\gamma, u)$} \end{array}  \, \bigg|\\
&\begin{array}{cc} (\gamma, u) \,\,\mbox{is equivalent to} \,\, (\tau, v) \,\,\mbox{$\Longleftrightarrow$} \\ \gamma(t) = \tau(t) \,\,\mbox{and} \,\, [u \# (-v)] \in \ker([\omega]) \cap \ker(c_1) \end{array} \Bigg\}, 
\end{align*}
For each $[\gamma, u] \in \widetilde{\mathcal{L}(X)}$, there are two functions associated to it. One is {\it action functional $\mathcal A_{H}: \widetilde{\mathcal{L}(X)} \to \R$} defined by 
\[ \mathcal A_H([\gamma, u]) = -\int_{D^2} u^* \omega + \int_0^1 H(t, \gamma(t)) dt. \]
The other one is Conley-Zehnder index $\mu_{CZ}: \widetilde{\mathcal{L}(X)} \to \Z$ defined by, roughly speaking, counting rotation of $d\phi_H^t$ along $\gamma(t)$ with the help of trivialization induced by $u$. Its explicit definition can be referred to \cite{RS93}. Because of the conditions in $\widetilde{\mathcal{L}(X)}$ above, action functional and Conley-Zehnder index of $[\gamma, u]$ are both well-defined. As a vector space over ground field $\K$, for any $k \in \Z$, degree-$k$ part of Floer chain complex $CF_k(H,J)$ (or $CF_k(H, J)_{\{pt\}}$ if necessary) is defined as 
\begin{align*} 
	\Bigg\{& \sum_{\tiny{\begin{array}{c}[\gamma,u]\in\widetilde{\mathcal{L}(X)},\\ \mu_{CZ}([\gamma,u])=k   \end{array}}}a_{[\gamma,u]}[\gamma,u]\;\Bigg|\\
&\quad a_{[\gamma,u]}\in\mathcal{K},(\forall C\in \R)(\#\{[\gamma,u]|a_{[\gamma,u]}\neq 0,\,\mathcal{A}_{H}([\gamma,u])>C\}<\infty) \Bigg\}. 
\end{align*}
Now denote \[ s_0 =\left\{ w\co S^2\to X \,| \, \langle c_1(TX),w_*[S^2]\rangle=0\right\}.\]
Note that if we change $[\gamma, u]$ by gluing some sphere $w \in s_0$ on the capping $u$, it will change the action functional by $-\int_{S^2} w^* \omega$, possibly not zero, but keep the degree the same. More importantly, if such action difference is non-zero, then $[\gamma, u] \neq [\gamma, u \# w]$ in $CF_k(H, J)$, which implies that as a vector space over $\K$, $CF_k(H, J)$ is in general infinitely dimensional. It is finite dimensional if $\omega$ vanishes on the image of Hurewicz map $i: \pi_2(X) \to H_2(X, \Z)/{\mbox{Torsion}}$, in which case $X$ is usually called {\it weakly exact} (a stronger condition called {\it symplectically aspherical} if both $\omega$ and $c_1$ vanish).

In order to overcome this dimension issue, \cite{HS} suggests to consider a bigger coefficient field - Novikov field $\Lambda^{\K, \Gamma}$ defined as follows, 
\[ \Lambda^{\mathcal K,\Gamma}=\left\{\sum_{g\in \Gamma}a_gt^g\left|a_g\in\mathcal{K},(\forall C\in \R)(\#\{g|a_g\neq 0,\,g<C\}<\infty)\right.\right\} \]
where $\Gamma = \left\{ \int_{S^2} w^* \omega \,|\, w \in s_0\right\} \leq \R$ and $t$ is a formal variable. It is then easy to check that $CF_k(H, J)$ is now a finite dimensional vector space over $\Lambda^{\K, \Gamma}$ and its dimension is equal to the number of $\gamma \in {\mathcal{L}(X)}$ such that there is $u: D^2 \to X$ with $u|_{\partial D^2} = \gamma$ and $\mu_{CZ}([\gamma,u]) =k$. More explicitly, $t^g \cdot [(\gamma, u)] = [(\gamma, u \# w)]$ where $w$ is a sphere with symplectic area equal to $g$. Moreover, with the help of action functional $\mathcal A_H$, we can turn $CF_k(H, J)$ into a (finite dimensional) filtered vector space over $\Lambda^{\K, \Gamma}$ by defining an associated filtration function as follows, 
\[ \ell_H \left( \sum a_{[\gamma, u]} [\gamma, u] \right) = \max \{ \mathcal A_H([\gamma, u]) \,| \, a_{[\gamma, u]} \neq 0\}. \]
Note that field $\Lambda^{\K, \Gamma}$ itself has a well-defined valuation,
\[ \nu\left(\sum_{g\in \Gamma}a_gt^g\right) = \min\{ g \in \R \,| \, a_g \neq 0\}. \]
Then for any $\lambda \in \Lambda^{\K, \Gamma}$ and $c \in CF_k(H, J)$, $\ell_H(\lambda c) = \ell_H(c) - \nu(\lambda)$. 

Next, graded vector space $CF_*(H, J)$ will become a chain complex once we define the (Floer) boundary operator $(\partial_{H,J})_*$. The degree-$k$ part of boundary operator  $(\partial_{H,J})_{k}: CF_k(H, J) \to CF_{k-1}(H,J)$ is defined by counting the solution (modulo $\R$-translation) of the following partial differential equation, as a formal negative gradient flow of $\mathcal A_H$,
\begin{equation} \label{dfn-boundarymap}
 \frac{\partial u}{\partial s} + J_t(u(s,t)) \left( \frac{\partial u}{\partial t} - X_H(t, u(s,t)) \right) =0, 
 \end{equation} 
where $\{J_t\}_{0 \leq t \leq 1}$ is a family of almost complex structure compatible with $\omega$ and $u: \R \times \R/\Z \to X$ such that  \begin{itemize}
\item{} $u$ has finite energy $E(u) = \int_{\R \times \R/\Z} \left|\frac{\partial u}{\partial s} \right|^2 dtds$;
\item{} $u$ has asymptotic condition $u(s,\cdot) \to \gamma_{\pm} (\cdot)$ as $s \to \pm \infty$;
\item{} $\mu_{CZ}([\gamma_-,w_-]) -  \mu_{CZ}([\gamma_+,w_+]) =1$ and $[\gamma_+, w_+] = [\gamma_+, w_- \#u]$.
\end{itemize} 
By now, it is a non-trivial but well-known and standard fact that 
\begin{theorem}[See \cite{HS} for semi-positive $X$ or see \cite{Pa} for a general $X$] \label{bd-Floer} $(\partial_{H,J})_*$ is well-defined such that $\partial_{H,J} \circ \partial_{H,J} =0$. Moreover, there exists some $\delta>0$ (coming from Gromov-Floer compactness theorem) such that for each $c \in CF_*(H, J)$, $\ell_H(\partial_{H,J} c) \leq \ell_H(c) - \delta$. \end{theorem}

Next, for complex $(CF_*(H, J), (\partial_{H,J})_*, \ell_H)$, its construction clearly depends on the pair $(H, J)$. We will now recall the relation between two such complexes if they are constructed from different $(H,J)$. Specifically, if we have $(H_-, J_-)$ and $(H_+, J_+)$, the standard way is to form a regular homotopy $(\mathcal H, \mathcal J)$ (with homotopy parameter $s \in \R$) between them so that when $s \ll 0$, $(\mathcal H_s, \mathcal J_s) = (H_-, J_-)$ and when $s \gg0$, $(\mathcal H_s, \mathcal J_s) = (H_+, J_+)$. For instance, we can use a cut-off function $\alpha(s)$, i.e., $\alpha(s)$ is monotone increasing and $\alpha(s) = 0$ when $s \ll 0$ and $\alpha(s) =1$ when $s \gg 0$, to form the following homotopy, 
\[ \mathcal H_s(t,\cdot) = (1- \alpha(s)) H_-(t, \cdot) + \alpha(s) (H_+(t, \cdot)),\]
and it is similar to consider such a homotopy between almost complex structures $\mathcal J_s$. The upshot is that there exists a well-defined chain map (usually called {\it continuation map}) $\Phi_{(\mathcal H, \mathcal J)}: CF_*(H_-, J_-) \to CF_*(H_+, J_+)$ constructed (similar to (\ref{dfn-boundarymap})) by counting the solution of the following parametrized partial differential equation 
\begin{equation} \label{dfn-chainmap}
 \frac{\partial u}{\partial s} + (J_t)_s(u(s,t)) \left( \frac{\partial u}{\partial t} - X_{H_s}(t, u(s,t)) \right) =0, 
 \end{equation} 
where again $u$ is required to satisfy certain conditions as above except here we require that $\mu_{CZ}([\gamma_-, u_-]) = \mu_{CZ}([\gamma_+, u_+])$. It is a chain map by standard Floer gluing argument. See the proof of Theorem 11.1.15 in \cite{AD}. 

\begin{remark} \label{diff-htps} Here we emphasize that if we use another homotopy $(\mathcal H', \mathcal J')$, the same construction will give another chain map $\Phi_{(\mathcal H', \mathcal J')}: CF_*(H_-, J_-) \to CF_*(H_+, J_+)$. These two chain maps are actually chain homotopic to each other, induced by a ``1-homotopy'' (homotopy of a homotopy) between $(\mathcal H, \mathcal J)$ and $(\mathcal H', \mathcal J')$. So there exists a degree-1 map 
$$K: CF_*(H_-, J_-) \to CF_{*+1}(H_+, J_+)$$ such that 
\begin{equation}\label{shift-htp}
 \Phi_{(\mathcal H, \mathcal J)} - \Phi_{(\mathcal H', \mathcal J')} = \partial \circ K + K \circ \partial. 
 \end{equation}
The explicit construction of $K$ is carried out in Lemma 6.3 in \cite{SZ92} or Section~2 in \cite{Ush11}. Moreover, $K$ shifts the filtration up to $\int_0^1 \max_X(H_- - H_+) dt$. \end{remark} 

Now using another homotopy $(\tilde{\mathcal H}, \tilde{\mathcal J})$ from $(H_+, J_+)$ to $(H_-, J_-)$ gives a well-defined chain map $\Phi_{(\tilde{\mathcal H}, \tilde{\mathcal J})}: CF_*(H_+, J_+) \to CF_*(H_-, J_-)$. To sum up, we have the following picture, 
\begin{equation} \label{two-htps}
\xymatrix{
CF_*(H_-, J_-) \ar[r]^{\Phi_{(\mathcal H, \mathcal J)}} \ar@/^2.0pc/@[black][rr]^{{\mathds 1}} & CF_*(H_+, J_+) \ar[r]^{\Phi_{(\tilde{\mathcal H}, \tilde{\mathcal J})}} &CF_*(H_-, J_-)}
\end{equation}
where the identity map ${\mathds 1}$ can be regarded as the induced chain map by the obvious constant homotopy $(\mathcal H_{const}, \mathcal J_{const})$ between $(H_-, J_-)$ and itself. On the one hand, the well-known gluing argument (see Chapter 10 in \cite{MS} or B.10 in \cite{Pa}) implies that 
\begin{equation} \label{gluing}
 \Phi_{(\tilde{\mathcal H}, \tilde{\mathcal J})} \circ \Phi_{(\mathcal H, \mathcal J)} = \Phi_{(\mathcal H_R, \mathcal J_R)}
 \end{equation} where the right hand side is an induced chain map from ``gluing" homotopy $(\mathcal H_R, \mathcal J_R)$ (for some $R \gg 0$) from $(H_-, J_-)$ to itself constructed from $(\mathcal H, \mathcal J)$ and $(\tilde{\mathcal H}, \tilde{\mathcal J})$. On the other hand, it is also well-known (see explicit construction in Page 14 in \cite{Ush11}) that the resulting $\Phi_{(\mathcal H_R, \mathcal J_R)}$ is chain homotopic to ${\mathds 1}$. This leads to the following two important consequences. 

In terms of the algebraic structure, different choices of pair $(H, J)$ induce the same (up to isomorphism) Floer homology $H_*(CF(H,J), \partial_{H,J}) : = HF_*(X)$. Therefore, we can compute $HF_*(X)$ by choosing a preferred Hamiltonian function $H$. In most cases, we will choose a $C^2$-small $H$ so the Hamiltonian orbit $\gamma(t)$ of $H$ will be degenerated to points, which makes the corresponding analysis much easier. Moreover, we have the following result. 

\begin{theorem} [See Theorem 6.1 in \cite{HS} for semi-positive $X$ or see Theorem 10.7.1 in \cite{Pa} for a general $X$]\label{Ar-conj} For any degree $k \in \Z$, 
\[ HF_k(X) \simeq \bigoplus _{j =k \, {\mod} \, 2N} H_j(X,\K) \otimes \Lambda^{\K, \Gamma}. \]
In particular, when $k =0$ or $\dim(X)$, $HF_k(X) \neq 0$.
\end{theorem}
 
In terms of the filtration, neither chain maps in (\ref{two-htps}) nor the chain homotopy from $\Phi_{(\mathcal H_R, \mathcal J_R)}$ to ${\mathds 1}$ preserves filtrations. Instead, they have certain shifts related with the initial Hamiltonian functions $H_-$ and $H_+$. We will summarize this into the following short theorem with the help of terminology defined in Definition \ref{dfn-mc}. 

\begin{theorem} \label{Floer-energy} Given any two filtered Floer chain complexes $(CF_*(H_-, J_-), (\partial_{H_-, J_-})_*, \ell_{H_-})$ and $(CF_*(H_+, J_+), (\partial_{H_+, J_+})_*, \ell_{H_+})$, they are $(\int_0^1 \!\max_X (H_- - H_+) dt, \int_0^1 -\min_X (H_- - H_+) dt)$-quasiequivalent.\footnote{Here we choose normalized Hamiltonians $H_+$ and $H_-$ in the sense that $\int_X H_+(t, \cdot) = \int_X H_-(t, \cdot) =0$ for every $t \in [0,1]$. Therefore, $ \int_0^1 -\min_X (H_- - H_+) dt \geq 0$.}  \end{theorem} 

Now we move to the Hamiltonian Floer theory with respect to a non-contractible loop (representing a homotopy class $\alpha$). We will denote the corresponding Floer chain complex as $CF_*(H, J)_{\alpha}$ and Floer homology as $HF_*(X)_{\alpha}$. We simply remark here that almost all the ingredients above can be defined and constructed in a parallel way by starting from a covering space of $\mathcal{L}_{\alpha}(X)$ once we fixed a reference non-contractible loop in the homotopy class $\alpha$. The general construction has been carried out explicitly in Section 5 in \cite{Ush13}. Proposition 5.1 in \cite{Ush13} implies that Theorem \ref{Floer-energy} still holds for $CF_*(H, J)_{\alpha}$. What we want to emphasize is that Theorem \ref{Ar-conj} is not true for $HF_*(X)_{\alpha}$. In fact, we can readily show the following result. 

\begin{theorem} \label{Ar-conj-cont} $HF_*(X)_{\alpha} = 0$ if $\alpha$ is represented by a non-contractible loop. \end{theorem}

Indeed, since $HF_*(X)$ is independent of Hamiltonian $H$. A $C^2$-small $H$ will only provide critical points (as constant periodic orbits), which make $HF_*(X)_{\alpha}$ has no generators.

\section{Proof of Proposition \ref{cont2}} \label{p14}
\begin{lemma} \label{cont1} Assume that $H = F^{(p)}$ for some Hamiltonian function $F$. For any $a \in \N$, we have the following commuting diagram
\begin{center}
\begin{tikzpicture}
  \matrix (m) [matrix of math nodes,row sep=3em,column sep=4em,minimum width=2em]
  {
      CF_{k}(H^{(p)}, J_{t + \frac{a}{p^2}})_{\alpha} & CF_{k}(H^{(p)}, J_{t + \frac{a+1}{p^2}})_{\alpha} \\
     CF_{k}(H^{(p)}, J_{t})_{\alpha} & CF_{k}(H^{(p)}, J_{t + \frac{1}{p^2}})_{\alpha} \\};
  \path[-stealth]
    (m-1-1) edge node [left] {$C_a$} (m-2-1)
            edge node [above] {$R_{p^2}$} (m-1-2)
    (m-2-1.east|-m-2-2) edge node [below] {$R_{p^2}$}
            node [above] {}(m-2-2)
    (m-1-2) edge node [right] {$(R_{p^2})_*(C_a)$} (m-2-2);
           \end{tikzpicture}
\end{center}
where $C_a$ is some continuation map. 
\end{lemma} 

\begin{proof} Suppose that $x(t+a/p^2)$ is a (capped) periodic orbit as a generator of $CF_{k}(H^{(p)}, J_{t + a/p^2})_{\alpha}$ and that there exists a Floer connecting orbit $u(s,t)$ satisfying parametrized partial differential equation (\ref{dfn-chainmap}) connecting $x(t+a/p^2)$ and $y(t)$ for some $y(t)$ as a generator of $CF_k(H^{(p)}, J_{t})$, followed by a rotation on the parameter $t$ to $t+1/p^2$. In other words, we have 
\[  x(t+a/p^2) \xrightarrow{u(s,t)} y(t) \xrightarrow{R_{p^2}} y(t+ 1/p^2). \]
Then via a rotation by $1/p^2$, we will also get a Floer connecting orbit, namely $u(s,t+1/p^2)$, which satisfies the following parametrized partial differential equation (used to construct $(R_{p^2})_*(C_a)$),
\begin{equation} 
 \frac{\partial u}{\partial s} + (J_{t+1/p^2})_s(u(s,t)) \left( \frac{\partial u}{\partial t} - X_{H_s}(t+1/p^2, u(s,t)) \right) =0, 
 \end{equation} 
connecting $x(t + (a+1)/{p^2})$ and $y(t+ 1/{p^2})$. In other words, 
\[ x(t+ a/p^2) \xrightarrow{R_{p^2}} x(t+ (a+1)/p^2) \xrightarrow{u(s,t+1/p^2)} y(t+1/p^2). \]
By a symmetric argument, there is a one-to-one correspondence between these two approaches. Therefore, this diagram commutes by the definition of Floer continuation map.  \end{proof}

\begin{proof} [Proof of Proposition \ref{cont2}] Because $R_{p^2}$ is well-defined, consider composition $C_1 \circ R_{p^2}$ (which is a map on $CF_k(H^{(p)}, J_t)_{\alpha}$ itself). We observe that 
\begin{align*}
 (C_1 \circ R_{p^2}) \circ (C_1 \circ R_{p^2}) &= C_1\circ (R_{p^2} \circ C_1) \circ R_{p^2} \\
 &= C_1 \circ ((R_{p^2})_*(C_1) \circ R_{p^2}) \circ R_{p^2}\\
 &= (C_1 \circ (R_{p^2})_*(C_1)) \circ R_{p^2}^2   
 \end{align*} 
where the second equality comes from Lemma \ref{cont1}. By setting 
\[ C_2 = C_1 \circ (R_{p^2})_*(C_1) \]
we have $(C_1 \circ R_{p^2})^2 = C_2 \circ R_{p^2}^2.$ Applying Lemma \ref{cont1} inductively, we can get 
\[ (C_1 \circ R_{p^2})^p = C_p \circ R_{p^2}^p = C_p \circ R_p \]
where $C_p$ is recursively defined by 
\[ C_p = C_1 \circ (R_{p^2})_*(C_{p-1}). \]
Therefore, in order to get the conclusion, set $C = C_p$ (which is determined by $C_1$) and set $C'= C_1$ so $S = C_1 \circ R_{p^2}$.\end{proof}

\section{Proof of Proposition \ref{p-he}}\label{p16}

We will prove Proposition \ref{p-he} by proving the following general result first. Recall the double map $\mathcal D_{\Phi}$ on self-mapping cone $Cone_C(\Phi)_*$ is defined by, for any $(x,y) \in Cone_C(\Phi)_*$, 
\[ \mathcal D_{\Phi}((x,y)) = (\Phi(x), \Phi(y)).\]
Also, recall that two chain maps $\Phi$ and $\Psi$ are filtered homotopic if there exists a homotopy $K$ such that  $\Phi - \Psi = \partial K + K \partial$ and $K$ preserves the filtration. Meanwhile, two Floer-type complexes are filtered isomorphic to each other if they are chain isomorphic to each other and the corresponding isomorphisms preserve the filtrations. 

\begin{lemma} \label{htp-cone} For two filtration preserving chain maps $\Phi$ and $\Psi$ on Floer-type complex $(C_*, \partial_*, \ell)$, if $\Phi$ and $\Psi$ are filtered homotopic, then the associated self-mapping cones $Cone_C(\Phi)_*$ and $Cone_C(\Psi)_*$ are filtered isomorphic to each other. \end{lemma}

\begin{proof} Suppose the homotopy between $\Phi$ and $\Psi$ is $K$, which preserves the filtration. Construct a map $F: Cone_C(\Phi)_* \to Cone_C(\Psi)_*$ by 
\[ F = \left(\begin{array}{cc}
{\mathds 1} & -K \\
0 & {\mathds 1}
\end{array}\right). \]
First $F$ preserves filtration. In fact, for any $(x,y) \in Cone_C(\Phi)_*$, we have 
\begin{align*}
\ell_{co}(F(x,y)) = \ell_{co} ((x - Ky, y)) &= \max \{\ell_C(x- Ky), \ell_C(y) \}\\
& \leq \max\{\ell_C(x), \ell_C(Ky), \ell_C(y) \} \\
& =  \max\{\ell_C(x), \ell_C(y) \}\\
& = \ell_{co}((x,y)).
\end{align*}
Second, $F$ is a chain map because 
\begin{align*} 
\left( \begin{array}{cc}
{\mathds 1} & -K \\
0 & {\mathds 1}
\end{array} \right) \cdot \left( \begin{array}{cc}
\partial_C & -\Phi \\
0 & - \partial_C
\end{array} \right)&= \left( \begin{array}{cc}
\partial_C & - \Phi + K \partial_C \\
0 & -\partial_C
\end{array} \right) \\
&=   \left( \begin{array}{cc}
\partial_C  & - \Psi - \partial_C K \\
0 & -\partial_C
\end{array} \right)\\
&= \left( \begin{array}{cc}
\partial_C & -\Psi \\
0 & - \partial_C
\end{array} \right) \cdot \left( \begin{array}{cc}
{\mathds 1} & -K \\
0 & {\mathds 1}
\end{array} \right).
\end{align*}
Third, $F$ is an isomorphism because we have its inverse 
\[ F^{-1} = \left(\begin{array}{cc}
{\mathds 1} & K \\
0 & {\mathds 1}
\end{array}\right) \]
which is easily seen to preserve the filtration too. \end{proof}

\begin{proof} [Proof of Proposition \ref{p-he}] For the first conclusion, by the standard result of Floer theory explained Section 2, different choices of continuation maps $C_1$ and $C_2$ will result in different but filtered homotopic chain maps. Therefore, it comes from the direct application of Lemma \ref{htp-cone} to 
\[ \Phi = C_1 \circ R_p - \xi_p \cdot {\mathds 1} \,\,\,\,\,\,\,\mbox{and}\,\,\,\,\,\,\,\, \Psi = C_2 \circ R_p - \xi_p \cdot {\mathds 1}. \] 
For the second conclusion, note that for the following diagram 
\[ \xymatrix{  CF_{k}(H^{(p)}, J_{t})_{\alpha} \ar[r]^{R_p} \ar[d]_{\partial_k} & CF_{k}(H^{(p)}, J_{t + \frac{1}{p}})_{\alpha} \ar[r]^{C} \ar[d]^{(R_{p})_*(\partial_k)} & CF_{k}(H^{(p)}, J_t)_{\alpha} \ar[d]^{\partial_k} \\    CF_{k-1}(H^{(p)}, J_{t})_{\alpha} \ar[r]^{R_p} & CF_{k-1}(H^{(p)}, J_{t + \frac{1}{p}})_{\alpha} \ar[r]^{C}  & CF_{k-1}(H^{(p)}, J_t)_{\alpha} }\]
the standard argument in Floer theory shows continuation map $C$ is a chain map, that is, $\partial_k \circ C = C \circ (R_p)_*(\partial_k)$ (where $(R_p)_*(\partial_k)$ is the degree-$k$ differential of complex $CF_k(H^{(p)}, J_{t+\frac{1}{p}})_{\alpha}$). Therefore, let $T : = C \circ R_p$, we can check that 
\[ \partial_k \circ T = \partial_k \circ C \circ R_p = C \circ (R_p)_*(\partial_k) \circ R_p = C \circ R_p \circ \partial_k = T \circ \partial_k. \]
Similarly, $S := C' \circ R_{p^2}$ commutes with $\partial_k$ for any $k \in \Z$. Then, by the definition of $\partial_{co}$, we can check that 
\begin{align*} 
\left( \begin{array}{cc}
\partial & -(T-\xi_p \cdot {\mathds 1}) \\
0 & -\partial
\end{array} \right) \cdot \left( \begin{array}{cc}
T & 0\\
0 & T
\end{array} \right)&= \left( \begin{array}{cc}
\partial T & -(T-\xi_p \cdot {\mathds 1}) \circ T \\
0 & - \partial T
\end{array} \right) \\
&=   \left( \begin{array}{cc}
T \partial   & - T \circ (T-\xi_p \cdot {\mathds 1}) \\
0 & -T \partial
\end{array} \right) \\
& = \left( \begin{array}{cc}
T & 0\\
0 &T
\end{array} \right) \cdot \left( \begin{array}{cc}
\partial & -(T-\xi_p \cdot {\mathds 1}) \\
0 & -\partial
\end{array} \right).
\end{align*}
Similarly, since $S$ commutes with $T$, 
\begin{align*} 
\left( \begin{array}{cc}
\partial & -(T-\xi_p \cdot {\mathds 1}) \\
0 & -\partial
\end{array} \right) \cdot \left( \begin{array}{cc}
S & 0\\
0 & S
\end{array} \right)&= \left( \begin{array}{cc}
\partial S & -(T-\xi_p \cdot {\mathds 1}) \circ S \\
0 & - \partial S
\end{array} \right) \\
&=   \left( \begin{array}{cc}
S \partial   & - S \circ (T-\xi_p \cdot {\mathds 1}) \\
0 & -S \partial
\end{array} \right) \\
& = \left( \begin{array}{cc}
S & 0\\
0 &S
\end{array} \right) \cdot \left( \begin{array}{cc}
\partial & -(T-\xi_p \cdot {\mathds 1}) \\
0 & -\partial
\end{array} \right).
\end{align*}

 \end{proof}
 
\section {Proof of Theorem \ref{p-tuples}}\label{p17}

Recall our set-up. For a non-degenerate Hamiltonian $H$, there is a Floer chain complex $(CF_*(H^{(p)}, J_t), \partial_{H, J})$. Its self-mapping cone of linear map $T - \xi_p \cdot {\mathds 1}$, $(Cone_*(T - \xi_p \cdot {\mathds 1}), \partial_{co})$, is in general a Floer-type complex over Novikov field $\Lambda^{\K, \Gamma}$ where $T$ is a strictly lower filtration perturbation of rotation $R_p$, \emph{i.e.}, $T = C \circ R_p = R_p + P$ where for any $x$ from domain, $\ell(P(x)) < \ell(x)$. Moreover, there exists a $\Lambda^{\K, \Gamma}$-linear chain map $\mathcal D_{T}$ on $(Cone_*(T - \xi_p \cdot \I), \partial_{co})$, defined by an action of $T$ on each component. If in the $p$-th power situation defined in the introduction, there exists a $\Lambda^{\K, \Gamma}$-linear chain map $\mathcal D_S$ on $(Cone_*(T - \xi_p \cdot {\mathds 1}), \partial_{co})$ such that $\mathcal D_S^p = \mathcal D_T$ where $S$ is a strictly lower filtration perturbation of rotation $R_{p^2}$. In this section, we will prove the important Theorem \ref{p-tuples}, 

\subsection{Perturb to be a group action}
Recall our definitions,
\begin{equation} \label{dfn-T}
\mathcal D_T = \mathcal D_{R_p} + C_T\,\,\,\,\mbox{and}\,\,\,\,\mathcal D_S = \mathcal D_{R_{p^2}} + C_S,
\end{equation}
and by Proposition \ref{p-he},
\begin{equation} \label{dfn-S}
\D_{T} \partial_{co} = \partial_{co} \D_{T} \,\,\,\,\mbox{and}\,\,\,\, \D_{S} \partial_{co} = \partial_{co} \D_{S}
\end{equation}
because $T \partial = \partial T$ and $S \partial = \partial S$. In particular, we know that $\mathcal D_{T}$ and $\mathcal D_{S}$ exactly preserve filtrations. Moreover, due to Floer continuation map, there exists a constant $\hbar>0$ such that for any $x \in Cone_*(T - \xi_p \cdot {\mathds 1})$, 
\begin{equation} \label{cont-map}
\ell_{co}(C_T x) \leq \ell(x) - \hbar \,\,\,\,\mbox{and}\,\,\,\, \ell_{co}(C_S x) \leq \ell(x) - \hbar. 
\end{equation}
Because of this, we have the following lemma. 

\begin{lemma} \label{inv}
$\mathcal D_T$ and $\mathcal D_S$ (if it exists) are invertible. 
\end{lemma}
\begin{proof} First, because $\D^p_{R_p} = {\mathds 1}$ and $\D^{p^2}_{R_{p^2}} = {\mathds 1}$,
\begin{equation} \label{al-ga}
\mathcal D_T^p = {\mathds 1} - Q_T \,\,\,\,\mbox{and}\,\,\,\, \D_{S}^{p^2} = {\mathds 1} - Q_S.
\end{equation} 
Here, $Q_T$ is a combination of $C_T$ and $\D_{R_p}$, so it strictly lowers filtration (by at least $\hbar$). Similarly, $Q_S$ is a combination of $C_S$ and $\D_{R_{p^2}}$, so it also strictly lowers filtration (by at least $\hbar$). Then $\mathcal D_T \circ (\mathcal D_T)^p =  (\mathcal D_T)^{p+1}=  (\mathcal D_T)^p \circ \mathcal D_T$, so by (\ref{al-ga}), 
\begin{equation} \label{commutes-T}
\D_T \circ ({\mathds 1} - Q_T) = ({\mathds 1} - Q_T)\circ \D_T \,\,\,\,\Rightarrow \,\,\,\, \D_T Q_T = Q_T \D_T. 
\end{equation} 
Similarly, $\mathcal D_S \circ (\mathcal D_S)^{p^2} =  (\mathcal D_S)^{p^2+1}=  (\mathcal D_S)^{p^2} \circ \mathcal D_S$, so by (\ref{al-ga}), 
\begin{equation} \label{commutes-S}
\D_S \circ ({\mathds 1} - Q_S) = ({\mathds 1} - Q_S)\circ \D_S \,\,\,\,\Rightarrow \,\,\,\, \D_S Q_S = Q_S \D_S. 
\end{equation} 
For $\D_T$, on the one hand, we can define 
\[ B_T = ({\mathds 1} - Q_T)^{-1} ={\mathds 1}+ Q_T + Q_T^2 +  \cdots.\]
It is a well-defined operator over $\Lambda^{\K, \Gamma}$ since by (\ref{cont-map}) $\ell(Q_T^k (x))$ diverges to $-\infty$ (as $k \to \infty$) for any $x$. Moreover, by (\ref{commutes-T}), $B_T \D_{T} = \D_T B_T$. On the other hand, for operator $B'_T = (\mathcal D_T)^{p-1} B_T$, we know 
\[ \mathcal D_TB'_T =  (\mathcal D_T)^p B_T = ({\mathds 1} - Q_T) ({\mathds 1} - Q_T)^{-1} = {\mathds 1} =  (\mathcal D_T)^{p-1} B_T \mathcal D_T = B'_T\mathcal D_T. \]
Therefore, $B'_T$ is the required inverse of $\mathcal D_T$. Similarly for $\D_S$, we can define 
\[ B'_S = (\D_S)^{p^2-1} B_S \,\,\,\,\mbox{where}\,\,\,\, B_S = ({\mathds 1} - Q_S)^{-1} = {\mathds 1} + Q_S + Q_S^2 + \cdots. \]
and then $\D_S B'_S = B'_S \D_S$, so $B'_S$ is the desired inverse of $\D_S$.  \end{proof}

\begin{remark} \label{TS-inv} Note that $T$ and $S$ also satisfy perturbed group relations in the form of (\ref{al-ga}), therefore, the same argument in Lemma \ref{inv} implies both $T$ and $S$ are invertible too. \end{remark}

Note that by (\ref{al-ga}), $\D_T$ and $\D_S$ do not generate finite group action. However, from the following lemma, these can always be reduced to be group actions. 

\begin{lemma} \label{cancel-per} There exist $T'$ and $S'$, as strictly lower filtration perturbations of $T$ and $S$ respectively, such that $\D_{T'}$ and $\D_{S'}$ are group actions, that is, 
\[ \D_{T'}^p = {\mathds 1} \,\,\,\,\mbox{and}\,\,\,\,\D_{S'}^p = \D_{T'} \,\,(\mbox{so $\D_{S'}^{p^2}= {\mathds 1}$}). \]
\end{lemma}

\begin{proof} Denote $T^p = {\mathds 1}  - P_T$ where $\ell(P_T x) \leq \ell(x) - \hbar$ for any $x$. We want to find $T'$ such that $(T')^p = (T^p +P_T) = {\mathds 1}$. Since $T$ is invertible by Remark~\ref{TS-inv}, define 
\[ P^{(1)}_T = T^{-p} P_T \,\,\,\,\mbox{and}\,\,\,\, T' = T({\mathds 1} + P^{(1)}_T)^{\frac{1}{p}} \]
where $({\mathds 1} + P^{(1)}_T)^{\frac{1}{p}}$ is defined using the binomial expansion, that is, 
\[ ({\mathds 1} + P^{(1)}_T)^{\frac{1}{p}} = {\mathds 1} + \binom{\frac{1}{p}}{1} P^{(1)}_T + \binom{\frac{1}{p}}{2} (P^{(1)}_T)^2 + \cdots  := {\mathds 1} + P^{(2)}_T \]
where $P^{(2)}_T =\binom{\frac{1}{p}}{1} P^{(1)}_T + \binom{\frac{1}{p}}{2} (P^{(1)}_T)^2 + \cdots  $ and it defines an operator over $\Lambda^{\K, \Gamma}$. Hence, denote $P^{(3)}_T = T P_T^{(2)}$, 
\[ T'= T + P^{(3)}_T \]
which is the required (group action) $T'$ such that it is a strictly lower filtration perturbation of $T$. Moreover, $\D_{T'}^p = \D_{(T')^p} = \D_{{\mathds 1}} = {\mathds 1}$.

Now suppose $T = S^p$. We want to find $S'$ such that $(S')^p = T' = T + P_T^{(3)}$. Again, by Remark \ref{TS-inv}, define 
\[ P^{(1)}_S = S^{-p} P^{(3)}_T \,\,\,\,\mbox{and}\,\,\,\, S' = S ({\mathds 1} + P^{(1)}_S)^{\frac{1}{p}} \]
where $({\mathds 1} + P^{(1)}_S)^{\frac{1}{p}}$ is defined using the binomial expansion, that is 
\[ ({\mathds 1} + P^{(1)}_S)^{\frac{1}{p}} = {\mathds 1} + \binom{\frac{1}{p}}{1} P^{(1)}_S + \binom{\frac{1}{p}}{2} (P^{(1)}_S)^2 + \cdots  := {\mathds 1} + P^{(2)}_S \]
where $P_S^{(2)} = \binom{\frac{1}{p}}{1} P^{(1)}_S + \binom{\frac{1}{p}}{2}(P^{(1)}_S)^2 + \cdots $. Denote $P_S^{(3)} = S P_S^{(2)}$, 
\[ S' = S + P_S^{(3)}\]
which is the required (group action) $S'$ such that it is a strictly lower filtration perturbation of $S$. Moreover, $\D^p_{S'}= \D_{(S')^p} = \D_{T'}$.  
 \end{proof}
 
To simplify the notation of proofs below, denote $[\cdot, \cdot]$ as commutator (of two matrices or operators). So $A$ and $B$ commutes if and only if $[A, B]=0$. An important observation from that definition of $T$ is that $[T, P_T] =[P_T, T] = 0$. Then we have the following corollary. 

\begin{cor} \label{b-commutes} For $\D_{T'}$ and $\D_{S'}$ constructed from Lemma \ref{cancel-per}, we have $[\D_{T'}, \partial_{co}] =0$ and  $[\D_{S'}, \partial_{co}] =0$. \end{cor}

\begin{proof} First, we claim $[T', \partial]=0$ and $[S', \partial]=0$. In fact, starting from $[T, \partial]=0$ and the definitions of $P^{(1)}_T$, $P^{(2)}_T$ and $P^{(3)}_T$, we have 
\begin{align*}
 [T, \partial] =0 \,\Rightarrow \,[P_T, \partial] =0 \, \Rightarrow\, [P^{(1)}_T, \partial] =0 \, &\Rightarrow \,[P^{(2)}_T, \partial]=0 \\
 & \Rightarrow \,[P^{(3)}_T, \partial]=0 \,\Rightarrow \,[T', \partial]=0.
 \end{align*}
Similarly, for $S'$, starting from $[S, \partial] =0$, $[P^{(3)}_T, \partial]=0$ and definitions of $P^{(1)}_S$, $P^{(2)}_S$ and $P^{(3)}_S$, we have 
\begin{align*}
 ([S, \partial] =0], [P^{(3)}_T, \partial]=0]) \, \Rightarrow \,  [P^{(1)}_S, \partial] =0 \, &\Rightarrow \,[P^{(2)}_S, \partial]=0 \\
 & \Rightarrow \, [P^{(3)}_S, \partial]=0 \,\Rightarrow \,[S', \partial]=0.
 \end{align*}
Second, we claim that $[T',T]=0$ and $[S',T]=0$. In fact, starting from $[P_T, T] =0$, we have 
\[ [P_T, T]=0 \, \Rightarrow \, [P^{(1)}_T, T]= 0 \, \Rightarrow \, [P^{(2)}_T, T] \, \Rightarrow \, [P^{(3)}_T, T] =0\ \Rightarrow \, [T', T]=0. \]
Similarly for $S'$, starting from $[S, T]=0$ and $[P^{(3)}_T, T] =0$, we have 
\begin{align*}
 ([S, T]=0, [P^{(3)}_T, T] =0) \,\Rightarrow\, [P^{(1)}_S, T]= 0 \,&\Rightarrow \,[P^{(2)}_S, T] \\
 &\Rightarrow \, [P^{(3)}_S, T] =0\ \Rightarrow \, [S', T]=0.
 \end{align*}
Third, we conclude that $[\D_{T'}, \partial_{co}] =0$ and $[\D_{S'}, \partial_{co}] =0$. In fact, 
\begin{align*}
\left( \begin{array}{cc} 
T' & 0  \\
0 & T'
\end{array} \right) \cdot  \left( \begin{array}{cc} 
\partial & - (T - \xi_p \cdot {\mathds 1}) \\
0 & - \partial 
\end{array} \right) &= \left( \begin{array}{cc} 
T' \partial & -T'(T-\xi_p\cdot {\mathds 1}) \\
0 & - T' \partial 
\end{array} \right) \\
& = \left( \begin{array}{cc} 
\partial T' & - (T- \xi_p \cdot {\mathds 1}) T' \\
0 & -\partial T'
\end{array} \right) \\
& = \left( \begin{array}{cc} 
\partial & - (T - \xi_p \cdot {\mathds 1}) \\
0 & - \partial 
\end{array} \right) \cdot \left( \begin{array}{cc} 
T' & 0  \\
0 & T'
\end{array} \right)
\end{align*}
where the second equality comes from the first part of two claims above. Similarly, 
\begin{align*}
\left( \begin{array}{cc} 
S' & 0  \\
0 & S'
\end{array} \right) \cdot  \left( \begin{array}{cc} 
\partial & - (T - \xi_p \cdot {\mathds 1}) \\
0 & - \partial 
\end{array} \right) &= \left( \begin{array}{cc} 
S' \partial & -S'(T-\xi_p\cdot {\mathds 1}) \\
0 & - S' \partial  
\end{array} \right) \\
& = \left( \begin{array}{cc} 
\partial S' & - (T- \xi_p \cdot {\mathds 1}) S' \\
0 & -\partial S'
\end{array} \right) \\
& = \left( \begin{array}{cc} 
\partial & - (T - \xi_p \cdot {\mathds 1}) \\
0 & - \partial 
\end{array} \right) \cdot \left( \begin{array}{cc} 
S' & 0  \\
0 & S'
\end{array} \right)
\end{align*}
where the second equality comes from the second part of two claims above. 
\end{proof}

\begin{remark} \label{per-rek} Since $\D_{T}$ is a strictly lower filtration perturbation of $\mathcal D_{R_p}$, so is $\D_{T'}$ by its construction. Similarly, $\D_{S'}$ is a strictly lower filtration perturbation of $\D_{R_{p^2}}$. For orthogonality, strictly lower filtration perturbation behaves well as proved by the following easy lemma which will be used frequently later. \end{remark}

\begin{lemma} \label{or-per}
Strictly lower filtration perturbation preserves orthogonality. Specifically, given a set of orthogonal elements over $\Lambda^{\K, \Gamma}$, say $\{v_1, \dots , v_n\}$ and any strictly lower filtration for each $v_i$, that is $v_i + w_i$ where $\ell(w_i) < \ell(v_i)$, for $1 \leq i \leq n$, 
\[ \{v_1 + w_1, \dots , v_n + w_n\} \,\,\,\mbox{are orthogonal over $\Lambda^{\K, \Gamma}$}. \]
\end{lemma}

\begin{proof} For any $\lambda_1, \dots , \lambda_n \in \Lambda^{\K, \Gamma}$, by Proposition 2.3 in \cite{UZ15}, 
\begin{align*} 
\ell(\lambda_1 (v_1+w_1) + \cdots  \lambda_n (v_n + w_n)) &= \ell(\lambda_1v_1 + \cdots  \lambda_n v_n) \\
&= \max_{1 \leq i \leq n} \{\ell(\lambda_i v_i)\} =  \max_{1 \leq i \leq n} \{\ell(\lambda_i (v_i+w_i))\}.
\end{align*}
So $\{v_1 + w_1, \dots, v_n + w_n\}$ are also orthogonal. \end{proof} 

\subsection{Preparation}

\subsubsection{Orthogonal invariant complement} 

\begin{prop} \label{inv-comp} Suppose that $V$ is acted by a group action $T$ such that $T^p = {\mathds 1}$ and it exactly preserves filtration. For any $T$-invariant subspace $V_1$, there exists an {\it orthogonal} complementary $T$-invariant $W$ in the sense that $V_1$ is orthogonal to $W$ and $V = V_1 \oplus W$. \end{prop}

Recall that if vector space $V$ is a representation of a finite group $G$ and the field of scalars for $V$ satisfies the following condition \footnote{See Theorem 4.1 and it equivalent conclusion like Lemma 2.2.11 in \cite{Kow13}. Non-division between characteristic of working field and order of group is the only hypothesis. When we apply to Novikov field $\Lambda^{\mathcal K, \Gamma}$ and group $G$ with order $p$ here, $char(\mathcal K) =0$ implies $char(\Lambda^{\mathcal K, \Gamma})=0$. Therefore it satisfies this hypothesis. }
\begin{equation} \label{char}
char(\mathcal K) \nmid |G|,
\end{equation}
then given any $G$-invariant subspace $V_1 \leq V$, there exists a $G$-invariant complementary subspace $W \leq V$. Actually we can construct $W$ explicitly. 

\begin{const} \label{mskc} Taking any complement of $V_1$ (in the sense of vector space, no orthogonality is involved and not necessarily $G$-invariant), say $U$, consider the projection map $\pi_{V_1}: V \to V_1$ with respect to the decomposition $V= V_1 \oplus U$, and define
\begin{equation} \label{cons-rep-comp} 
W = \ker \left( \frac{1}{|G|} \sum_{g \in G} g \cdot \pi_{V_1}  \cdot g^{-1} \right).
\end{equation} 
It is easy to check that $W$ is $G$-invariant for any $g \in G$ and $\dim W = \dim U$. There is a useful observation: (\ref{cons-rep-comp}) implies that each $x \in W$ satisfies  
\[\frac{1}{|G|} \sum_{g \in G} g ({\mathds 1}-\pi_U) g^{-1}(x) = 0. \]
Therefore, we have 
\begin{equation} \label{mas-re}
x = \frac{1}{|G|} \sum_{g \in G} g \cdot \pi_U \cdot g^{-1}(x) \,\,\,\,\,\,\mbox{for all $x \in W$.}
\end{equation} 
\end{const} 
In order to get Proposition \ref{inv-comp}, we need to show that $W$ is orthogonal to $V_1$. We will use the following result which is Lemma 7.5 in \cite{UZ15}, 

\begin{lemma} \label{or&pr} Let $(V, \ell)$ be an orthogonalizable $\Lambda^{\mathcal K, \Gamma}$-space and let $V_1, U, W \leq V$ be such that $U$ is an orthogonal complement to $V_1$ and $\dim(U) = \dim(W)$. Consider the projection $\pi_U: V \to U$ associated to the direct sum decomposition $V = U \oplus V_1$. Then $W$ is an orthogonal complement of $V_1$ if and only if $\ell(\pi_U x) = \ell(x)$ for all $x \in W$. \end{lemma}

\begin{proof} [Proof of Proposition \ref{inv-comp}] Let $U$ be the orthogonal complement of $V_1$ that is used in the Construction \ref{mskc}. By Lemma \ref{or&pr}, we will need to show that for any $x \in W$, $\ell(x) = \ell(\pi_U x)$. Decompose $x = u + v$ where $u \in U$ and $v \in V_1$. By (\ref{mas-re}), we have 
\begin{align*}
x &= \frac{1}{p} \left(\pi_U x + \sum_{i=1}^{p-1} T^i \pi_U T^{p-i} x \right) \\
& = \frac{1}{p} \left(u + \sum_{i=1}^{p-1} T^i \pi_U T^{p-i}u \right). 
\end{align*} 
The key step for the second line is that since $V_{1}$ is invariant, $\pi_U T^{p-i}(u+v) = \pi_U T^{p-i}u + \pi_U T^{p-i} v = \pi_U T^{p-i}u$ since $\pi_U T^{p-i} v =0$. So we can write $v$ in terms of $u$, that is 
\[ v = - \frac{p -1}{p} u + \frac{1}{p} \sum_{i=1}^{p-1} T^i \pi_U T^{p-i}u. \]
Because $T$ exactly preserves filtration and $U$ is orthogonal to $V_1$, it follows that $\ell(v) \leq \ell(u)$. Therefore, 
\[ \ell(x) =\max\{\ell(u), \ell(v)\} = \ell(u) = \ell(\pi_U x).\]
Therefore, we get the conclusion.
\end{proof}

\begin{ex} \label{o-i-ker} For group action (guaranteed by Lemma \ref{cancel-per}) 
$$\D_{T'}: Cone_{k}(T - \xi \cdot {\mathds 1}) \to Cone_{k}(T - \xi\cdot {\mathds 1}),$$
since $\ker(\partial_{co})$ is a $\D_{T'}$-invariant subspace of $Cone_k(T- \xi \cdot {\mathds 1})$ by Corollary~\ref{b-commutes}, there exists an orthogonal complement of $\ker( \partial_{co})$ in $Cone_k(T- \xi \cdot {\mathds 1})$, denoted by  $W$ which is also $\D_{T'}$-invariant. \end{ex}

\subsubsection{Restriction to 0-level}

In this subsection, we will work on the vector space over universal Novikov field, that is, for $\Lambda^{\K, \Gamma}$, $\Gamma = \R$. The advantage is that we can always rescale preferred element $x$ in the vector space such that $\ell(x) =0$. Moreover, since we have seen that our obstruction (see Definition \ref{dfn-o}) will be constructed only from generalized boundary depth (which does not involve the specific value of end points), By Proposition 6.8 in \cite{UZ15}, it will be invariant under the coefficient extension. 

From the idea of \cite{Ush08}, any orthogonalizable $\Lambda^{\K, \R}$-space $(V, \ell)$ can be identified with $((\Lambda^{\K, \R})^n, -\vec{\nu})$ (for some $n = \dim_{\Lambda^{\K, \R}} V \in \N$) under an orthonormal basis, where $\ell(v) = - \vec{\nu}(\lambda_1, \dots , \lambda_n)$ if $v$ is identified with a vector $(\lambda_1, \dots , \lambda_n)$ under this basis. Therefore, for such $V$, we can associated a $\K$-vector space 
\[ [V] = V_{\leq 0}/{V_{<0}} \]
where 
\[ V_{\leq 0}: = \{ v \in V \,| \,\ell(v) \leq 0 \} \,\,\,\,\mbox{and}\,\,\,\, V_{<0} :=\{ v \in V\,|\, \ell(v) <0 \}. \]
In particular, denote 
$$\Lambda^{\K, \R_{\geq0}} = \{ \lambda \in \Lambda^{\K, \R} \,| \, \nu(\lambda) \geq 0\}\,\,\,\,\mbox{and}\,\,\,\, \Lambda^{\K, \R_{>0}} = \{ \lambda \in \Lambda^{\K, \R} \,| \, \nu(\lambda) > 0\}.$$
Note that $\K \simeq \Lambda^{\K, \R_{\geq0}}/\Lambda^{\K, \R_{>0}}$. There is a quotient projection $\pi: V_{\leq 0} \to [V]$ by taking only the filtration level-0 term. Explicitly, for $\lambda \in \Lambda^{\K, \R}$, 
\begin{align} \label{proj}
\lambda = \sum_{g} a_g T^g & \xrightarrow{\pi} a_0
\end{align}
Then we can show the following result. 
\begin{lemma}\label{li-or}
A $\Lambda^{\K, \R}$-orthonormal set $\{e_1, \dots , e_n\}$ reduces to a $\K$-linearly independent set $\{[e_1], \dots , [e_n]\}$ under projection $\pi$ defined in (\ref{proj}). Conversely, for a set $\{e_1, \dots , e_n\}$ over $\Lambda^{\K, \R}$ with $\ell(e_1) = \cdots  = \ell(e_n) = 0$, if its reduction $\{[e_1], \dots , [e_n]\}$ are $\K$-linearly independent, then $\{e_1, \dots , e_n\}$ are $\Lambda^{\K, \R}$-orthogonal. Therefore, in particular, $[V] = (\mathcal K)^{ \dim_{\Lambda^{\K, \R}} V}$. \end{lemma}
\begin{proof} Suppose that $\{[e_1], \dots , [e_n]\}$ are not $\K$-linearly independent. There exists $\eta_1, \dots , \eta_n \in \K$, not all zero, such that 
\[ \eta_1 [e_1] + \cdots  + \eta_n [e_n] =0. \]
Then
\[ \eta_1 e_1 + \cdots  + \eta_n e_n = \eta_1 [e_1] + \cdots  + \eta_n [e_n] + \left\{\begin{array}{cc} \mbox{strictly lower} \\ \mbox{filtration terms} \end{array}\right\}.\]
So
\[ \ell(\eta_1 e_1 + \cdots  + \eta_n e_n ) < 0 = \max_{1 \leq i \leq n} \{\ell(e_i) - \nu(\eta_i)\}. \]
Therefore, $\{e_1, \dots , e_m\}$ are not $\Lambda^{\K, \R}$-orthogonal.

Conversely, suppose that $\{e_1, \dots , e_n\}$ are not $\Lambda^{\K, \R}$-orthogonal. There exist $\lambda_1,\dots,  \lambda_n \in \Lambda^{\K, \R}$, not all zero, such that 
\[ \ell(\lambda_1 e_1 + \cdots  + \lambda_n e_n) < \max_{1 \leq i \leq n} \{\ell(e_i) - \nu(\lambda_i)\} = 0 - \min_{1 \leq i \leq n} \{\nu(\lambda_i)\}. \]
If we rescale $\lambda_i$ on both sides such that $\min_{1 \leq i \leq n} \{\nu(\lambda_i)\} =0$, then still we have the inequality. However, reducing to the filtration level-0, which is the highest filtration level, we have
\[ [\lambda_1] [e_1] + \cdots  [\lambda_n] [e_n] =0 \,\,\,\mbox{where}\,\,\, [\lambda_i] \,\,\mbox{in}\, \,\K.\]
Due to our rescaling, not all $[\lambda_i]$ are zero, which means $\{[e_1], \dots , [e_n]\}$ are not $\K$-linearly independent.

In particular, if $\{e_1, \dots , e_n\}$ is an $\Lambda^{\K, \R}$-orthonormal basis of $V$, then $\dim_{\K}[V] \geq n$. Meanwhile, $[V]$ is a submodule of $[\Lambda^{\K, \R}]^n = \K^n$. So $\dim_{\K}[V] = n = \dim_{\Lambda^{\K, \R}} V$. \end{proof}

Not only can we reduce spaces, but also we can reduce maps. For a $\Lambda^{\K, \R}$-linear map $A$ on $(V, \ell)$ which is exact filtration preserving (for instance, $\D_{T'}$ or $\D_{S'}$), under an orthonormal basis, $A \in M_{n\times n}(\Lambda^{\K, \R_{\geq0}})$. Note that then $A(V_{<0}) \leq V_{<0}$ which implies that we have a well-defined reduced map of $A$, denoted by  $[A]$,
\[ [A]: [V] \to [V]. \]

\begin{ex} Suppose under an orthonormal basis,
\[ A = \left( \begin{array}{ccc} 
1 +t^2 & t^6 & t^2 + t^4 \\
t^4 & 2 & t^6 - t^{10} \\
2 & t^2 & 5 + t^2  
\end{array} \right). \]
Then
\[ [A] =  \left( \begin{array}{ccc} 
1  & 0 & 0 \\
0 & 2 & 0 \\
2 & 0 & 5   
\end{array} \right). \]
\end{ex}

\begin{ex} \label{reduce-eigen} Suppose $A$ is exact filtration preserving and $\lambda \in \Lambda^{\K, \R}$, 
\[ A^n = \lambda \cdot {\mathds 1} \xrightarrow{\mbox{reduces to}} [A]^n = [\lambda] \cdot {\mathds 1}.\]
In particular, for any $x \in V_{\leq 0}$, $A x = \lambda x$ reduces to $[A][x] = [\lambda] [x]$. \end{ex}

\subsubsection{Irreducible condition}

\begin{lemma} \label{irrd-N} $\mathcal K$ satisfies ``irreducible condition'' if and only if $\Lambda^{\mathcal K, \Gamma}$ satisfies ``irreducible condition''. \end{lemma}
\begin{proof} 
The direction ``$\Leftarrow$" is trivial since $\K \hookrightarrow \Lambda^{\K, \Gamma}$. We will just prove the other direction ``$\Rightarrow$". Suppose not. Then there exists some $x \in \Lambda^{\mathcal K, \Gamma}$ and number $q$ such that 
\begin{equation}\label{no-sol}
x^p  = \xi_p^q 
\end{equation}
but $p$ does not divide $q$. The general form of $x$ is $x = a_m t^{\lambda_m} + a_{m+1} t^{\lambda_{m+1}} + \cdots$
where $a_m \neq 0$ and $\lambda_m < \lambda_{m+1} < \cdots$ (which diverges to infinity). If $\lambda_m \neq 0$, then 
\[ x^p = a_m^p t^{p \lambda_m } + \cdots \]
such that the lowest degree $p \lambda_m$ is either strictly positive or strictly negative, which by equation (\ref{no-sol}) above, forces $a_m^p=0$. So $a_m =0$. Contradiction. Now we are left to the case that $\lambda_m =0$, so we may rewrite $x$ as $x = a_m + a_{m+1} t^{\lambda_1} + \cdots$ where $a_m \in \mathcal K$. Therefore, (\ref{no-sol}) implies $a_m^p = \xi_p^q$ which contradicts the hypothesis that $\mathcal K$ satisfies ``irreducible condition". 
\end{proof}

From now on, we will always assume that $\K$ satisfies ``irreducible condition" and $\Gamma = \R$. For brevity, we introduce the following notation.

\begin{dfn} For any $x \in V$ and an operator $A$ such that $A^p = \lambda \cdot {\mathds 1}$ for some non-zero scalar $\lambda \in \Lambda^{\K, \Gamma}$, denote 
\[ V_x = {\Span}_{\Lambda^{\K, \Gamma}} \left<x, Ax, \dots , A^{p-1} x \right> \]
and call it {\it the cyclic span (of $A$) by $x$}. \end{dfn}

Note that by Lemma \ref{irrd-N} and Lemma 4.15 in \cite{PS14}, if $\lambda = \xi_p^q$ for $1 \leq q \leq p-1$, we know that $\{x, Ax, \dots , A^{p-1}x\}$ are $\Lambda^{\K, \Gamma}$-linearly independent, so $\dim_{\Lambda^{\K, \Gamma}} V_x = p$. However, we can get a stronger result as follows. 

\begin{cor} \label{cyc-ele} Let $V$ be a $\Lambda^{\K, \Gamma}$-vector space associated with an $\Lambda^{\K, \Gamma}$-linear exact filtration preserving operator $A$ such that 
	\[ A^p = \xi_p^q \cdot {\mathds 1}\,\,\,\,\mbox{for } 1 \leq q \leq p-1.\]
Then for any $x \in V$, $\{x, Ax, \dots , A^{p-1}x\}$ are $\Lambda^{\K, \Gamma}$-orthogonal. \end{cor}

\begin{proof} Rescale $x$ such that $\ell(x) = 0$ if necessary. The cyclic span $V_x$ of $A$ is an $A$-invariant subspace. Then, $[V_x]$ is an $[A]$-invariant subspace where 
\[ [V_x] = {\Span}_{\K} \left<[x], [A] [x], \dots , [A]^{p-1}[x] \right> \]
is a cyclic span of $[x]$ since $[A]^p = \xi_p^q \cdot {\mathds 1}$. Because $\K$ satisfies ``irreducible condition", we know that $p \,| \, \dim_{\K} [V_x]$. But $\dim_{\K} [V_x] \leq p$. The rigidity $\dim_{\K} [V_x] = p$ implies $\{[x], [A] [x], \dots , [A]^{p-1}[x]\}$ are $\K$-linearly independent. By Lemma \ref{li-or}, $\{x, Ax, \dots , A^{p-1}x\}$ are $\Lambda^{\K, \Gamma}$-orthogonal. \end{proof}

\begin{remark} \label{fail-1}
Note that in the situation of Corollary \ref{cyc-ele}, if $A$ satisfies $A^p = {\mathds 1}$, then we can't conclude that $\{x, Ax, \dots , A^{p-1}x\}$ are $\Lambda^{\K, \Gamma}$-orthogonal (even $\Lambda^{\K, \Gamma}$-linearly independent) directly because ``irreducible condition'' does not apply here. To get the expected result on the multiplicity of $p$, more structure of self-mapping cone will be used later. \end{remark}

\subsubsection{Filtration optimal pair}

The following lemma is the key to construct the desired singular value decomposition which will be used later.\footnote{This lemma is exactly the same as Lemma 3.5 in an early version of \cite{UZ15}. For its submitted version, this lemma has been deleted for brevity. For reader's convenience, we add/repeat it here.}

\begin{lemma}\label{opt-ele} Let $(V_1, \ell_1)$ and $(V_2, \ell_2)$ be orthogonalizable $\Lambda^{\mathcal K, \Gamma}$-vector spaces and let $A: V_1 \rightarrow V_2$ be any nonzero $\Lambda^{\K, \Gamma}$-linear map. Then there exists some $y_* \in V_1\setminus\{0\}$ such that, for all $y \in V_1\setminus\{0\}$, 
\begin{equation}\label{y0cond}
\ell_2(Ay_*) - \ell_1(y_*) \geq \ell_2(Ay) - \ell_1(y).\end{equation}  \end{lemma}

\begin{proof} Because $\ker A$ is a subspace of $V_1$, by Corollary 2.17 and 2.18 in \cite{UZ15}, there exists an orthogonal basis of $V_1$, say $(v_1,\dots, v_r, v_{r+1},\dots, v_n)$ such that $(v_{r+1}, \dots , v_n)$ is an orthogonal basis for $\ker A$ where $r = rank(A)$. Meanwhile, let $(w_1, \dots , w_m)$ be an orthogonal basis for $V_2$. Represent $A$ by an $m \times n$ matrix $\{A_{ij}\}$ over $\Lambda^{\K, \Gamma}$ with respect to these two bases, so that $A v_j = {\textstyle \sum_{1 \leq i \leq m} A_{ij} w_i}$ for each $j \in \{1, \dots , n\}$. For any $y = {\textstyle \sum_{1 \leq j \leq n} \lambda_j v_j}$ in $V$, we have 
\begin{align*}
	\ell_2(Ay) &= \ell_2 \left( \sum_{1\leq i \leq m} \left( \sum_{1 \leq j \leq n} A_{ij} \lambda_j\right) w_i \right)\\
	&= \max_{1 \leq i \leq m} \left( \ell_2(w_i) - \nu\left( \sum_{1 \leq j \leq n} A_{ij}\lambda_j \right) \right) \\
&= \ell_2(w_{i(y)}) - \nu \left( \sum_{1 \leq j \leq n} A_{i(y) j} \lambda_j \right)
\end{align*}
where $i(y) \in \{1, \dots , m\}$ is the index  attaining the maximum in the middle. By the definition of the valuation $\nu$,
\begin{align*}
	-\nu\left( \sum_{1 \leq j \leq n} A_{i(y) j} \lambda_j \right) &\leq \max_{1 \leq j \leq n} \left(-\nu(A_{i(y) j}) - \nu(\lambda_j)\right)\\
	&= - \nu (A_{i(y)j(y)}) - \nu(\lambda_{j(y)})
\end{align*} 
where, again, $j(y) \in \{1, \dots , n\}$ is the index  attaining the maximum in the middle. Also due to the orthogonality of $(v_1, \dots , v_n)$,
\[ \ell_1(y)= \max_{1 \leq j \leq n} (\ell_1(v_j) - \nu(\lambda_j) ) \geq \ell_1(v_{j(y)}) - \nu(\lambda_{j(y)}),\]
so 
\begin{align}\label{opt3}
\ell_2(Ay) - \ell_1(y) &\leq \left(\ell_2(w_{i(y)})- \nu (A_{i(y)j(y)}) - \nu(\lambda_{j(y)})\right) \\
	\nonumber &\quad -  \left(\ell_1(v_{j(y)}) - \nu(\lambda_{j(y)})\right) \\ \nonumber
&= \ell_2(w_{i(y)}) - \nu (A_{i(y)j(y)}) - \ell_1(v_{j(y)}).
\end{align}
Now choose $(i_0, j_0)$ among $(i,j) \in \{1, \dots ,m\} \times \{1, \dots , n\}$ so that 
\begin{equation} \label{optchoice}
\ell_2(w_{i_0}) - \nu(A_{i_0j_0}) - \ell_1(v_{j_0}) \geq \ell_2(w_{i}) - \nu(A_{ij}) - \ell_1(v_{j})
\end{equation} for all $i$ and $j$.  Then due to the orthogonality of $(w_1, \dots , w_m)$, $\ell_2(A v_{j_0}) = \max_{1 \leq i \leq n} (\ell_2(w_i) - \nu(A_{ij_0}))$. So using (\ref{optchoice}), we have 
\begin{align*}
	\ell_2(A v_{j_0}) - \ell_1(v_{j_0}) &= \max_{1 \leq i \leq m} \left( \ell_2(w_{i}) - \nu(A_{ij_0}) - \ell_1(v_{j_0})\right) \\
	&= \ell_2(w_{i_0}) - \nu(A_{i_0j_0}) - \ell_1(v_{j_0})
\end{align*}

Given any $y\in C$, (\ref{optchoice}) holds for $i=i(y), j=j(y)$, so using (\ref{opt3}) we get $\ell_2(Av_{j_0})-\ell_1(v_{j_0})\geq \ell_2(Ay)-\ell_1(y)$.
Therefore $y_0 = v_{j_0}$ obeys the desired optimality property. \end{proof}

\subsection{$p$-cyclic singular value decomposition} \label{sub-p-c-svd}
The idea of proving our main result, Theorem \ref{p-tuples}, goes as follows. 
\[ \xymatrix{
\boxed{\begin{array}{cc} \mbox{eigenspace decomposition} \\ \mbox{of action $\D_{T'}$} \end{array}} \ar[r]^-{(a)} \ar[d]^{(c)}& \boxed{\begin{array}{cc} \mbox{singular value decomposition} \\ \mbox{of $\partial_{co}$ compatible with $\D_{T'}$} \end{array}} \ar[d]^{(b)} \\
\boxed{\begin{array}{cc} \mbox{special care of} \\ \mbox{eigenvalue 1} \end{array} } & \boxed{\begin{array}{cc} \mbox{singular value decomposition} \\\mbox{of $\partial_{co}$ compatible with $\D_{S'}$} \end{array}} }\]

First, since $\D_{T'}^p = {\mathds 1}$, by Example \ref{o-i-ker}, there exists a $\D_{T'}$-invariant orthogonal complement of $\ker(\partial_{co}) = \ker((\partial_{co})_{k+1}) = {\rm Im}((\partial_{co})_{k+2})$ in $Cone_{k+1}(T - \xi_p \cdot {\mathds 1})$ (recall that, by our assumption on homotopy class $\alpha$, homology of the mapping cone is zero), denoted by  $W$. Let's start from $(a)$. Since $\D^p_{T'} = {\mathds 1}$, eigenvalues of $\D_{T'}$ are among $\{1, \xi_p, \dots , \xi_p^{p-1}\}$. By assumption, our working field $\Lambda^{\K, \Gamma}$ contains all the $p$-th root of unity, then $\D_{T'}$ is diagonalizable. Therefore, we have the following eigenspace decompositions (for the part that we need to compute degree-$k$ barcode)
\begin{equation} \label{eign-dec2} 
W = F_0 \oplus F_1 \oplus \cdots \oplus F_{p-1} \,\,\,\,\, \mbox{and}\,\,\,\,\, {\rm{Im}}(\partial_{co}) = G_0 \oplus G_1 \oplus \cdots \oplus G_{p-1}
\end{equation}
where $F_i$ and $G_i$ are eigenspaces of $\D_{T'}$ of eigenvalue $\xi_p^{i}$ for each $i \in \{0, \dots , p-1\}$ (in their corresponding $\D_{T'}$-invariant subspaces). Note that it is possible that some of them are trivial. Study the algebraic relation between different eigenspaces is crucial. From linear algebra, different $F_i$'s  (so are $G_i$'s) are $\Lambda^{\K, \Gamma}$-linearly independent over $\Lambda^{\K, \Gamma}$. Now we will show they are actually mutually orthogonal to each other. This is the following lemma.

\begin{lemma} \label{orth-eigenspace} $\{F_i\}_{i=1}^{p-1}$ are mutually orthogonal to each other, so are $\{G_i\}_{i=1}^{p-1}$.\end{lemma}

\begin{proof} Assume that all the basis elements have filtration $0$. $\mathcal D_{T'}$ acting on $F_i$ implies $[(\mathcal D_{T'})] = \mathcal D_{R_p}$ acting on $[F_i]$. Moreover, by Example \ref{reduce-eigen}, eigenspace $F_i$ (of eigenvalue $\xi_p^i$) reduces to $\K$-space $[F_i]$ which is a subspace of eigenspace (of eigenvalue $[\xi_p^i] = \xi_p^i$) of $[\ker(\partial_{co})]$. Because different eigenspaces $[F_i]$ are $\K$-linearly independent, by Lemma \ref{li-or}, $F_i$ is $\Lambda^{\K, \Gamma}$-orthogonal to $F_j$ for all $i \neq j$. A similar argument holds for $\{G_i\}_{i=0}^{p-1}$. \end{proof}

On the other hand, $W$ is isomorphic to ${\rm{Im}}(\partial_{co})$, so in particular, 
\begin{equation} \label{dim-equ}
\dim_{\Lambda^{\K, \Gamma}}(W) = \dim_{\Lambda^{\K, \Gamma}}({\rm{Im}}(\partial_{co})).
\end{equation}
By Corollary \ref{b-commutes}, $\partial_{co}$ and $\D_{T'}$ commutes, which gives rise to an important observation that $\partial_{co}$ brings a $\D_{T'}$-eigenspace into a $\D_{T'}$-eigenspace of the same eigenspace.  In particular, we have $\partial_{co}(F_i) \leq G_i$, so $\dim_{\Lambda^{\K, \Gamma}}(F_i) \leq \dim_{\Lambda^{\K, \Gamma}}(G_i)$. (\ref{dim-equ}) implies the dimensions are actually equal for each $i$. So restrictions 
\[ \partial_{co}|_{F_i}: F_i \to G_i, \]
are isomorphisms between two (smaller) filtered  $\Lambda^{\K, \Gamma}$-vector spaces. Then

\begin{prop} \label{SVD-T}
For each degree $k\in \Z$, there exists a singular value decomposition of $\partial_{co} = (\partial_{co})_{k+1}: Cone_{k+1}(T- \xi_p \cdot {\mathds 1}) \to {\rm{Im}}(\partial_{co})$ such that elements from this singular value decomposition span $\D_{T'}$-eigenspaces of eigenvalues $\xi_p^q$ for $0 \leq q \leq p-1$. 
\end{prop} 
\begin{proof} Because $\ker((\partial_{co})_{k+1}) = {\rm Im}((\partial_{co})_{k+2})$, Lemma \ref{orth-eigenspace} says that the decomposition of corresponding $G_i$'s for ${\rm Im}((\partial_{co})_{k+2})$ provides an orthogonal basis for $\ker(\partial_{co})$ and they form a singular value decomposition of $\partial_{co}|_{\ker(\partial_{co})}$. Meanwhile, by Theorem 3.5 in \cite{UZ15}, there exists a singular value decomposition of each $\partial_{co}|_{F_i}$. Together all $0 \leq i \leq p-1$, they will form a singular value decomposition of $\partial_{co}|_W$ again by Lemma \ref{orth-eigenspace}. \end{proof}

Now for (b), suppose $\D_{S'}$ is defined. By its construction (see Lemma \ref{cancel-per}), $\D_{S'}$ acts on each eigenspace of $\D_{T'}$, and we will work on $\partial_{co}|_{F_i}: F_i \to G_i$ piece by piece for each $1 \leq i \leq p-1$. Explicitly, we will formulate a theoretic process (in accordance with the algorithmic process, Theorem 3.5 in \cite{UZ15} of finding a singular value decomposition) which is the following lemma. The idea of generating a singular value decomposition is by orthogonally cutting down $p$-dimensional subspaces. In general, we have the following result. 

\begin{lemma} \label{cut-p} Let $(V_1, \ell_1)$ and $(V_2, \ell_2)$ be orthogonalizable $\Lambda^{\mathcal K, \Gamma}$-vector spaces with group action $S$ on both spaces and $S^p = \xi_p^q \cdot \I$ for $1 \leq q \leq p-1$. For any non-zero $\Lambda^{\K, \Gamma}$-linear map $A: (V_1, \ell_1) \to (V_2,  \ell_2)$ such that $A \circ S = S \circ A$, there exists a pair of dimension $p$ subspaces in the cyclic span form of $(V_y, V_x)$ where $A(V_y) = V_x$ such that there exists an orthogonal complement pair $(W_1, W_2)$ in the sense that 
\begin{itemize}
\item[(1)] $A(W_1) = W_2$;
\item[(2)] $V_1 = V_y \oplus W_1$ and $W_1$ is orthogonal to $V_y$;
\item[(3)] $V_2 =V_x \oplus W_2$ and $W_2$ is orthogonal to $V_x$. 
\end{itemize}
\end{lemma} 

\begin{proof} By Lemma \ref{opt-ele}, there exists an optimal pair $(x^*,y^*)$ where $A(y^*) = x^*$ such that for all $y \in V_1$, 
\begin{equation} \label{opt}
 \ell_1(y^*) - \ell_2(x^*) \leq \ell_1(y) - \ell_2(Ay).
\end{equation} 
By our assumption, group action $S$ exactly preserves filtrations. Therefore, 
\begin{equation} \label{opt2}
 \ell_1(S^iy^*) - \ell_2(S^ix^*) \leq \ell_1(y) - \ell_2(Ay)
\end{equation}
for all the $i =0, \dots , p-1$. By Corollary \ref{cyc-ele} we know elements from $\{x^*, \dots , S^{p-1}x^*\}$ are orthogonal, so are elements from $\{y^*, \dots , S^{p-1}y^*\}$. Consider the span $V_{x^*}$ and denote its orthogonal complement as $W_2$ and its preimage under $A$ as $W_1 = A^{-1}(W_2)$. Now we only need to show proposition (2) above. We will show this inductively (with finitely many steps). By Lemma 2.9 in \cite{UZ15}, 
\[ W_1 \perp \left<y^*\right> \,\,\mbox{and}\,\, \left<Sy^*\right> \perp W_1 \oplus \left<y^*\right> \,\,\Rightarrow \,\, W_1 \perp \left<y^*, Sy^*\right>. \]
Then 
\[ W_1 \perp \left<y^*, Sy^*\right> \,\,\mbox{and}\,\, \left<S^2y^*\right> \perp W_1 \oplus \left<y^*, Sy^*\right>\,\,\Rightarrow \,\, W_1 \perp \left<y^*, Sy^*, S^2y^*\right>.\]
Inductively, we will get the conclusion. 

Indeed, in order to prove $W_1 \perp \left<y^*\right>$, which is equivalent to the statement that for any $v \in W_1$, $\ell_1(y^*) \leq \ell_1(y^* - v)$, we note since $W_2$ is orthogonal to $\left<x^*\right>$, $\ell_2(A(y^*-v)) \geq \ell_2(x^*)$, so by optimal choice (\ref{opt}) we get the conclusion. Now let us prove $\left<Sy^*\right> \perp W_1 \oplus \left<y^*\right>$. The proof is similar. Again, this is equivalent to show for any $v \in W_1 \oplus \left<y^*\right>$, $\ell_1(Sy^*) \leq \ell_1(Sy^* - v)$. Because $W_2 \oplus \left<x^*\right>$ is orthogonal to $\left<Sx^*\right>$, $\ell_2(A(Sy^* - v)) \geq \ell_2(Sx^*)$. Then by optimal choice (\ref{opt2}), we get the conclusion. in general, the same argument works for the proof of $\left<S^{i+1} y^*\right> \perp W_1 \oplus \left<y^*, \dots , S^iy^* \right>$. \end{proof}

\begin{remark} \label{relax-cond} Note that the condition $S^p = \xi_p^q \cdot {\mathds 1}$ of Lemma \ref{cut-p} is only used to formulate a cyclic span $V_{x^*}$. \end{remark} 

Inductively applying Lemma \ref{cut-p} to $A = \partial_{co}|_{F_i}$, $V_1 = F_i$ and $V_2 = G_i$, we get a singular value decomposition of $\partial_{co}|_{F_i}$ for each $1\leq i \leq p-1$ in $p$-tuple. Together, we get a singular value decomposition of $\partial_{co}|_W$, that is 
\begin{equation} \label{SVD-S}
\begin{array}{cccccc} 
y_{1} & \D_{S'}y_{1} & \cdots & \D_{S'}^{p-1}y_{1}  & y_2 &\cdots   \\
\downarrow & \downarrow & \downarrow & \downarrow & \downarrow & \downarrow\\
x_{1} & \D_{S'}x_{1} & \cdots & \D_{S'}^{p-1} x_{1} & x_2 &\cdots 
\end{array}
\end{equation} 
and this is the required $p$-tuple form. In summary, we get the following proposition. 

\begin{prop} \label{SVD-S-p} Assume $\D_{S'}$ is defined, then there exists a singular value decomposition in $p$-tuple form as (\ref{SVD-S}) of 
\[ \partial_{co}\big|_{\bigoplus_{1 \leq i \leq p-1}  F_i}: \bigoplus_{1 \leq i \leq p-1}  F_i \to \bigoplus_{1 \leq i \leq p-1} G_i\]
which is compatible with action $\D_{S'}$ in the sense that their cyclic spans (under $\D_{S'}$) generate each eigenspace. \end{prop}

Now, we deal with (c). By Remark \ref{fail-1}, we will take a special consideration of eigenspace of eigenvalue $1$. Denote $K_0 = G^+_0 \oplus F_0$ where $G^+_0$ is the $1$-eigenspace from decomposition (\ref{eign-dec2}) of ${\rm Im}((\partial_{co})_{k+2})$. 

\begin{prop} \label{0-0}
The restriction map $\partial_{co}|_{K_0}: K_0 \to G_0$ only contributes $0$-length bars. 
\end{prop} 
\begin{proof} On the one hand, we know the general form of the generators of $K_0$. Indeed, for each Hamiltonian orbit (for rotation), say $x(t)$\footnote{All these orbits are normalized to be filtration 0 by adjusting cappings.}, denote $x(t)_{(i)} = x\left(t + \frac{i}{p^2} \right)$ where $0 \leq i \leq p^2-1$. Then 
\[ Cone_{k+1}(T - \xi_p \cdot {\mathds 1}) = {\Span}_{\Lambda^{\K, \Gamma}} \left< \begin{array}{cc} (0, x_i(t)),\dots,(0, x_i(t)_{(p^2-1)}) \\ (y_j(t),0),\dots, (y_j(t)_{(p^2-1)},0) \end{array} \right>_{\tiny{\begin{array}{cc} 1 \leq i \leq m\\ 1 \leq j \leq n\end{array}}}\!\!, \]
where $x_i(t)$'s have indices equal to $k$ and $y_j(t)$'s have indices equal to $k+1$. 
Moreover, these (initial) generators are  $\Lambda^{\K, \Gamma}$-orthogonal because homotopy class $\alpha$ is assumed to be primitive. We know explicitly that for unperturbed operator $\D_{R_p}$ the generators of eigenvalue $1$ are,  
\begin{align} \label{dfn-v-w} 
&v_i = x_i(t) + x_i(t)_{p}+ \cdots  + x_i(t)_{p^2-p} \\
\notag &\mbox{and}\,\,\, w_j= y_j(t) + y_j(t)_{p}+ \cdots  + y_j(t)_{p^2-p}.
\end{align}
Since $\D_{T'}$ is a strictly lower filtration perturbation of $\D_{R_p}$, the corresponding generators of eigenvalue $1$ are also strictly lower filtration perturbations of $v_i$ and $w_j$, which will be denoted by  $\tilde{v}_i$ and $\tilde{w}_j$. Then 
\begin{equation} \label{gene-k_0}
K_0 = {\Span}_{\Lambda^{\K, \Gamma}} \left< \begin{array}{cc} (0, \tilde{v}_i), (0, S' \tilde{v}_i), \dots , (0, (S')^{p-1} \tilde{v}_i) \\ (\tilde{w}_j, 0), (S'\tilde{w}_j, 0), \dots , ((S')^{p-1} \tilde{w}_j, 0) \end{array} \right>_{\tiny{\begin{array}{cc} 1 \leq i \leq m\\ 1 \leq j \leq n\end{array}}}\!\!. 
 \end{equation}
Note that generators are all eigenvector of eigenvalue $1$ of $\mathcal D_{T'}$ because $\D_{S'}$ commutes with $\D_{T'}$ and more importantly, they are  pairwise $\Lambda^{\K, \Gamma}$-orthogonal. 

On the other hand, for each $\tilde{w}_j$, 
\[ T \tilde{w}_j = (T' - P_T^{(3)}) \tilde{w}_j = \tilde{w}_j + \left\{\begin{array}{cc} \mbox{strictly lower} \\ \mbox{filtration terms} \end{array} \right\}. \]
Since $[S', T]=0$, 
\[ \partial_{co}(0, (S')^h \tilde{w}_j) =  \left((\xi_p - 1)(S')^h \tilde{w}_j + \left\{\begin{array}{cc} \mbox{strictly lower} \\ \mbox{filtration terms} \end{array} \right\} , - \partial (S')^h \tilde{w}_j  \right). \]
By Lemma \ref{or-per}, orthogonality of $\{(S')^h \tilde{w}_j, 0)\}_{h=0}^{p-1}$ implies the orthogonality of $\{\partial_{co}(0, (S')^h \tilde{w}_j)\}_{h=0}^{p-1}$ which are also eigenvectors of eigenvalue $1$ of $\mathcal D_{T'}$. Meanwhile, we can rewrite 
\begin{equation*}
K_0 = {\Span}_{\Lambda^{\K, \Gamma}} \left< \begin{array}{cc} (0, \tilde{v}_i), (0, S' \tilde{v}_i), \dots , (0, (S')^{p-1} \tilde{v}_i) \\ \partial_{co}(0, \tilde{w}_j), \partial_{co}(0, S' \tilde{w}_j), \dots , \partial_{co}(0, (S')^{p-1} \tilde{w}_j) \end{array} \right>_{\tiny{\begin{array}{cc} 1 \leq i \leq m\\ 1 \leq j \leq n\end{array}}}.
 \end{equation*}
Hence, 
\[ {\rm{Im}}(\partial_{co} |_{K_0}) = {\Span}_{\Lambda^{\K, \Gamma}} \left< \partial_{co}(0, \tilde{v}_i),\dots , \partial_{co}(0, (S')^{p-1}\tilde{v}_i)\right>_{\tiny 1 \leq i \leq m}\]
because $(\partial_{co})^2 =0$ and moreover $\{ \partial_{co}(0, v_i),\dots , \partial_{co}(0, (S')^{p-1}v_i)\}$ are $\Lambda^{\K, \Lambda}$-orthogonal by the same reason as above. In other words, for each degree $k+1$, we get a $p$-tuple singular value decomposition of $\partial_{co}|_{K_0}$ with building block, 
\[ \begin{array}{cccccc} \label{p-0-b}
(0, \tilde{v}_i) & (0, (S')\tilde{v}_i)& \cdots & (0, (S')^{p-1}\tilde{v}_i)   \\
\downarrow & \downarrow & \cdots & \downarrow \\
\partial_{co} (0, \tilde{v}_i) & \partial_{co} (0, (S')\tilde{v}_i) & \cdots & \partial_{co} (0, (S')^{p-1}\tilde{v}_i).
\end{array} \]
Finally, it's easy to check $\ell_{co}((0, \tilde{v}_i)) =\ell_{co}(\partial_{co} (0, \tilde{v}_i))$ for any perturbed $\tilde{v}_i$, so these only contribute $0$-length bars. \end{proof}

\begin{remark} The argument in subsection~5.3 is based on our perturbation of $\D_{T}$ (or $\D_{S}$) being a group action which was developed in subsection~5.1. This gives a diagonalizable operator $\D_{T'}$ (or $\D_{S'}$) with an easy eigenspace decomposition. Here, we remark that without the process of perturbation it is probably impossible to deal with $\D_{T}$ and $\D_{S}$ directly. On the one hand, $\D_{T}$ is not guaranteed to be diagonalizable. This could be overcome by introducing generalized eigenspace decomposition. On the other hand, the most difficulty part is to prove a perturbed version of Proposition \ref{inv-comp} (i.e., $T^p = {\mathds 1} + Q$ and $Q$ strictly lowers filtration), since if we regard $\D_{T}$ and $\D_S$ as perturbations of group action $\D_{R_p}$ and $\D_{R_{p^2}}$, the invariant subspace in Construction~\ref{mskc} easily loses invariance property under perturbation. \end{remark}

\subsection{Proof of Theorem \ref{p-tuples}} 
\begin{proof} Proposition \ref{SVD-S-p} and Proposition \ref{0-0}. \end{proof}

\section{Proof of Proposition \ref{non-dis}}\label{p19}
\begin{proof} The first conclusion directly comes from the Theorem \ref{p-tuples} and definition of $\mathfrak{o}_X(\phi)$ in Definition \ref{dfn-o}. For the second conclusion, suppose $s_{max}$ is the largest multiple of $p$ smaller than $m_k$, the multiplicity of degree-$k$ concise barcode of $Cone(H)_*$. Then $p \nmid m_k$ implies that $\beta_{s_{max} p+1}(\phi_H) \neq 0$ and $\beta_{(s_{max}+1)p}(\phi_H) =0$. Therefore, 
\[ \mathfrak{o}_X(\phi_H)_k \geq \beta_{s_{max} p+1}(\phi_H) - \beta_{(s_{max}+1)p}(\phi_H) =  \beta_{s_{max} p+1}(\phi_H) \geq \beta_{m_k}(\phi_H). \]
By definition again, $\mathfrak{o}_X(\phi) \geq \mathfrak{o}_X(\phi_H)_k \geq \beta_{m_k}(\phi_H)$. \end{proof}

\section{Lipschitz continuity} \label{Lc}
In this section, we will prove all three Lipschitz continuity results mentioned in the introduction. 

\subsection{Algebraic set-up}

\begin{dfn} \label{dfn-d-m} For a $\Lambda^{\K, \Gamma}$-linear map $F: (V, \ell_1) \to (W, \ell_2)$ between two orthogonalizable $\Lambda^{\K, \Gamma}$-space, $F$ is called a {\it $\delta$-morphism} if there exists a $\delta \geq 0$ such that for any $v \in V$, $\ell_2(Fv) \leq \ell_1(v) + \delta$. In particular a filtration preserving map is a $0$-morphism. 
\end{dfn}
Note that by definition, a $\delta$-morphism is automatically a $\eta$-morphism for any $\delta \leq \eta$. 

\begin{dfn} \label{dfn-d-htp} Given two Floer-type complexes 
$$(C_*, \partial_C, \ell_C)\quad \text{and}\quad (D_*, \partial_D, \ell_D),$$
suppose that chain map $\Phi: C_* \to D_*$ is a $\delta_+$-morphism and $\Psi: C_* \to D_*$ is a $\delta_-$-morphism. Then we say that $\Phi$ and $\Psi$ are {\it $(\delta_+, \delta_-, \lambda)$-homotopic} where $\lambda \in \R_{>0}$ if there exists a degree-1 $\lambda(\delta_+ +\delta_-)$-morphism $K: C_* \to D_{*+1}$ such that $\Phi - \Psi =K \circ \partial_C + \partial_D \circ K$, where $K$ is then called a {\it $(\delta_+, \delta_-, \lambda)$-homotopy}. 
\end{dfn}

\begin{ex} \label{alg-quasi} Using Definition \ref{dfn-d-htp}, the definiton of $(\delta_-, \delta_+)$-quasiequivalence between $(C_*, \partial_C, \ell_C)$ and $(D_*, \partial_D, \ell_D)$ (defined in Definition \ref{dfn-mc}) can be rephrased as 
\begin{itemize}
\item{} there exist a $\delta_+$-morphism $\Phi: C_* \to D_*$ and a  $\delta_-$-morphism $\Psi: D_* \to C_*$;
\item{} $\Psi \circ \Phi$ and ${\mathds 1}_{C_*}$ are $(\delta_+, \delta_-, 1)$-homotopic; $\Phi \circ \Psi$ and ${\mathds 1}_{D_*}$ are $(\delta_+, \delta_-, 1)$-homotopic.
\end{itemize}
\end{ex}

Now consider the following algebraic set-up. For two Floer-type complexes $(C_*, \partial_C, \ell_C)$ (simply denoted by  $C$) and $(D_*, \partial_D, \ell_D)$ (simply denoted by  $D$), we will study the following diagram (which is not necessarily commutative!)
\begin{equation*} 
\xymatrixcolsep{5pc} \xymatrix{  C \ar[r]^{S} \ar[d]_{\phi} \ar[rd]^{P} \ar@/_2.0pc/@[black][dd]_{K_C} & C \ar[d]^{\phi } \ar@/^2.0pc/@[black][dd]^{K_C} \\  D \ar[r]^{S'} \ar[d]_{\psi} \ar[rd]^{Q}  \ar@/_2.0pc/@[black][dd]_{K_D} & D\ar[d]^{\psi} \ar@/^2.0pc/@[black][dd]^{K_D} \\ C \ar[r]^S \ar[d]_{\phi} \ar[rd]^P &C \ar[d]^{\phi}\\  D \ar[r]^{S'} &D \\}
\end{equation*}
where
\begin{itemize}
\item{} $S$ and $S'$ $0$-morphism chain maps;
\item{} $\phi$ and $\psi$ induces $(\delta_{+}, \delta_{-})$-quasiequivalence with homotopies $K_C$ and $K_D$;
\item{} $S' \circ \phi$ and $\phi \circ S$ are $(\delta_+, \delta_+, \frac{1}{2})$-homotopic with homotopy $P$;
\item{} $S \circ \psi$ and $\psi \circ S'$ are $(\delta_-, \delta_-, \frac{1}{2})$-homotopic with homotopy $Q$.
\end{itemize}

Here, we give a supporting example from our story to this algebraic set-up. 

\begin{ex} Consider the chain complexes with (almost) rotation actions constructed from different Hamiltonians $H$ and $G$, that is, the following diagram for some continuation maps $C_H$, $C_G$ and $C_{H, G}$,  
\[ \xymatrix{  CF_{k}(H^{(p)}, J_{t})_{\alpha} \ar[r]^{R_p} \ar[d]_{C_{H,G}} & CF_{k}(H^{(p)}, J_{t + \frac{1}{p}})_{\alpha} \ar[r]^{C_H} \ar[d]^{(R_{p})_*(C_{H,G})} & CF_{k}(H^{(p)}, J_t)_{\alpha} \ar[d]^{C_{H,G}} \\    CF_{k}(G^{(p)}, J_{t})_{\alpha} \ar[r]^{R_p} & CF_{k}(G^{(p)}, J_{t + \frac{1}{p}})_{\alpha} \ar[r]^{C_G}  & CF_{k}(G^{(p)}, J_t)_{\alpha}. }\]
The second (right) diagram above is not necessarily commutative (where continuation map commutes with boundary operator but not necessarily commutes with other continuation maps), so the following collapsed diagram is then {\it not} necessarily a commutative diagram, 
\begin{equation} \label{ham-diag}
 \xymatrix{  CF_{k}(H^{(p)}, J_{t})_{\alpha}   \ar[r]^{T^H} \ar[d]_{C_{H,G}} & CF_{k}(H^{(p)}, J_t)_{\alpha} \ar[d]^{C_{H,G}} \\
CF_{k}(G^{(p)}, J_{t})_{\alpha} \ar[r]^{T^G} & CF_{k}(G^{(p)}, J_t)_{\alpha} } 
\end{equation}
where $T^H = C_H \circ R_p$ and $T^G = C_G \circ R_p$. Meanwhile, we can symmetrically form continuation map $C_{G,H}:  CF_{k}(G^{(p)}, J_{t})_{\alpha}  \to CF_{k}(H^{(p)}, J_{t})_{\alpha}$. Likewise, we have a (not necessarily commutative) diagram, 
\begin{equation} \label{ham-diag2}
 \xymatrix{  CF_{k}(G^{(p)}, J_{t})_{\alpha}   \ar[r]^{T^G} \ar[d]_{C_{G,H}} & CF_{k}(G^{(p)}, J_t)_{\alpha} \ar[d]^{C_{G,H}} \\
CF_{k}(H^{(p)}, J_{t})_{\alpha} \ar[r]^{T^H} & CF_{k}(H^{(p)}, J_t)_{\alpha}. } 
\end{equation}
However, we can prove these two diagrams are not far from being commutative. In fact, 
\begin{lemma} \label{htp-diag} Denote 
$$\delta_+ = p\int_0^1 \max_X (H-G) dt\quad \text{and}\quad \delta_- = p\int_0^1 - \min_X (H-G) dt.$$
In diagram (\ref{ham-diag}), 
\[ \mbox{$T^G \circ C_{H,G}$ and $ C_{H,G} \circ T^H$ are $(\delta_+, \delta_+, \frac{1}{2})$-homotopic with a homotopy $P$}.\]
Similarly, in diagram (\ref{ham-diag2}), 
\[ \mbox{$T^H \circ C_{G,H}$ and $C_{G,H} \circ T^G$ are $(\delta_-, \delta_-, \frac{1}{2})$-homotopic with a homotopy $Q$}.\]
\end{lemma}

\begin{proof} The homotopy $P$ can be constructed explicitly by $P = K \circ R_p$ where $K$ is a homotopy between $C_G \circ (R_p)_*(C_{H,G})$ and $C_{H,G} \circ C_H$. Since $R_p$ exactly preserves filtrations, the filtration shift of $P$ is given by $K$ where, by Remark \ref{diff-htps} in Section 2, it is just $\delta_+ = \frac{1}{2} (\delta_+ + \delta_+)$. The other homotopy $Q$ is constructed in a similar way.\end{proof}

Therefore, (\ref{ham-diag}), (\ref{ham-diag2}), Lemma \ref{htp-diag} and Theorem \ref{Floer-energy} together gives rise to the following diagram compatible with the algebraic set-up above. 
\begin{equation*} 
\xymatrixcolsep{5pc} \xymatrixrowsep{3pc} \xymatrix{  CF_{k}(H^{(p)}, J_{t})_{\alpha} \ar[r]^{T^H} \ar[d]_{C_{H,G}} \ar[rd]^{P} \ar@/_4.0pc/@[black][dd]_{K_H} & CF_{k}(H^{(p)}, J_t)_{\alpha} \ar[d]^{C_{H,G}} \ar@/^4.0pc/@[black][dd]^{K_H} \\  CF_{k}(G^{(p)}, J_{t})_{\alpha}  \ar[r]^{T^G} \ar[d]_{C_{G,H}} \ar[rd]^{Q}  \ar@/_4.0pc/@[black][dd]_{K_G} & CF_{k}(G^{(p)}, J_{t})_{\alpha} \ar[d]^{C_{G,H}} \ar@/^4.0pc/@[black][dd]^{K_G} \\ CF_{k}(H^{(p)}, J_{t})_{\alpha} \ar[r]^{T^H} \ar[d]_{C_{H,G}} \ar[rd]^P &CF_{k}(H^{(p)}, J_t)_{\alpha} \ar[d]^{C_{H,G}}\\  CF_{k}(G^{(p)}, J_{t})_{\alpha}  \ar[r]^{T^G} &CF_{k}(G^{(p)}, J_{t})_{\alpha}  \\}
\end{equation*}
\end{ex}

Now we remark that Proposition \ref{lip1} can be easily proved if we can prove the following general algebraic proposition. 

\begin{prop} \label{equ-cone} With the algebraic set-up above, there exist finite constants $\Delta_+$ and $\Delta_-$ such that the self-mapping cones of complex $C$ and complex $D$ with respect to maps $S$ and $S'$, $Cone_C(S)$ and $Cone_D(S')$, are $(\Delta_{+}, \Delta_{-})$-quasiequivalent. Moreover, $\Delta_+ + \Delta_- \leq 6 (\delta_+ + \delta_-)$. \end{prop}
In other words, by Example \ref{alg-quasi}, we are looking for $\Delta_+$-morphism and $\Delta_-$-morphism with $(\Delta_+, \Delta_-)$-homotopies. Assuming Proposition \ref{equ-cone}, we have 

\begin{proof} [Proof of Proposition \ref{lip1}] Denote 
\[\delta_+ = p \int_0^1 \max_X (H-G) dt \,\,\,\,\mbox{and}\,\,\,\, \delta_- = p \int_0^1 -\min_X (H-G) dt.\]
Let $S = T^H - \xi_p \cdot {\mathds 1}$ and $S' = T^G - \xi_p \cdot {\mathds 1}$. First, they are $0$-morphisms since $T^H$ and $T^G$ preserves filtration. Moreover, they are chain maps since $T^H$ and $T^G$ commutes with boundary operator $\partial$. Second, $C_{G, H} \circ C_{H, G}$ provide a $(\delta_+, \delta_-)$-quasiequivalence by Theorem \ref{Floer-energy}. Third, $S' \circ C_{H,G}$ and $C_{H, G} \circ S$ are $(\delta_+, \delta_+, \frac{1}{2})$-homotopic and $S \circ C_{G,H}$ and $C_{G, H} \circ S'$ are $(\delta_-, \delta_-, \frac{1}{2})$-homotopic by Lemma \ref{htp-diag}. Therefore, all the conditions in the algebraic set-up are satisfied. By Proposition \ref{equ-cone}, there exists a $(\Delta_+, \Delta_-)$-quasiequivalence between $Cone(H)_*$ and $Cone(G)_*$. So we have 
\[ d_Q(Cone(H)_*, Cone(G)_*) \leq \frac{ \Delta_+ + \Delta_-}{2}  \leq 3(\delta_+ + \delta_-) = 3p \cdot||H-G||_H.\]
Thus we get the conclusion. \end{proof}

For the rest of this section, we will focus on the proof of Proposition \ref{equ-cone}. 

\subsection{Homotopy category}

\begin{dfn} Denote by $\mathcal F$ the following category:\footnote{It is routine to check that $\mathcal F$ is an additive category.}
\begin{itemize}
\item{} Object in $\mathcal F$: Floer-type complex $(C_*, \partial_C, \ell_C)$;
\item{} Morphism in $\mathcal F$: $0$-morphism chain map.
\end{itemize}
\end{dfn}
Note that by this definition, any $\delta$-morphism $\phi$ with $\delta>0$ is {\it not} a morphism in $\mathcal F$. On the other hand, fixing a constant $\delta$, we can define a map on $\mathcal F$ itself, denoted by  $\Sigma^{\delta}$ by
\begin{itemize}
\item{} $\Sigma^{\delta}(C_*, \partial_C, \ell_C) = (C_*, \partial_C, \ell_C + \delta)$;
\item{} $\Sigma^{\delta}\phi = \phi$. 
\end{itemize}
It is easy to see this map is a functor. This functor does not change morphism and shifts the filtration of any Floer-type complex by $\delta$. In short notation, $\Sigma^{\delta} C: = \Sigma^{\delta}(C_*, \partial_C, \ell_C)$ and  
\[ \ell_{{\Sigma^{\delta}} C} = \ell_C + \delta.\]
Therefore, for any $\delta$-morphism $\phi: C \to D$, we know that $\Sigma^{\delta} C  \xrightarrow{\phi} D$ is a well-defined morphism in $\mathcal F$ since $\ell_D(\phi(x)) \leq \ell_C(x) + \delta = \ell_{\Sigma^{\delta}C}(x)$ for any $x \in C$. 

\begin{ex} \label{ex-quasi-htp} Suppose $C$ and $D$ are $(\delta_+, \delta_{-})$-quasiequivalent with chain maps $\phi$ and $\psi$. By discussion above, we have well-defined maps in $\mathcal F$, 
\[ \Sigma^{\delta_+ + \delta_-} C \xrightarrow{\phi} \Sigma^{\delta_{-}} D \xrightarrow{\psi} C \,\,\,\,\,\,\,\,\mbox{and} \,\,\,\,\,\,\,\,\, \Sigma^{\delta_+ + \delta_-} D \xrightarrow{\psi} \Sigma^{\delta_{+}} C \xrightarrow{\phi} D. \]
Meanwhile, there is an obvious map, defined by shifting the filtration of {\bf domain}, $i_{\delta_+ + \delta_-}: \Sigma^{\delta_+ + \delta_{-}} C \to C$ which is the identity map on $C$ {\it as a vector space}. We emphasize that in the category $\mathcal F$, 
\[ \Sigma^{\delta_+ + \delta_-} C \neq C, \]
because their filtrations are different. Therefore, $i_{\delta_{+} + \delta_{-}}$ is {\bf not} an identity map. Moreover, by definition of $(\delta_+, \delta_-)$-quasiequivalence, we know
\[ \psi \circ \phi\,\,\,\mbox{is homotopic to} \,\,\, i_{\delta_+ + \delta_{-}} \,\,\,\,\mbox{in} \,\,\,\mathcal F\]
and 
\[ \phi \circ \psi\,\,\,\mbox{is homotopic to} \,\,\, i_{\delta_+ + \delta_{-}} \,\,\,\,\ \mbox{in} \,\,\, \mathcal F.\]
\end{ex}

The following easy lemma will be useful later. 

\begin{lemma} \label{i-commute} For any $\delta \geq 0$, morphism $i_{\delta}$ commutes with any morphism in $\mathcal F$. \end{lemma}

\begin{proof} Suppose $\phi: C \to D$ is defined in $\mathcal F$. We have the following commutative diagram 
\[ \xymatrix{
\Sigma^{\delta} C \ar[r]^{i_{\delta}} \ar[d]_{\Sigma^{\delta}{\phi}} & C \ar[d]^{\phi}\\
\Sigma^{\delta} D \ar[r]_{i_{\delta}} & D,} \]
that is $\phi \circ i_{\delta} = i_{\delta} \circ \Sigma^{\delta} \phi = i_{\delta} \circ \phi$. Here, we want to emphasize that for the third term of this equality, even if we write $i_{\delta} \circ \phi$ for the same notation $\phi$ but this $\phi$'s domain and target are changed, i.e., the filtrations have been shifted by $\delta$. \end{proof}

Motivated by Example \ref{ex-quasi-htp}, we will consider the following more ``concise'' category defined as follows.
\begin{dfn} Define the filtered homotopy category of Floer-type complexes $K(\mathcal F)$:
\begin{itemize}
\item{} Object in $K(\mathcal F)$: the same as object in $\mathcal F$;
\item{} Morphism in $K(\mathcal F)$: filtered homotopy class (in $\mathcal F$). 
\end{itemize}
\end{dfn}

\begin{ex} \label{identity} The condition that $C$ and $D$ are $(\delta_+, \delta_{-})$-quasiequivalent with chain maps $\phi$ and $\psi$ is equivalent to the condition $\psi \circ \phi = i_{\delta_+ + \delta_-}$ and $\phi \circ \psi = i_{\delta_+ + \delta_-}$ in $K(\mathcal F)$. \end{ex} 

Using this language, we transfer the problem of constructing certain chain maps which are (desired) bounded filtration shifting homotopies into an existence question of morphisms in $K(\mathcal F)$ up to a finite filtration shift. Therefore, we can restate Proposition \ref{equ-cone} in this new category as follows.
\begin{prop}\label{equ-cone2}
Suppose that we have the algebraic set-up above. There exist finite positive constants $\Delta_+$ and $\Delta_-$ such that there exist morphism $\Phi$ and $\Psi$ in $\mathcal F$ to form the following sequence of maps, 
\[ \Sigma^{\Delta_+ + \Delta_-} Cone_D(S') \xrightarrow{\Psi} \Sigma^{\Delta_+} Cone_C(S) \xrightarrow{\Phi} Cone_D(S') \xrightarrow{\Psi} \Sigma^{- \Delta_-}Cone_C(S) \]
and 
\[ \Phi \circ \Psi = i_{\Delta_+ + \Delta_-} \,\,\,\,\,\mbox{and}\,\,\,\,\,  \Psi \circ \Phi = i_{\Delta_+ + \Delta_-} \,\,\,\,\, \mbox{in} \,\,\,K(\mathcal F). \]
Moreover, $ \Delta_+ + \Delta_-\leq 6(\delta_+ + \delta_-)$.
\end{prop}

\begin{remark} Compared with the category of chain complexes, the advantage of working on homotopy category is that homotopy category $K(\mathcal F)$ is a triangulated category, therefore we can introduce mapping cones to form distinguished triangles (cf. diagram (\ref{dist-tri})) and use various homological algebra results. Importantly, thanks to Theorem B in \cite{UZ15}, even though in this filtered homotopy category ``equality'' relation is represented only by filtered homotopy equivalences (instead of chain isomorphisms), concise barcode is still well-defined. \end{remark}

Consider the following diagram in $K(\mathcal F)$, 
\begin{equation}\label{dist-tri}
  \xymatrix{ \Sigma^{\delta_+} C \ar[r]^{S} \ar[d]^{\phi} & \Sigma^{\delta_+} C \ar[r]^-{\iota}\ar[d]^{\phi} & \Sigma^{\delta_+} Cone_C(S)
    \ar[r]^-{\pi}\ar[d]^{h} & \Sigma^{\delta_+} C[1] \ar[d]^{\phi[1]} \\ D \ar[r]_{S'} & D 
    \ar[r]_-{\iota} & Cone_D(S') \ar[r]_-{\pi}& D[1]}
\end{equation}
where $\iota$ is the inclusion and $\pi$ is the projection. When $\delta_+ = \delta_- = 0$ (back to the classical case), by 4. Corollary in \cite[p. 242]{GM}, if $\phi$ is an isomorphism (invertible), then the middle morphism $h$ is also an isomorphism and existence of this map is given by axiom one (TR1) of triangulated category. The basic tool of proving this is the fact that functors ${\Hom}(Cone_D(S'), -)$ and ${\Hom}(-, Cone_C(S))$ are cohomological functors (so they induce long exact sequences), together with Five Lemma. 

Second, here in our situation $\phi$ and $\psi$ are not necessarily invertible (again because $i_{\delta_++ \delta_-}$ is not the identity map). However, a trick that we will use is (1) composing (or pre-composing) with map $i_{\delta_+ + \delta_-}$, which by definition is shifting the filtration of domains and then (2) using the identity from Example \ref{identity}. We give an example demonstrating this. 

\begin{ex} \label{ex-shift}
Fix an object $A$ in $K(\mathcal F)$ and consider 
\[{\Hom}(A, \Sigma^{\delta_+}C) \xrightarrow{ \phi \circ} {\Hom}(A, D). \]
Taking an element $f \in {\Hom}(A, D)$, it is not necessary that there exists some element $g \in {\Hom}(A, \Sigma^{\delta_+} C)$ such that $f = \phi \circ g$. In other words, existence of preimage of $f$ under $\phi$ is not guaranteed. However, consider $f \circ i_{\delta_+ + \delta_-} \in {\Hom}(\Sigma^{\delta_+ + \delta_-} A, D)$. By Lemma \ref{i-commute}, 
\[ f \circ i_{\delta_+ + \delta_-} = i_{\delta_+ + \delta_-} \circ f = \phi \circ \psi \circ f = \phi \circ (\psi \circ f)\]
and $\psi \circ f \in {\Hom}(\Sigma^{\delta_+ + \delta_-} A, \Sigma^{\delta_+} C)$. So there exists a preimage of ``shifted'' $f$ by $i_{\delta_+ + \delta_-}$. \end{ex}

\begin{prop}\label{inverse}
For $i_{3(\delta_+ + \delta_-)}  \in {\Hom}(\Sigma^{3(\delta_++ \delta_-)} Cone_D(S'), Cone_D(S'))$, there exists an element $j  \in {\Hom}(\Sigma^{3(\delta_++ \delta_-)} Cone_D(S'), \Sigma^{\delta_+} Cone_C(S))$ such that $h \circ j = i_{3(\delta_+ + \delta_-)}$. 
\end{prop}

\begin{proof} 
{\bf Step one}. First consider the following diagram 
\[  \small{
\xymatrix{ {\Hom}(\Sigma^{\delta_++ \delta_-} Cone_D(S'), \Sigma^{\delta_+} Cone_C(S)) \ar[r]^-{\pi \circ} \ar[d]^{h \circ } & {\Hom}(\Sigma^{\delta_++ \delta_-} Cone_D(S'), \Sigma^{\delta_+} C[1]) \ar[d]^-{\phi[1] \circ } \\  
{\Hom}(\Sigma^{\delta_++ \delta_-} Cone_D(S'), Cone_D(S'))\ar[r]^-{\pi \circ } &  {\Hom}(\Sigma^{\delta_++ \delta_-} Cone_D(S'), D[1]).}}\]
For $i_{\delta_+ + \delta_-} \in {\Hom}(\Sigma^{\delta_++ \delta_-} Cone_D(S'), Cone_D(S'))$, 
\[ \pi \circ i_{\delta_+ + \delta_-} = i_{\delta_+ + \delta_-} \circ \pi = \phi[1] \circ \psi[1] \circ \pi = \phi[1] \circ (\psi[1] \circ \pi) \]
where $\psi[1] \circ \pi \in {\Hom}(\Sigma^{\delta_++ \delta_-} Cone_D(S'), \Sigma^{\delta_+} C[1])$. Moreover, we can check 
\[ S[1] \circ \psi[1] \circ \pi = \psi[1] \circ S'[1] \circ \pi = \psi[1] \circ (S'[1] \circ \pi) = 0.\]
Therefore, $\psi[1] \circ \pi \in \ker(S[1] \circ)$. By the exactness of the top row (after applying cohomological functor ${\Hom}(\Sigma^{\delta_+ + \delta_-} Cone_D(S'), \cdot)$), we know 
\begin{equation}\label{exact1}
\pi \circ z = \psi[1] \circ \pi
\end{equation}
for some $z \in {\Hom}(\Sigma^{\delta_++ \delta_-} Cone_D(S'), \Sigma^{\delta_+} Cone_C(S))$. Now consider
\[ \tilde{z} := i_{\delta_+ + \delta_-} \circ z = z \circ i_{\delta_+ + \delta_-}, \]
so $\tilde{z} \in {\Hom}(\Sigma^{2(\delta_++ \delta_-)} Cone_D(S'), \Sigma^{\delta_+} Cone_C(S))$. Moreover, by (\ref{exact1}), 
\begin{equation}\label{exact2}
\pi \circ \tilde{z} = i_{\delta_+ + \delta_-} \circ \psi[1] \circ \pi.
\end{equation} 

{\bf Step two}. Considering the following diagram 
\[  \small{
\xymatrix{  {\Hom}(\Sigma^{2(\delta_++ \delta_-)} Cone_D(S'), \Sigma^{\delta_+} C) \ar[r]^-{\iota \circ} \ar[d]^{\phi \circ} & {\Hom}(\Sigma^{2(\delta_++ \delta_-)} Cone_D(S'), \Sigma^{\delta_+} Cone_C(S)) \ar[d]^{h \circ} \\  {\Hom}(\Sigma^{2(\delta_++ \delta_-)} Cone_D(S'), D)\ar[r]^-{\iota \circ } &  {\Hom}(\Sigma^{2(\delta_++ \delta_-)} Cone_D(S'), Cone_D(S'))}} \]
and element 
\[ h \circ \tilde{z}  - i_{2(\delta_+ + \delta_-)} \in {\Hom}(\Sigma^{2(\delta_++ \delta_-)} Cone_D(S'), Cone_D(S')). \]
By the commutativity of previous diagram and (\ref{exact2}), 
\begin{align*}
\pi \circ (h \circ \tilde{z} - i_{2(\delta_+ + \delta_-)}) &= \pi \circ h \circ \tilde{z} - \pi \circ i_{2(\delta_+ + \delta_-)} \\
& = \phi[1] \circ \pi \circ \tilde{z} - \pi \circ i^2_{(\delta_+ + \delta_-)} \\
& = \phi[1] \circ i_{(\delta_+ + \delta_-)} \circ \psi[1] \circ \pi - \pi \circ i^2_{(\delta_+ + \delta_-)}\\
& = 0. 
\end{align*}
Therefore, $h \circ \tilde{z}  - i_{2(\delta_+ + \delta_-)} \in \ker(\pi \circ)$. By the exactness of the bottom row (after applying cohomological functor ${\Hom}(\Sigma^{2(\delta_+ + \delta_-)} Cone_D(S'), \cdot)$), we know 
\begin{equation} \label{exact3}
h \circ \tilde{z}  - i_{2(\delta_+ + \delta_-)} = \iota \circ w
\end{equation}
for some $w \in {\Hom}(\Sigma^{2(\delta_+ + \delta_-)} Cone_D(S'), D)$. This is exactly the situation that $w$ does not necessarily have any preimage under $\phi \circ$ in 
$${\Hom}(\Sigma^{2(\delta_++ \delta_-)} Cone_D(S'), \Sigma^{\delta_+} C).$$
Now we will use the trick from Example \ref{ex-shift} by considering ``shifted'' relation of (\ref{exact3}), that is 
\begin{equation} \label{exact4}
h \circ (i_{\delta_+ + \delta_-} \circ \tilde{z})  - i_{3(\delta_+ + \delta_-)} = \iota \circ (i_{\delta_+ + \delta_-} \circ w)
\end{equation}
in ${\Hom}(\Sigma^{3(\delta_++ \delta_-)} Cone_D(S'), Cone_D(S'))$. As in Example \ref{ex-shift}, for element $i_{\delta_+ + \delta_-} \circ w \in {\Hom}(\Sigma^{3(\delta_+ + \delta_-)} Cone_D(S'), D)$, we have 
\[ i_{\delta_+ + \delta_-}  \circ w  = \phi \circ (\psi \circ w) \]
where $\psi \circ w \in {\Hom}(\Sigma^{3(\delta_++ \delta_-)} Cone_D(S'), \Sigma^{\delta_+} C)$. 

{\bf Step three}. Consider element 
\[ j : = -\iota \circ \psi \circ w  + i_{\delta_+ + \delta_-} \circ \tilde{z} \]
in ${\Hom}(\Sigma^{3(\delta_++ \delta_-)} Cone_D(S'), \Sigma^{\delta_+} Cone_C(S))$. Then by (\ref{exact4}), it is easy to check 
\begin{align*}
h \circ j & = - h \circ \iota \circ \psi \circ w + h \circ i_{\delta_+ + \delta_-} \circ \tilde{z} \\
& = - \iota \circ \phi \circ \psi \circ w + i_{3(\delta_+ + \delta_-)} + \iota \circ i_{\delta_+ + \delta_-} \circ w \\
& = i_{3(\delta_+ + \delta_-)}.
\end{align*}
This $j$ is the required element. 
 \end{proof}

\begin{proof} [Proof of Proposition \ref{equ-cone2}]
Similarly to the proof of Proposition \ref{inverse}, we can apply contravariant functor ${\Hom}(-, \Sigma^{*}Cone_C(S))$ (for some shift $*$) on the first row of diagram (\ref{dist-tri}) shifted by functor $\Sigma^{\delta_+}$ and we will get a $j' \in {\Hom}(Cone_D(S'), \Sigma^{-2\delta_+ - 3 \delta_-} Cone_C(S))$ such that 
\begin{equation} \label{other-comp}
j'\circ h = i^3_{\delta_+ + \delta_-}. 
\end{equation} 
Combined with Proposition \ref{inverse}, we have the following composition, 
\begin{align} \label{composition2}
	\Sigma^{3 \delta_+ + 3\delta_-} Cone_D(S') &\xrightarrow{j} \Sigma^{\delta_+} Cone_C(S)\\
	\notag &\xrightarrow{h} Cone_D(S') \xrightarrow{j'} \Sigma^{-2\delta_+ - 3\delta_-}Cone_C(S). 
\end{align}
Now $j$ and $j'$ are not necessarily the same since $i^3_{\delta_+ + \delta_-}$ is not an isomorphism. But taking advantage that $i^3_{\delta_+ + \delta_-}$ commutes with any morphism, we have
\[ (j' \circ h \circ j) \circ h = j'\circ (h\circ j)\circ h = j' \circ i^3_{\delta_+ + \delta_-} \circ h = (j' \circ h) \circ i^3_{\delta_+ + \delta_-} = i^6_{\delta_+ + \delta_-} \]
and 
\[ h \circ (j' \circ h \circ j) = h \circ (j' \circ h) \circ j = h \circ i^3_{\delta_+ +\delta_-} \circ j = (h\circ j) \circ i^3_{\delta_+ + \delta_-} = i^6_{\delta_+ + \delta_-}. \]
Denote $g = j' \circ h \circ j$, then we can improve chain of maps (\ref{composition2}) into 
\begin{align}\label{new-composition}
	\Sigma^{6 \delta_+ + 6\delta_-} Cone_D(S') &\xrightarrow{g} \Sigma^{\delta_+} Cone_C(S)\\
	\notag &\xrightarrow{h} Cone_D(S') \xrightarrow{g} \Sigma^{-5\delta_+ - 6\delta_-}Cone_C(S). 
 \end{align}
Hence, we will get the conclusion by setting $\Delta_+ = \delta_+$, $\Delta_- = 5 \delta_+ + 6 \delta_-$ and $\Phi = h$ and $\Psi = g$. \end{proof} 

\subsection{Proofs of Proposition \ref{lip2} and Proposition \ref{lip3}}

\begin{proof}[Proof of Proposition \ref{lip2}] By definition, $(\Delta_+, \Delta_-)$-quasiequivalence is in particular a $(\Delta_+ + \Delta_-, \Delta_+ + \Delta_-)$-quasiequivalence. For any given $\ep>0$, there exists a $(\Delta_+, \Delta_-)$-quasiequivalence between $Cone(H)_*$ and $Cone(G)_*$ with 
\[ \Delta_+ + \Delta_- \leq 2d_Q(Cone(H)_*, Cone(G)_*) + 2\ep \]
Moreover, by Corollary 8.8 in \cite{UZ15}, we know 
\begin{align*}
 |\beta_i(\phi_H) - \beta_i(\phi_G)| \leq 2(\Delta_+ + \Delta_-) & \leq 4 d_Q(Cone(H)_*, Cone(G)_*) + 4\ep.
 \end{align*}
Since $\ep$ is arbitrarily chosen, we get the conclusion. \end{proof}

\begin{proof} [Proof of Proposition \ref{lip3}] Clearly, we only need to prove the following conclusion 
\begin{equation} \label{degree-k dsi}
 \mathfrak{o}_X(\phi)_k -\mathfrak{o}_X(\psi)_k \leq 24p \cdot d_H(\phi, \psi)
 \end{equation}
for each degree $k$. Indeed, for any $\ep>0$, there exists some $k \in \Z$ such that $\mathfrak{o}_X(\phi) \leq \mathfrak{o}_X(\phi)_k + \ep$. Meanwhile, for this $k$, we have $\mathfrak{o}_X(\psi)_k \leq \mathfrak{o}_X(\psi)$. Therefore, (\ref{degree-k dsi}) implies 
\[ \mathfrak{o}_X(\phi) - \mathfrak{o}_X(\psi) \leq \mathfrak{o}_X(\phi)_k - \mathfrak{o}_X(\psi)_k + \ep \leq 24p \cdot d_H(\phi, \psi) + \ep\]
and by symmetry, we get the other direction.

Now suppose $\phi = \phi_H$ and $\psi = \psi_G$. Proposition \ref{lip1} and Proposition \ref{lip2} give
\[ |\beta_i(\phi) - \beta_i(\psi)| \leq 12p \cdot ||H-G||_H. \]
If $\mathfrak{o}_X(\phi)_k$ is realized by some index $s_0 \in \N$, then 
\begin{align*}
\mathfrak{o}_k(\phi) - \mathfrak{o}_k(\psi) &\leq (\beta_{s_0p+1}(\phi) - \beta_{(s_0+1)p}(\phi)) - (\beta_{s_0p+1}(\psi) - \beta_{(s_0+1)p}(\psi)) \\
& \leq (\beta_{s_0p+1}(\phi) - \beta_{s_0p+1}(\psi)) + ( \beta_{(s_0+1)p}(\psi) - \beta_{(s_0+1)p}(\phi))\\
	&\leq 24p \cdot ||H-G||_H.
\end{align*}
So we get the conclusion by taking the infimum over the generating Hamiltonians of Hamiltonian diffeomorphisms $\phi$ and $\psi$. 
\end{proof}

\section{Computation of egg-beater model $\Sigma_g$} \label{egg}
First, we will state the CZ-index formula for generator of $CF_*(\Sigma_g, \phi_{\lambda}^p)_{\alpha_{\lambda}}$ to help us label the indices later in this section. Recall that from the construction of egg-beater model, for each fixed point $z$, it will naturally come up with another $2p-1$ intermediate points, denoted by  $z_1, \dots , z_{2p-1}$ and set $z = z_0$. Using coordinates, we denote 
\[ z_{2j} = (x_{2j}, y_{2j}) \,\,\,\,\,\mbox{and} \,\,\,\,\, z_{2j+1} = (x_{2j+1}, y_{2j+1}). \]
By (41) in \cite{PS14}, we know 
\[ (x_{2j+1}, y_{2j+1}) = (-y_{2j+2}, x_{2j}) \,\,\,\,\,\mbox{and} \,\,\,\,\,\,\, (x_{2p}, y_{2p}) = (x_0, y_0). \]
So we will expect our formula only involves $x$ and $y$ coordinate of {\it even} index intermediate points. 
\begin{theorem}\label{CZ}[Theorem 5.2 in \cite{egg-team}]
In egg-beater model (\ref{egg-beater}), the CZ-index of fixed point $z$ with intermediate points $z_1, \dots, z_{2p-1}$ is given by
\begin{equation} \label{CZ-index}
\mu_{CZ}(z) = 1+ \frac{1}{2} \sum_{j=0}^{p-1} ({\sign}(x_{2j}) - {\sign}(y_{2j})).
\end{equation}
\end{theorem}

Note that each summand ${\sign}(x_{2j}) - {\sign}(y_{2j})$ only takes value $-2, 0$ and $2$, so the range of index is $[-p+1, p+1]$. 

\begin{cor} \label{large-gap} For each degree $k \in [-p+1, p+1]$, there are $\binom{2p}{k +p-1}$-many fixed points $z$ coming in with their $p$-tuple (\ref{egg-p-tuples}). \end{cor} 
\begin{proof} For the first conclusion, suppose that, within all the possible choices for ${\sign}(x_{2j})$ and $-{\sign}(y_{2j})$, we have $a$-many $1$ and $b$-many $-1$, then by (\ref{CZ-index}), we know $a - b = 2(k-1)$. Meanwhile, $a+b = 2p$, therefore, we need to choose $a = k+p-1$-many $1$'s. \end{proof}

Now we will consider the self-mapping cone of egg-beater model, denoted by  
\[ Cone_{\Sigma_g}(H_{\lambda})_*: = Cone_{CF(\Sigma_g, \phi^p_{\lambda})_{\alpha_{\lambda}}}(T - \xi_p \cdot {\mathds 1})_*. \]
First, by a standard result from Floer theory (see Theorem \ref{Ar-conj-cont}), 
$$ H_*(CF(\Sigma_g, \phi^p_{\lambda})_{\alpha_{\lambda}}) = 0$$
because $\alpha_{\lambda}$ is represented by a family of non-contractible loops,
\begin{equation} \label{hm-cone}
H_*(Cone_{\Sigma_g}(H_{\lambda}))= 0
\end{equation}
by the following long exact sequence 
\begin{align*}
	\cdots &\to H_k(CF(\Sigma_g, \phi^p_{\lambda})_{\alpha_{\lambda}}) \xrightarrow{((T - \xi_p \cdot {\mathds 1})_k)_*} H_k(CF(\Sigma_g, \phi^p_{\lambda})_{\alpha_{\lambda}}) \\
&\xrightarrow{\iota_*} H_{k}(Cone_{\Sigma_g}(H_{\lambda})) \xrightarrow{p_*} H_{k-1}(CF(\Sigma_g, \phi^p_{\lambda})_{\alpha_{\lambda}}) \to \cdots
\end{align*} 
where $\iota$ is inclusion and $p$ is projection. Therefore, in terms of barcode, there are only finite length bars, possibly with zero length. Meanwhile, by rank-nullity theorem, it is easy to check 
\begin{align} \label{dim-ima}
	\dim (\ker(\partial_{co})_k) &=\dim({\rm Im}(\partial_{co})_{k+1})\\
	\notag &=  \dim(CF_k(\Sigma_g,\phi^p_{\lambda})_{\alpha_{\lambda}})= p \cdot \binom{2p}{k+p-1}.
\end{align}
Therefore, the total multiplicity of verbose barcode of $Cone_{\Sigma_g}(H_{\lambda})_*$ is 
\[ \sum_{k} \dim \ker(\partial_{co})_k = \frac{1}{2} \sum_{k} \dim(Cone_{\Sigma_g}(H_{\lambda})_k) = p \cdot 2^{2p}. \]
Therefore, to prove Proposition \ref{multi-egg}, we need to prove there are exactly $(p-1)\cdot 2^{2p}$-many zero length bars in total for self-mapping cone of egg-beater model. The way to prove this is by running the algorithm, Theorem 3.5 in \cite{UZ15}, to get a singular value decomposition so that we can compute barcode in this concrete example. We will demonstrate this by an explicit computational example here. 

\begin{remark} The structure of mapping cone is the key to succeed computing barcode and count its multiplicity. If we work only on $CF_*(\Sigma_g, \phi^p_\lambda)_{\alpha_{\lambda}}$, in general, we don't have enough information of Floer boundary operator of $CF_*(\Sigma_g, \phi^p_\lambda)_{\alpha_{\lambda}}$ to compute the associated barcode. \end{remark}

\subsection{Example of computing barcode} \label{ssec-exe-computing-barcode}
Here we first briefly recall the process of generating a singular value decomposition of linear map $T: (V, \ell_1) \to (W, \ell_2)$ in a special case that $\Gamma$ being trivial (here egg-beater model satisfies). For a general process, see Theorem 3.5 in \cite{UZ15}. Given an ordered orthogonal basis $(v_1, \dots ,v_n)$ for $V$ and $(w_1, \dots , w_m)$ for $W$, we will run the Gaussian elimination by choosing pivot column where the optimal index pair $(i_0, j_0)$ lies in. Here optimal index pair $(i_0, j_0)$ means 
\begin{align} \label{opt-index}
	&\ell_1(v_{j_0}) - \ell_2(w_{i_0}) \leq \ell_1(v_j) - \ell_2(w_i) \\
	\notag &\mbox{for all} \,\,(i,j) \in \{1, \dots, n\} \times \{1, \dots, m\}.
\end{align}
Therefore, after Gaussian elimination, $v_j$ is modified to be $v_j' = v_j - c v_{j_0}$ for some constant $c \in \mathcal K$. Moreover, we know 
\[ \ell(v'_j) = \ell(v_j)  \,\,\,\,\,\mbox{and} \,\,\,\,\,\, \ell(Tv_j') \leq \ell(Tv_j). \]
Note that for a filtration preserving map $T$,
\begin{equation} \label{zero-length}
 \ell(v'_j) - \ell(T v_j') = 0 \,\,\,\,\,\mbox{implies} \,\,\,\,\, \ell(v_j) - \ell(T v_j)=0 .
\end{equation}

Each step will start from choosing a pivot column and end up with deleting this column from consideration of pivot column for next step. Here is an example relating with our model.\footnote{We recommend reader to go through this example carefully before reading the proof of Proposition \ref{multi-egg} right after it.} 

\begin{ex} \label{cpt-bar} Let $p=3$. For ${Cone_{\Sigma_g}(H_{\lambda})}_{-1} \xrightarrow{(\partial_{co})_{-1}} {Cone_{\Sigma_g}(H_{\lambda})}_{-2}$, our initial (orthogonal) bases are
\[ \left((b_1,0), (\phi_{\lambda}(b_1),0), (\phi_{\lambda}^{2}(b_1), 0), \dots , (b_6,0), (\phi_{\lambda}(b_6),0), (\phi_{\lambda}^2(b_6), 0) \right) \]
and
\[ \left((0, a), (0, \phi_{\lambda}(a)), (0, \phi_{\lambda}^2(a)) \right) \]
for ${Cone_{\Sigma_g}(H_{\lambda})}_{-1}$ where $b_i$ (for $i = 1, \dots , 6$) and $a$ are fixed points; and 
\[  \left((a,0), (\phi_{\lambda}(a),0), (\phi_{\lambda}^2(a),0) \right)\]
for ${Cone_{\Sigma_g}(H_{\lambda})}_{-2}$.
First, by action functional formula for each fixed point (see $(49)$ in \cite{PS14}), we can generically get, for $s \in \{0, 1, 2\}$,
\[ \ell_{co}((\phi_{\lambda}^s(b_i),0)) <  \ell_{co}((\phi_{\lambda}^s(b_j),0)) \,\,\,\,\,\mbox{whenever}\,\,\,\,\,\, i<j.\]
Therefore, under the boundary map $\partial_{co}$, we have,
\[ \partial_{co}(\phi_{\lambda}^s (b_i), 0) = \begin{bmatrix}
    \partial & - C \circ R_3 + \xi_3 \cdot {\mathds 1} \\
            0 & -\partial
  \end{bmatrix} \begin{bmatrix} \phi_{\lambda}^s(b_i) \\ 0 \end{bmatrix} = \begin{bmatrix}  L(\phi_{\lambda}^{t}(a)) \\ 0 \end{bmatrix} \]
where $L(\phi_{\lambda}^{t}(a))$ is a linear combination of generators $\phi_{\lambda}^{t}(a)$ for $t \in \{0, 1,2\}$. Meanwhile, 
\begin{align*} \partial_{co}((0, \phi_{\lambda}^s(a)) &= \begin{bmatrix} - C (\phi_{\lambda}^{(s+1) \, \mod\,3}(a)) + \xi_3 \phi_{\lambda}^s(a) \\0 \end{bmatrix} \\
&= \begin{bmatrix} -\phi_{\lambda}^{(s+1) \, \mod\,3}(a) + \xi_3 \phi_{\lambda}^s(a) \\0 \end{bmatrix} 
\end{align*}
since continuation map $C$ is always in the form 
\[ C(x) = x + \, \left\{\begin{array}{cc} \mbox{strictly  lower} \\ \mbox{filtration term} \end{array} \right\} \]
and by the speciality of generators (with lowest grading), there is no strictly lower filtration terms. In other words, if $(\partial_{co})_{-1}$ is represented by a matrix, it will be a $3$ by $21$ matrix as follows,  
\begin{equation} \label{3-matrix}
 \begin{bmatrix} \xi_3 & 0 & -1 & * \cdots * \\
-1 & \xi_3 & 0 & * \cdots *\\
0 & -1 & \xi_3 & * \cdots * 
\end{bmatrix} .
\end{equation}
More importantly, we note 
\[ \ell_{co}((0, \phi_{\lambda}^s(a)) - \ell_{co}((\phi_{\lambda}^t(a),0)) = \ell(a) - \ell(a) = 0 \]
while for any $*$ position in the matrix (\ref{3-matrix}), 
\begin{equation} \label{stri-lower}
 \ell_{co}(((\phi_{\lambda}^s(b_i),0)) - \ell_{co}(0, \phi_{\lambda}^t(a)) >0 
 \end{equation}
therefore, our first choice of pivot column should come from one of the first three columns. Let's take column one. After Gaussian elimination, we have 
\begin{equation*} 
 \begin{bmatrix} \xi_3 & 0 & 0 & * \cdots * \\
-1 & \xi_3 & -\xi_3^2 & * \cdots *\\
0 & -1 & \xi_3 & * \cdots * 
\end{bmatrix} 
\end{equation*}
and $((0, \phi_{\lambda}^2(a))$ is changed to $(0, \phi_{\lambda}^2(a) + \xi_3^2 a)$. Again, by the same reason, our second pivot column can be taken as the second column, so Gaussian elimination will give 
\begin{equation*} 
 \begin{bmatrix} \xi_3 & 0 & 0 & * \cdots * \\
-1 & \xi_3 & 0 & * \cdots *\\
0 & -1 & 0 & * \cdots * 
\end{bmatrix} 
\end{equation*}
and $(0, \phi_{\lambda}^2(a) + \xi_3^2 a)$ is changed to $(0, \phi_{\lambda}^2(a) + \xi_3 \phi_{\lambda}(a) +\xi_3^2 a )$. Note that second factor is in the kernel of $R_3 - \xi_3 \cdot {\mathds 1}$. Thus we get two elements in the singular value decomposition, $(0, a) \to (-\phi_{\lambda}(a) + \xi_3 a, 0)$ and $(0, \phi_{\lambda}(a)) \to (-\phi_{\lambda}^2(a) + \xi_3\phi_{\lambda}(a), 0)$. Both of them will give zero length bars. The choice of next pivot column will start from column corresponding to generator $(\phi_{\lambda}^s (b_i), 0)$. But by (\ref{stri-lower}) and (\ref{zero-length}), we know that none of them will give zero length bars. Moreover, by (\ref{dim-ima}), multiplicity of degree-$(-2)$ verbose barcode is $3$. So multiplicity of degree-$(-2)$ concise barcode is $1$. 
\end{ex}

\subsection{Proof of Proposition \ref{multi-egg}} 
For any $p$-tuple generator of $CF_*(\Sigma_g, \phi^p_{\lambda})_{\alpha_{\lambda}}$ denoted by  $\{z, \phi(z), \dots , \phi^{p-1}(z)\}$, for the span $V= {\Span}_{\mathcal K} \left<z, \phi(z), \dots ,\phi^{p-1}(z) \right>$, the operator $R_p - \xi_p \cdot {\mathds 1}$ on $V$ is represented by the matrix 
\begin{equation} \label{xi-matrix}
Q_p =  \begin{bmatrix}
    - \xi_p & 0         & \cdots & \cdots & 1 \\
            1 & - \xi_p & \cdots & \cdots &0\\
            0  &   1      & - \xi_p & \cdots & 0\\
             \vdots &  \vdots & \vdots & \ddots & \vdots\\
             0 & 0& \cdots & \cdots & - \xi_p
  \end{bmatrix}
 \end{equation}
  and it has rank $p-1$. Its kernel is 
\[ \ker(R_p - \xi_p \cdot {\mathds 1}) = {\Span}_{\mathcal K} \left< \xi_p^{p-1} z + \xi_p^{p-2} \phi(z) + \cdots  + \phi^{p-1}(z) \right>. \]

\begin{proof} For any degree $k$, after choosing the standard orthogonal basis of the initial generating loops, boundary map $(\partial_{co})_{k+1}$ can be represented as the following matrix
\begin{equation} \label{cone-bd-matrix}
  \left[
    \begin{array}{ccccc}
  * &   *  & \dots   &  *   & * \\ 
  \boxed{P}  &  *   &  \cdots  &    * & \vdots \\ 
            0  &  \boxed{P}  &   \cdots   &   \vdots   &\vdots \\ 
  \vdots &  \vdots &   \ddots &   *  & \vdots  \\ 
        0  &    0    &  \cdots     &   \boxed{P} & *   \\ 
  \end{array}\right]
\end{equation}
where boxed $P = - Q_p$ in (\ref{xi-matrix}). Starting from the most left boxed $P$, we know that after choosing pivot columns as describe in the Example \ref{cpt-bar} above, we will have $(p-1)$-many elements in the singular value decomposition which contribute to $(p-1)$-many zero length bars. Moreover, after the first $p-1$ step of Gaussian elimination, we will change the basis element corresponding to the $p$-th column in (\ref{cone-bd-matrix}) in the form 
\[ v = v_* + \left\{\begin{array}{cc} \mbox{strictly lower}\\ \mbox{filtration term} \end{array} \right\}\]
where $v_* \in \ker(R_p - \xi_p \cdot {\mathds 1})$. Therefore 
$$\ell(v) - \ell((T - \xi_p \cdot {\mathds 1})(v)) >0.$$
Meanwhile even though Gaussian elimination will possibly change entries for each column (except pivot columns), due to the nice position of boxed $P$ and the order that we start from the most left boxed $P$ and consecutively move to the last boxed $P$, we will eventually get 
\[ \frac{p-1}{p} \dim(CF_k(\Sigma, \phi^p_{\lambda})_{\alpha}) = (p-1) \cdot \binom{2p}{k+p-1} \]
many zero length bars. By (\ref{zero-length}), we know that after we use up all the pairs (under boundary map) from original basis having zero difference on the filtration, others will always give positive length bar. So in total, the multiplicity of zero length bars is 
\[ \sum_{k = -p+1}^{p+1} (p-1) \cdot \binom{2p}{k+p-1} = (p-1) \cdot 2^{2p}. \]
So we have exactly $2^{2p}$-many non-zero positive length bars in total. 
\end{proof}

\begin{cor} When $p\geq 3$, $power_p(\Sigma_g) = \infty$. \end{cor}

\begin{proof} Proposition \ref{multi-egg}, Corollary \ref{large-gap} and Theorem \ref{imm-thm}. \end{proof}

\section{Computation of product $\Sigma_g \times M$}\label{prod}

\subsection{Product structure of barcode} 

Recall that given two Floer-type complexes $(C_*, \partial_C, \ell_C)$ and $(D_*, \partial_D, \ell_D)$ over $\Lambda^{\mathcal K, \Gamma}$, we can form its tensor product 
\[ ((C \otimes D)_*, \partial_{\otimes}, \ell_{\otimes} ) \]
where 
\begin{equation} \label{ts-pd}
 \partial_{\otimes} (a \otimes b) = \partial_C a \otimes b + (-1)^{|a|} a \otimes \partial_D b 
 \end{equation}
where $|\cdot|$ denotes degree of the element and 
\[ \ell_{\otimes}(a \otimes b) = \ell_C(a) + \ell_D(b). \]
A singular value decomposition of $(C\otimes D)_*$ can be built from singular value decompositions of $C_*$ and $D_*$. Specifically, by Proposition 7.4 in \cite{UZ15}, each Floer-type complex can be (up to filtered chain isomorphism or filtered homotopy equivalence) decomposed as a direct sum of {\it elementary filtered complex} (see Definition 7.2 in \cite{UZ15}). In other words, the building blocks are in the form 
\begin{align*} 
&\cdots \to 0 \to {\Span}_{\Lambda^{\mathcal K, \Gamma}} \left<x\right> \to 0 \to \cdots \,\,\,\,\,\mbox{and} \\
&\cdots \to 0 \to {\Span}_{\Lambda^{\mathcal K, \Gamma}} \left<y\right> \to {\Span}_{\Lambda^{\mathcal K, \Gamma}} \left<\partial y\right> \to 0 \to \cdots,  
\end{align*}
where $\partial y \neq 0$. Therefore, depending on the property of $y$ and $z$ (whether they are cycle or not), by definition of boundary map (\ref{ts-pd}), we have the following four types elementary complexes for the tensor product structure, 
\begin{itemize}
\item{} $\cdots \to 0 \to {\Span}_{\Lambda^{\mathcal K, \Gamma}} \left<y \otimes z \right> \to 0 \to \cdots$;
\item{} $\cdots \to 0 \to {\Span}_{\Lambda^{\mathcal K, \Gamma}} \left<y \otimes z \right> \to {\Span}_{\Lambda^{\mathcal K, \Gamma}} \left<\partial y \otimes z \right>  \to 0 \to \cdots$;
\item{} $\cdots \to 0 \to {\Span}_{\Lambda^{\mathcal K, \Gamma}} \left<y \otimes z \right> \to {\Span}_{\Lambda^{\mathcal K, \Gamma}} \left<y \otimes \partial z \right>  \to 0 \to \cdots$;
\item{} \begin{align*}
&\cdots \to 0 \to {\Span}_{\Lambda^{\mathcal K, \Gamma}} \left<y \otimes z \right> \to {\Span}_{\Lambda^{\mathcal K, \Gamma}} \left<\partial y \otimes z \pm y \otimes \partial z \right>  \to 0 \to \cdots \\
& \mbox{and}\\
& \cdots \to 0 \to {\Span}_{\Lambda^{\mathcal K, \Gamma}} \left<\min_{{\tiny \mbox{filtration}}}\{ \partial y \otimes z, y \otimes \partial z\} \right>\\
&\quad\ \qquad \to {\Span}_{\Lambda^{\mathcal K, \Gamma}} \left<\partial y \otimes \partial z \right> \to 0 \to \cdots 
\end{align*}
\end{itemize}
If $y$ and $z$ are elements from singular value decompositions of $C_*$ and $D_*$, then the orthogonality of each of these building blocks are guaranteed by Corollary 8.2 in \cite{Ush13}. 

\begin{remark} \label{rk-minimum} Here, we need to emphasize that in the last case (the second case of the fourth type), the minimum is chosen according to the filtrations of $\partial y \otimes z$ and $y \otimes \partial z$ and taking minimum is necessary. In fact, the degree-$({\rm deg}(y) + {\rm deg}(z) - 1)$-piece, as a $\Lambda^{\K, \Gamma}$-vector space, could also be generated by 
\[ \partial y \otimes z \pm y \otimes \partial z \,\,\,\,\mbox{and} \,\,\,\, \partial y \otimes z \mp y \otimes \partial z. \]
However, they are not orthogonal (when filtrations of $\partial y \otimes z$ and $y \otimes \partial z$ are not equal). Therefore, taking the minimum is to make sure the basis elements are orthogonal. \end{remark}

Based on this, we have the following proposition describing the barcodes of tensor product. 

\begin{prop} \label{bar-tensor} Barcode\footnote{Here degree is not specified.} of $((C\otimes D)_*, \partial_{\otimes}, \ell_{\otimes})$ is given by carrying on the following operations for the original barcodes from $C$ and $D$, 
\begin{itemize}
\item{} For $\partial_C y = \partial_D z = 0$,
\begin{align*}
&\quad\ ((\ell_C(y) \,{\mod} \, \Gamma, \infty),\,\, (\ell_D(z) \,{\mod} \, \Gamma, \infty))  \\
&\rightarrow ((\ell_C(y)+\ell_D(z)) \,{\mod} \, \Gamma, \infty).
\end{align*}
\item{} For $\partial_C y =x$ and $\partial_D z = 0$,
\begin{align*}
&\quad\ ((\ell_C(x) \,{\mod} \, \Gamma, \ell_C(y) - \ell_C(x) ),\,\, (\ell_D(z) \,{\mod} \, \Gamma, \infty))\\
&\rightarrow ((\ell_C(x)+\ell_D(z)) \,{\mod} \, \Gamma, \ell_C(y) - \ell_C(x)).
\end{align*}
\item{} For $\partial_C y =0$ and $\partial_D z = w$,
\begin{align*} 
&\quad\ ((\ell_C(y) \,{\mod} \, \Gamma, \infty ),\,\, (\ell_D(w) \,{\mod} \, \Gamma, \ell_D(z) - \ell_D(w)))\\
&\to ((\ell_C(y)+\ell_D(w)) \,{\mod} \, \Gamma, \ell_D(z) - \ell_D(w)).
\end{align*}
\item{} For $\partial_C y =x$ and $\partial_D z = w$,
\begin{align*}
& (\ell_C(x) \,{\mod} \, \Gamma, \ell_C(y) - \ell_C(x)), \,\,(\ell_D(w)\, {\mod} \, \Gamma, \ell_D(z) - \ell_D(w))\\
	\rightarrow\; &\big((\max\{\ell_C(x) + \ell_D(z), \ell_C(y) + \ell_D(w)\}) \,{\mod}\, \Gamma,\\
&\ \min\{\ell_C(y) - \ell_C(x), \ell_D(z) - \ell_D(w) \}\big)\\
&\cup( (\ell_C(x) + \ell_D(w))\, {\mod} \, \Gamma, \min\{\ell_C(y) - \ell_C(x), \ell_D(z) - \ell_D(w) \})
\end{align*}
\end{itemize}
\end{prop}

\begin{proof} All the items are easy to verify. Here, we give the proof of the last item. If $\partial_C y =x$ and $\partial_D z = w$, then under the boundary map, 
\[ \partial_{\otimes} (y \otimes z) = \partial_C y \otimes z + (-1)^{|y|} y \otimes \partial_D z = x \otimes z + (-1)^{|y|} y \otimes w \]
since $x\otimes z$ is orthogonal to $y \otimes w$, for filtration, if $\ell_{\otimes}(x \otimes z) \leq \ell_{\otimes}(y \otimes w)$ (which is $\ell_D(z) - \ell_D(w) \leq \ell_C(y) - \ell_C(x)$), then 
\begin{align*} 
&\quad\ \ell_{\otimes} (y \otimes z) - \ell_{\otimes} (x \otimes z + (-1)^{|y|} y \otimes w)\\
&= \ell_C(y) + \ell_D(z) - (\ell_C(y) + \ell_D(w))\\
&= \ell_D(z) - \ell_D(w) 
\end{align*}
and if $\ell_{\otimes}(x \otimes z) \geq \ell_{\otimes}(y \otimes w)$ (which is $\ell_C(y) - \ell_C(x) \leq \ell_D(z) - \ell_D(w)$), then 
\begin{align*} 
&\quad\ \ell_{\otimes} (y \otimes z) - \ell_{\otimes} (x \otimes z + (-1)^{|y|} y \otimes w)\\
&= \ell_C(y) + \ell_D(z) - (\ell_C(x) + \ell_D(z))\\
&= \ell_C(y) - \ell_C(x).
\end{align*}
Meanwhile, we have another type of boundary relation as described above, that is,  
\[ \partial_{\otimes} (y \otimes w) = \partial_{\otimes} (x \otimes z) = x \otimes w. \]
Because the optimal generator is chosen by taking minimum according to the filtrations as explain in Remark \ref{rk-minimum}, the length of the bar is the minimum among 
\[ \ell_{\otimes} (y \otimes w) - \ell_{\otimes} (x \otimes w) = \ell_C(y) - \ell_C(x)\]
and 
\[ \ell_{\otimes} (x \otimes z) - \ell_{\otimes} (x \otimes w) = \ell_D(z) - \ell_D(w). \]
Therefore, we have a bar
	\[ ( (\ell_C(x) + \ell_D(w))\, {\mod} \, \Gamma, \min\{\ell_C(y) - \ell_C(x), \ell_D(z) - \ell_D(w) \}).\vspace{-2em}\]
\end{proof}

\begin{ex} \label{prod-barcode-ex} Suppose we are given two Floer-type complexes over $\K$, 
\begin{align*} 
&(C_*, \partial_C, \ell_C)  \,\, \mbox{(with trivial homology)} \\
&\mbox{and} \,\,\,\,\,\,\, (CF_*(M, {\mathds 1}) = CF_*(M, {\mathds 1})_{\tiny\{pt\}}, \partial_{{\mathds 1}}, \ell_{{\mathds 1}})
\end{align*}
where $M$ is a symplectically aspherical manifold associated with a Hamiltonian diffeomorphism being identity map. First of all, the precise definition of a Hamiltonian Floer chain complex with identity map is needed. 

Fix a ($C^2$-small) Morse function on $M$, say $f$, and consider $CF_*(M, \ep f)$. On the one hand, by standard Floer theory, we have 
$$CF_*(M, \ep f) \simeq CM_*(M, \ep f)$$
	when $\ep$ is sufficiently small and $CM_*$ represents a Morse chain complex. On the other hand, for each $\ep>0$, we have a well-defined barcode of $CF_*(M, \ep f)$ which consists of two types of bars : 
\begin{itemize}
\item[(a)] infinite length bars which correspond to the generators of (Morse) homology of $M$ with respect to Morse function $\ep f$; 
\item[(b)] finite length bars whose lengths are, by Proposition 3.4 in \cite{Ush11}, at most $\ep ||f||$ (the total variation of $f$ on $M$). 
\end{itemize}
The second type is called {\it ``$\ep$-small interval''}.

Then note that a different choice of ($C^2$-small) Morse function, say $g$, results in another chain complex $CF_*(M, \ep g)$ which is $(\delta_+, \delta_-)$-quasiequivalent to $CF_*(M, \ep f)$, where $\delta_+ = \ep \cdot \max_M (g-f)$ and $\delta_- = -\ep \cdot \min_M (g-f)$. In particular, in this way, homology group $HF_*(M, {\mathds 1})$ is well-defined, i.e., independent of choice of Morse function $f$, that is, 
\[ HF_*(M, {\mathds 1}) = \lim_{\ep \to 0} H_*(CF(M, \ep f)). \]
Moreover, when $M$ is aspherical, it is easy to see $HF_*(M, {\mathds 1}) \simeq H_*(M; \K)$. Now we define 
\begin{dfn} \label{dfn-chain-identity} Floer chain complex with identity map is defined as 
\[ CF_*(M, {\mathds 1}) : = HF_*(M, {\mathds 1}). \]
Specifically, as a graded vector space, $CF_*(M, {\mathds 1})$ is equal/isomorphic to $HF_*(M, {\mathds 1})$ and the boundary operator of this chain complex $\partial_{{\mathds 1}}$ is trivial. 
\end{dfn}
In particular, $CF_*(M, {\mathds 1})$ is independent of the choice of Morse function and when $M$ is aspherical,
\begin{equation} \label{CM-HM} 
b_j(M) = \dim H_j(M; \K) = \dim HF_j(M, {\mathds 1}) = \dim CF_j(M, {\mathds 1}). 
\end{equation} 
for each degree $j \in \Z$.

\begin{remark} \label{valid} Definition \ref{dfn-chain-identity} is not only valid formally (roughly speaking, identity map induces the trivial boundary operator, so the chain complex is isomorphic to its homology) but also valid from the perspective of barcode. In fact, as described above, different choices of Morse functions result in different chain complexes (for each $\ep>0$), but by Stability Theorem, Theorem 1.4 in  \cite{UZ15}, we know that
\[ d_B(\mathcal B_c(CF_*(M, \ep f)), \mathcal B_c(CF_*(M, \ep g))) \leq 2 \ep \cdot ||g-f||\]
where $\mathcal B_c(C_*)$ denotes the concise barcode of a given chain complex $C_*$. Therefore, when $\ep \to 0$ we can conclude, 
\begin{itemize}
\item[(1)] we have a well-defined limit barcode of the concise barcodes $\mathcal B_c(CF_*(M, \ep f))$, i.e., it does not depend on the choice of Morse function, which will be denoted by  $\mathcal B_c(CF_*(M, {\mathds 1}))$;
\item[(2)] $\mathcal B_c(CF_*(M, {\mathds 1}))$ consists of only infinite length bars (as any $\ep$-small interval has its length approximated to $0$ when $\ep$ goes to $0$, again by Proposition 3.4 in \cite{Ush11}).
\end{itemize}
In particular, for any chain complex which has its concise barcode as the limit barcode $\mathcal B_c(CF_*(M, {\mathds 1}))$, (2) implies its boundary depth (the length of the longest finite length bar) is zero. Hence for any possible definition of a chain complex serving as the limit chain complex of $CF_*(M, \ep f)$ so that its concise barcode is $\mathcal B_c(CF_*(M, {\mathds 1}))$, its boundary operator $\partial_{{\mathds 1}}$ should be trivial. This motivates us to define our limit chain complex from the homology and $HF_*(M, {\mathds 1})$ (regarded as a chain complex with trivial boundary operator) provides a nice choice. \end{remark}

Now denote barcodes 
\begin{align*} 
&\mathcal B(C_*) = \{(a, L) \,| \,a \in \K, \,\,\mbox{$L$ is finite} \} \\
&\mbox{and} \quad \mathcal B_c(CF_*(M, {\mathds 1})) = \{(c, \infty)\,| \, c \in \K\}. 
\end{align*}
For its tensor product $(C \otimes_{\K} CF(M, {\mathds 1}))_*$, by Proposition \ref{bar-tensor}, only the second type is considered, that is,  
\[ ((a, L), (c, \infty)) \to (a + c, L) \]
By (\ref{CM-HM}), given $k \in \Z$, for any degree $i, j$ such that $i +j =k$ (where $0 \leq j \leq 2n = \dim(M)$)
\[ \mathcal B(C_i \otimes_{\K} CF_j(M, {\mathds 1})) = \left\{ (a + \ast, L) \,\bigg| \, \begin{array}{cc} (a, L) \in \mathcal B(C_i) \\ L \,\,\mbox{repeats $b_j(M)$-many times} \end{array} \right\}. \]
So the multiplicity of $(i,j)$-piece of degree-$k$ barcode of product is 
\begin{equation} \label{multi-prod-ex}
\mbox{\{multiplicity of $\mathcal B(C_i)$\}} \times b_j(M).
\end{equation} 
In general, for any symplectic manifold $(M, \omega)$, not necessarily symplectically aspherical, we count multiplicity of concise barcode in a more subtle way since for $CF_*(M, {\mathds 1})$, Novikov field will be involved, that is,
\[ CF_*(M, {\mathds 1}) = HF_*(M, {\mathds 1}) =  H_*(M, \K) \otimes \Lambda^{\K, \Gamma} \]
where $\Gamma \leq \R$. Explicitly, adding a homotopy class $S \in \pi_2(M)/(\ker(\omega) \cap \ker(c_1))$ to each critical point $p$ will result in
\begin{itemize}
\item{(a)} filtration is shifted by $\int_{S^2} S^* \omega$;
\item{(b)} CZ-index is shifted by $-2Nc_1(S)$,
\end{itemize}
where $N$ is minimal Chern number. However, in terms of barcode, shift of filtration (a) will not change the length of bars since the end point is defined modulo $\Gamma$. The issue will come from shift of degree since, different from the symplectically aspherical case, the range of indices of $HF_*(M, {\mathds 1})$ might be infinite by the shift of index (b) (when $N$ is nonzero). Quantum Betti number (see Definition \ref{qb}) is then helpful for our counting in the sense that, similarly to (\ref{multi-prod-ex}), the number of positive length bars that a $(i,j)$-piece will contribute is 
\begin{equation} \label{multi-prod-ex2}
\mbox{\{multiplicity of $B(C_i)$\}} \times qb_j(M).
\end{equation} 
\end{ex}
Meanwhile, the following lemma is helpful to understand the self-mapping cone of product. 

\begin{lemma} \label{tensor-cone} Considering the following two maps $C \xrightarrow{S} C$ and  $D \xrightarrow{{\mathds 1}} D$, then we have a {\it chain} isomorphism for any degree $m$,
\[ (Cone_{C \otimes D}(S \otimes {\mathds 1}))_m \simeq \bigoplus_k (Cone_C(S))_{k} \otimes D_{m-k} \]
where the left side is mapping cone of $C\otimes D \xrightarrow{S \otimes {\mathds 1}} C \otimes D$. 
\end{lemma}

\begin{proof} We have the following diagram 
\begin{equation*}
\resizebox{\textwidth}{!}{
  \xymatrix{ \oplus_k C_{k} \otimes D_{m-k} \ar[r]^{S \otimes {\mathds 1}} \ar[d]^{\phi} & \oplus_k C_{k} \otimes D_{m-k} \ar[r]^-{\iota}\ar[d]^{\phi} & \oplus_k (Cone_C(S))_{k} \otimes D_{m-k}
    \ar[r]^-{\pi}\ar[d]^{h} & \oplus_k C_{k-1} \otimes D_{m-k} \ar[d]^{\phi} \\ (C \otimes D)_m \ar[r]^{S \otimes {\mathds 1}} & (C \otimes D)_m 
	\ar[r]_{\iota} & Cone_m(S \otimes {\mathds 1}) \ar[r]_{\pi}& (C \otimes D)_{m-1}}.}
\end{equation*}
The top row is a distinguished triangle because tensor (of vector spaces) preserves distinguished triangle here\footnote{Note that in general, tensoring is only right exact. But here our category only consists of vector spaces, as a special case of flat module, tensoring here is then exact.} and direct sum of distinguished triangles is also distinguished triangle. Moreover, $\phi$ is an isomorphism by definition, so there exists a chain map $h$ being an isomorphism. 
\end{proof}

\subsection{Proof of Proposition \ref{action-prod} and \ref{multi-prod}}

\begin{proof}

(a) By Proposition \ref{bar-tensor} and the definition of $CF_*(M, {\mathds 1})$ which consists of only infinite length bars, the second type of tensor is considered and in particular the length of finite length bar keeps the same. Therefore, the smallest length of finite length bar of $Cone_{\otimes}(H_{\lambda})_*$ is the same as the smallest length of finite length bar of $Cone_{\Sigma_g}(H_{\lambda})_*$ coming from egg-beater model. Therefore, we draw the conclusion of Proposition \ref{action-prod} by Proposition 5.1 in \cite{PS14}. 

(b) Fix the degree of product being $1$. By (\ref{multi-prod-ex2}) and the proof of Proposition \ref{multi-egg}, multiplicity of degree-$1$ concise barcode of $Cone_{\otimes}(H_{\lambda})_*$ is 
\begin{equation} \label{prod-mul}
\sum_{k = -p+1}^{p+1} \binom{2p}{k+p-1} \cdot qb_{1-k}(M).
\end{equation}
We can split (\ref{prod-mul}) into two parts. One is 
\begin{equation} \label{A} 
\binom{2p}{0} \cdot qb_{p}(M) + \binom{2p}{p} \cdot qb_{0}(M) + \binom{2p}{2p} \cdot qb_{-p}(M) 
\end{equation} 
and the other is
\begin{equation} \label{B} \sum_{k \neq \{-p+1, 1, p+1\}} \binom{2p}{k+p-1} \cdot qb_{1-k}(M).
\end{equation}
Note that each binomial number in (\ref{B}) is divisible by $p$, therefore divisibility of (\ref{prod-mul}) depends only on (\ref{A}). First, it is never equal to zero since $qb_0(M) \neq 0$. Moreover, by Babbage's theorem in \cite{Bab}, 
we know 
\[ \binom{2p}{p} = 2 \binom{2p-1}{p-1} \equiv 2 \, {\mod} \,p. \]
Therefore, modulo $p$, we get the first conclusion of Proposition \ref{multi-prod}. For its second conclusion, by definition, 
\[ qb_{p}(M) = \sum_{s \in \Z} b_{p + 2 Ns} (M) \leq \sum_{i \,\tiny{\mbox{is odd}}} b_i(M) \]
since $p+ 2Ns$ is always odd. Similarly to $qb_{-p}(M)$; and
\[ qb_{0}(M) = \sum_{s \in \Z} b_{2Ns} (M) \leq \sum_{i \,\tiny{\mbox{is even}}} b_i(M) \]
since $2Ns$ is always even. Therefore, together, modulo $p$, we get (\ref{A}) is at most $2 \sum_{0 \leq i \leq 2n} b_i(M)$. So when $p >2\sum_{0 \leq i \leq 2n} b_i(M)$, non-divisibility always holds. \end{proof} 

\begin{remark} \label{div-imp} Here, we remark that the non-divisibility requirement for Proposition \ref{multi-prod} can sometimes be improved considerably. Here are two examples. 
\begin{itemize}
\item[(a)] If $c_1(TM) = 0$, then since $qb_k(M) = b_k(M)$ for any $k \in \Z$,  modulo $p$, (\ref{A}) is equal to 
\[ b_p(M) + 2. \]
Therefore, if $p \nmid b_p(M) + 2$, then $p \nmid m_1$.
\item[(b)] If $M = \C P^n$, so $c_1(T \C P^n) = n+1$ and $b_k(\C P^n)$ is nonzero only when $k$ is even and $k \in [0,2n]$. Therefore, back to (\ref{A}), we only need to consider the middle term and modulo $p$, we get $2$. Therefore, the non-divisibility holds for any  prime $p\geq 3$ in this case. 
\end{itemize}
\end{remark}

\medskip

\end{document}